\documentclass[12pt,oneside, reqno]{amsart}
\usepackage{txfonts}
\usepackage{amsfonts}
\usepackage{mathrsfs}
\usepackage{amsmath}
\usepackage{amssymb}
\usepackage{bbm}
\usepackage{stmaryrd}
\usepackage{esint}
\usepackage{xcolor}

\usepackage{enumerate}

\newtheorem{theorem}{Theorem}[section]
\newtheorem{lemma}[theorem]{Lemma}
\newtheorem{remark}[theorem]{Remark}
\newtheorem{definition}[theorem]{Definition}
\newtheorem{proposition}[theorem]{Proposition}

\newtheorem{corollary}[theorem]{Corollary}

\def\<{{\langle}}
\def\>{{\rangle}}

\def\bt{\begin{theorem}}
	\def\et{\end{theorem}}
\def\bl{\begin{lemma}}
	\def\el{\end{lemma}}
\def\br{\begin{remark}}
	\def\er{\end{remark}}
\def\bd{\begin{definition}}
	\def\ed{\end{definition}}
\def\bp{\begin{proposition}}
	\def\ep{\end{proposition}}
\def\bc{\begin{corollary}}
	\def\ec{\end{corollary}}

\allowdisplaybreaks

\allowdisplaybreaks
\numberwithin{equation}{section}
\begin{document}
	\title{Generic Strong Well-Posedness for McKean--Vlasov SDEs with Jumps in a Krylov-Stable Coefficient Space}
	\date{}
	\author{ MINGBO ZHANG${}^{1,\ast}$}
	
	%\thanks{}
	\dedicatory{
		${}^{1,\ast}$ School of Statistics and data science, Jiangxi University of Finance and Economics\\
		Nanchang, Jiangxi 333000, P. R. China\\
		M. Zhang: zmb1982zqz@hotmail.com }
	
	\footnote[0]{$^{\S}$This work was supported by NSFC (12361030).
		
		${}^\ast$Corresponding author}
	
	
	\date{}

	\begin{abstract}
		We prove that strong well-posedness is generic for McKean--Vlasov
		stochastic differential equations with jumps in a complete coefficient
		space that allows genuinely discontinuous state dependence. For every
		$T>0$, there exists a dense residual subset $\mathfrak R_T$, independent
		of the initial state, such that every coefficient $a\in\mathfrak R_T$
		generates a unique strong solution for every deterministic initial
		condition $x\in\mathbb R^d$. Moreover, the solution map
		$x\mapsto X^{a,x}$ is continuous in $\mathcal S_T^2$, and every
		approximation by spatially Lipschitz coefficients converges locally
		uniformly in $x$ to the same canonical solution family.
		
		The coefficients are uniformly bounded, uniformly non-degenerate in the
		diffusion component, and Lipschitz continuous in the law with respect to
		the $2$-Wasserstein distance, while their state dependence may be merely
		measurable. The ambient space is obtained by completing a spatially
		Lipschitz core under a Krylov-stable metric based on local
		$L^{2(d+1)}$ coefficient errors. A stopped conditional
		Krylov--Melnikov occupation estimate converts convergence in this metric
		into occupation convergence along non-degenerate It\^o--L\'evy
		trajectories.
		
		The key probabilistic ingredient is a set-valued stability theorem: near
		a core coefficient, every strong solution of every nearby solvable
		equation remains close to the unique core solution, without any
		uniqueness assumption on the nearby equation. Together with closedness
		of the strong-solution relation, this permits a Baire-category argument
		on the completed coefficient space. We also give an intrinsic
		spatial-translation criterion showing that the completion contains
		natural classes with essential spatial discontinuities. A projective-limit
		argument yields the corresponding global-time generic well-posedness
		result.
	\end{abstract}
	
	\maketitle

	\section{Introduction}
	\label{sec:introduction}
	
	We consider McKean--Vlasov stochastic differential equations driven by
	Brownian motion and a compensated Poisson random measure,
	\begin{equation}
		\label{eq:introduction-equation}
		\begin{aligned}
			X_t
			={}&x
			+\int_0^t b\bigl(s,X_s,\mathcal L(X_s)\bigr)\,ds
			+\int_0^t \sigma\bigl(s,X_s,\mathcal L(X_s)\bigr)\,dW_s
			\\
			&+\int_0^t\int_U
			\beta\bigl(s,X_{s-},\mathcal L(X_s),z\bigr)\,
			\widetilde N(ds,dz),
			\qquad 0\le t\le T.
		\end{aligned}
	\end{equation}
	Here $W$ is a Brownian motion, $\mathcal L(\xi)$ denotes the law of a
	random variable $\xi$, and $\widetilde N(dt,dz)=N(dt,dz)-dt\,\nu(dz)$
	is a compensated Poisson random measure. McKean--Vlasov equations have
	been studied since the seminal work of McKean \cite{Mckean-1966} and are
	a standard probabilistic model for mean-field interactions and
	propagation of chaos; see, for instance, \cite{Sznitman-1991}.
	
	The coefficients in \eqref{eq:introduction-equation} are assumed to be
	uniformly bounded, the diffusion component is uniformly non-degenerate,
	and the dependence on the probability law is Lipschitz continuous with
	respect to the $2$-Wasserstein distance. In the state variable, however,
	the coefficient space constructed below allows merely measurable and
	genuinely discontinuous dependence. At this level of spatial regularity,
	strong existence and pathwise uniqueness cannot in general be expected
	for every prescribed coefficient. This leads to a question different
	from the usual search for pointwise sufficient conditions:
	
Can one construct a complete coefficient space broad enough to
		contain genuine state discontinuities, yet sufficiently adapted to the
		stochastic dynamics that strong well-posedness holds for a topologically
		large set of coefficients?
	
	Classical well-posedness results for McKean--Vlasov equations impose
	regularity, monotonicity, coercivity, or related structural assumptions on
	each prescribed coefficient. For Brownian equations, strong well-posedness
	under irregular, singular, H\"older, or spatially discontinuous state
	dependence has been established by several different regularization
	mechanisms; see, among others,
	\cite{Bauer-MeyerBrandis-Proske-2018,Chaudru-de-Raynal-2020,
		Huang-Wang-2019,Rockner-Zhang-2021,Leobacher-Reisinger-Stockinger-2022}
	and the more recent rough-coefficient results
	\cite{Pascucci-Rondelli-2025,Pascucci-Rondelli-Veretennikov-2026}.
	For equations with jumps, direct well-posedness results range from
	Lipschitz and locally Lipschitz settings
	\cite{Graham-1992,Jourdain-Meleard-Woyczynski-2008,Hao-Li-2016,
		Erny-2022,Cavallazzi-2025,Chao-Duan-Wei-2023} to the recent
	irregular-coefficient regime considered in \cite{Wang-Ren-Miao-2024}.
	These works identify classes of prescribed coefficients for which the
	corresponding equation is well posed. The question addressed here is
	instead topological: whether strong well-posedness is generic in a complete
	ambient coefficient space that itself contains genuinely discontinuous
	coefficients and coefficients for which pointwise uniqueness criteria are
	unavailable. The distinction is relevant at low regularity, where strong
	existence or pathwise uniqueness may fail even for classical stochastic
	differential equations \cite{BARLOW-1982}, while nonlinear dependence on
	the law may introduce additional uniqueness phenomena
	\cite{Scheutzow-1987}.
	
	Baire-category approaches to generic well-posedness originate in the
	theory of differential equations \cite{Orlicz-1932,Lasota-1973-Yorke}
	and were developed for ordinary stochastic differential equations in
	\cite{Heunis-1986,Bahlali-Mezerdi-1996-Ouknine,
		Bahlali-Mezerdi-1998-Ouknine}. For Brownian McKean--Vlasov equations,
	approximation, stability, prevalence, and genericity for continuous or
	uniformly continuous coefficients were studied in
	\cite{Mezerdi-2019,Bahlali-Mezerdi-Mezerdi-2020,
		Mezerdi-Khelfallah-2021}. More recently, \cite{Mezerdi2025} obtained
	generic pathwise well-posedness and generic convergence of Picard and Euler
	approximations in a space of uniformly continuous coefficients. These
	results provide the closest generic counterparts to the present work.
	
	The contribution of this paper is not a new pointwise regularity
	criterion for a prescribed McKean--Vlasov equation. It is a generic
	well-posedness theory in a rough coefficient space whose topology is
	determined by stochastic occupation. The Krylov-stable completion used
	here goes beyond the uniformly continuous coefficient classes in the
	preceding McKean--Vlasov genericity results and contains coefficients with
	essential spatial discontinuities. At the same time, the jump structure
	forces both the coefficient-space construction and the stability argument
	to be compatible with c\`adl\`ag trajectories and an
	$L^2(U,\nu;\mathbb R^d)$-valued jump coefficient. Outside the Lipschitz
	core the strong-solution relation may be empty or set-valued, so stability
	cannot be formulated through a globally defined single-valued solution
	map. Instead, it must hold uniformly over all strong solutions of every
	nearby solvable equation. This set-valued stability, together with
	closedness of the solution relation, is the mechanism that permits the
	Baire-category argument in the discontinuous coefficient space.
	
	The topology is dictated by a stopped form of the classical
	Krylov--Melnikov occupation estimate. Schematically, for a localized
	non-degenerate It\^o--L\'evy process $Y$ and a stopping time $\tau$
	bounded by its exit time from $B_R$, the estimate applied to the square
	of a coefficient error has the form
	\[
	\mathbb E\int_0^\tau
	\Delta(a,\widetilde a)(t,Y_t)^2\,dt
	\leq
	C\,\|\Delta(a,\widetilde a)\|_{L^{2(d+1)}([0,T]\times B_R)}^2.
	\]
	The coefficient distance $\Delta$ is defined precisely in
	Section~\ref{sec:framework-main}. Thus local $L^{2(d+1)}$ convergence in
	space--time is converted into integrated convergence along stopped jump
	trajectories. We accordingly define $(\mathfrak A_T,d_T)$ as the
	completion of a spatially Lipschitz core
	$\mathfrak A_T^{\mathrm{Lip}}$ under the resulting Krylov-stable metric.
	The occupation estimate is classical in origin; the point here is its
	stopped conditional formulation and its use in constructing a coefficient
	topology for which stochastic stability survives completion. The precise
	historical positioning and the version used in the proof are given in
	Section~\ref{sec:krylov-ito-levy}.
	
	The completion is not simply the class of all measurable coefficient
	triples satisfying the structural assumptions: membership requires
	approximation by the Lipschitz core in $d_T$. We give an intrinsic,
	directly verifiable spatial-translation criterion for this approximation.
	It implies, in particular, that $\mathfrak A_T$ contains coefficients
	pasted across arbitrary fixed finite Borel partitions, finite-rank law
	interactions with bounded measurable spatial profiles, and
	integral-interaction coefficients with uniform spatial-translation
	control. Hence the ambient space contains coefficients with essential
	spatial discontinuities and is strictly larger than the uniformly
	continuous coefficient spaces used in the preceding genericity results.
	
	Our main conclusion is that the well-posed coefficients form a
	topologically large subset of this rough space. For every fixed time
	horizon $T>0$, there exists a dense residual set
	\[
	\mathfrak R_T\subset\mathfrak A_T
	\]
	such that every $a\in\mathfrak R_T$ generates a unique strong solution of
	\eqref{eq:introduction-equation} for every deterministic initial state
	$x\in\mathbb R^d$. A point worth emphasizing is that the residual set is
	chosen once and for all: it is independent of the initial state. Moreover,
	the solution map
	\[
	x\longmapsto X^{a,x}
	\]
	is continuous with values in $\mathcal S_T^2$, and for every
	Lipschitz-core approximation $a_n\to a$,
	\[
	\sup_{|x|\leq m}
	\|X^{a_n,x}-X^{a,x}\|_{\mathcal S_T^2}
	\longrightarrow0,
	\qquad m\geq1.
	\]
	The limit is independent of the approximating sequence. In this sense,
	the classical solution map on the Lipschitz core admits a canonical
	generic extension to the completion. A projective-limit construction
	yields the corresponding global-time result, again on a residual set
	independent of the initial state.
	
	Three ingredients underlie this conclusion.
	\begin{enumerate}
		\item
		A Krylov-stable coefficient space containing genuine
			discontinuities.
		We construct the complete metric space $(\mathfrak A_T,d_T)$ and
		prove that completed elements admit measurable representatives that
		preserve uniform boundedness, $W_2$-Lipschitz continuity in the law,
		and ellipticity. The spatial-translation criterion described above
		provides broad natural subclasses of the completion.
		
		\item
		Set-valued stochastic stability and closedness.
		Near every Lipschitz-core coefficient, every strong solution of every
		nearby solvable equation remains close to the unique core solution;
		no uniqueness assumption is imposed on the nearby equation. We then
		prove closedness of the strong-solution relation along core
		approximations. The compensated Poisson integrands are treated in
		their natural space $L^2(\Omega\times[0,T];H)$, where
		$H=L^2(U,\nu;\mathbb R^d)$.
		
		\item
		A generic canonical extension simultaneous in the initial
			state.
		An abstract Baire-category extension principle for a single-valued
		map on a dense regular core and a possibly set-valued relation on its
		completion yields a dense residual set independent of $x$. Every
		coefficient in this set is strongly well posed for all
		$x\in\mathbb R^d$, and every core approximation converges locally
		uniformly in the initial state to the same solution family.
	\end{enumerate}
	
	The simultaneous nature of the result is stronger than obtaining, for
	each fixed $x$, a residual set that may depend on $x$. It follows from
	stability neighborhoods that are uniform over compact sets of initial
	conditions and a two-parameter Baire-category construction; no additional
	intersection over a countable dense subset of $\mathbb R^d$ is required.
	
	The remainder of the paper is organized as follows.
	Section~\ref{sec:framework-main} introduces the stochastic framework, the
	Krylov-stable coefficient space, and the main results.
	Section~\ref{sec:krylov-ito-levy} proves the stopped conditional
	occupation estimate used in the stability theory.
	Section~\ref{sec:coefficient-space-core} develops the completion and its
	discontinuous subclasses. Sections~\ref{sec:stability-core} and
	\ref{sec:closedness-limits} establish stability and closedness,
	respectively. Section~\ref{sec:generic-strong-wellposedness} proves the
	generic canonical extension and its global-time version. The analytic
	regularization used in the Krylov argument is collected in
	Appendix~\ref{app:aleksandrov-regularization}.
	
	\section{Framework and main results}
	\label{sec:framework-main}
	
	Let $(\Omega,\mathcal F,(\mathcal F_t)_{t\ge0},\mathbb P)$ be a complete filtered probability space satisfying the usual conditions. Let $W=(W_t)_{t\ge0}$ be a $d$-dimensional $(\mathcal F_t)$-Brownian motion, and let $N(dt,dz)$ be a Poisson random measure on $[0,\infty)\times U$ with compensator $dt\,\nu(dz)$, where $\nu$ is a $\sigma$-finite measure on $(U,\mathcal U)$. We assume that $W$ and $N$ are independent.

	Let $\mathcal P(\mathbb R^d)$ denote the set of Borel probability
	measures on $\mathbb R^d$, and set
	\[
	\mathcal P_2(\mathbb R^d)
	:=
	\left\{\mu\in\mathcal P(\mathbb R^d):
	\int_{\mathbb R^d}|x|^2\,\mu(dx)<\infty\right\}.
	\]
	For $\mu,\eta\in\mathcal P_2(\mathbb R^d)$, let $\Pi(\mu,\eta)$
	denote the set of couplings of $\mu$ and $\eta$. For $p\in\{1,2\}$,
	define the $p$-Wasserstein distance by
	\[
	W_p(\mu,\eta)
	:=
	\left(
	\inf_{\pi\in\Pi(\mu,\eta)}
	\int_{\mathbb R^d\times\mathbb R^d}|x-y|^p\,\pi(dx,dy)
	\right)^{1/p}.
	\]
	In particular, whenever $\xi$ and $\eta$ are square-integrable random
	variables defined on the same probability space,
	\[
	W_2^2(\mathcal L(\xi),\mathcal L(\eta))
	\leq
	\mathbb E|\xi-\eta|^2.
	\]
	
	\subsection{The Krylov-Stable Coefficient Space}
	
	We use $a$ to denote a coefficient triple and write
	$\mathcal E_T(a,x)$ for equation \eqref{eq:introduction-equation}
	with coefficient $a$ and initial value $x$. The precise convention for
	the jump component of $a$ is fixed below after introducing its
	Hilbert-space-valued realization.
	
	We identify adapted c\'adl\'ag processes up to
	indistinguishability and define
	\[
	\mathcal S_T^2
	:=
	\left\{
	X:
	\begin{array}{l}
		X \text{ is an }(\mathcal F_t)\text{-adapted c\'adl\'ag process},\\[1mm]
		\displaystyle
		\mathbb E\big[\sup_{0\leq t\leq T}|X_t|^2\big]<\infty
	\end{array}
	\right\}.
	\]
	On this space we set
	\[
	\|X\|_{\mathcal S_T^2}
	:=
	\bigg(
	\mathbb E\big[\sup_{0\leq t\leq T}|X_t|^2\big]
	\bigg)^{1/2}.
	\]
	With this convention, $\mathcal S_T^2$ is a Banach space.

	Set
	\begin{equation}
		\label{eq:H}
		H:=L^2(U,\mathcal U,\nu;\mathbb R^d),
	\end{equation}
	and assume that $H$ is separable. Its norm is denoted by
	\[
	\|h\|_H^2 := \int_U |h(z)|^2 \nu(dz).
	\]
	All measurability statements concerning $\mathcal P_2(\mathbb R^d)$ refer to its Borel $\sigma$-field induced by the $2$-Wasserstein metric.
	
	A pointwise jump coefficient is a jointly measurable map
	\[
	\beta: [0,T]\times\mathbb R^d \times \mathcal P_2(\mathbb R^d)\times U \longrightarrow \mathbb R^d
	\]
	such that
	\[
	\int_U |\beta(t,x,\mu,z)|^2 \nu(dz) < \infty
	\]
	for every $(t,x,\mu)$. Its associated $H$-valued coefficient is defined by
	\begin{equation}
		\label{eq:beta}
		\overline{\beta}(t,x,\mu) := [\beta(t,x,\mu,\cdot)]_H.
	\end{equation}
	Thus
	\[
	\|\overline{\beta}(t,x,\mu)\|_H^2 = \int_U |\beta(t,x,\mu,z)|^2 \nu(dz).
	\]
	Henceforth, by a coefficient triple we mean
	\[
	a=(b,\sigma,\overline\beta),
	\]
	where $\overline\beta$ is the $H$-valued jump coefficient. The symbol
	$\beta$ is reserved for a jointly measurable pointwise representative
	of $\overline\beta$, as in \eqref{eq:beta} and \textnormal{(A1)}, and is
	used when writing the Poisson integral.
	
	Let $\mathcal P_{\mathrm{pred}}$ denote the predictable $\sigma$-field on $[0,T]\times\Omega$. Suppose that
	\[
	\Gamma: [0,T]\times\Omega \longrightarrow H
	\]
	is $\mathcal P_{\mathrm{pred}}/\mathcal B(H)$-measurable and satisfies
	\[
	\mathbb E \int_0^T \|\Gamma_t\|_H^2\,dt < \infty.
	\]
	Then, applying the representation result with
	\[
	(S,\mathcal S) = ([0,T]\times\Omega,\mathcal P_{\mathrm{pred}}),
	\]
	yields a $\mathcal P_{\mathrm{pred}}\otimes\mathcal U$-measurable representative
	\[
	\gamma: [0,T]\times\Omega\times U \longrightarrow \mathbb R^d
	\]
	such that
	\[
	[\gamma(t,\omega,\cdot)]_H = \Gamma_t(\omega).
	\]
	We then define
	\[
	\int_0^t\int_U \Gamma_s(z) \,\widetilde N(ds,dz) := \int_0^t\int_U \gamma(s,z) \,\widetilde N(ds,dz).
	\]
	This definition is independent of the chosen jointly predictable representative. Indeed, two such representatives agree $(dt,d\mathbb P,\nu(dz))$-almost everywhere. Moreover,
	\begin{equation}
		\label{eq:estimate}
		\mathbb E\left[ \sup_{0\leq t\leq T} \left| \int_0^t\int_U \Gamma_s(z) \widetilde N(ds,dz) \right|^2 \right] \leq C\, \mathbb E \int_0^T \|\Gamma_s\|_H^2\,ds.
	\end{equation}
	See, for example, \cite[Chapters~2 and~4]{Applebaum2009}.

	Fix constants
	\[
	\lambda>0,\qquad L_\mu>0,\qquad
	M>\sqrt{2\lambda d}.
	\]
	Uniform ellipticity and the Hilbert--Schmidt bound necessarily imply
	$M\geq\sqrt{2\lambda d}$. We impose the strict inequality above in
	order to retain a positive coefficient budget for non-trivial drift and
	jump terms and for the discontinuous examples constructed below.

	We impose the following assumptions on a coefficient triple  $a=(b,\sigma,\overline\beta)$.
	
	\begin{enumerate}
		\item[\textnormal{(A1)}]
		The functions
		\[
		b: [0,T]\times\mathbb R^d \times\mathcal P_2(\mathbb R^d) \longrightarrow \mathbb R^d
		\]
		and
		\[
		\sigma: [0,T]\times\mathbb R^d \times\mathcal P_2(\mathbb R^d) \longrightarrow \mathbb R^{d\times d}
		\]
		are jointly Borel measurable, and
		\[
		\overline\beta: [0,T]\times\mathbb R^d \times\mathcal P_2(\mathbb R^d) \longrightarrow H
		\]
		is Borel measurable.
		
		We choose a
		\[
		\mathcal B([0,T])\otimes\mathcal B(\mathbb R^d)\otimes\mathcal B(\mathcal P_2(\mathbb R^d))\otimes\mathcal U
		\]
		measurable representative
		\[
		\beta: [0,T]\times\mathbb R^d \times\mathcal P_2(\mathbb R^d)\times U \longrightarrow \mathbb R^d
		\]
		such that
		\begin{equation}
			\label{eq:beta-rep}
			[\beta(t,x,\mu,\cdot)]_H = \overline\beta(t,x,\mu)
		\end{equation}
		for every
		\[
		(t,x,\mu)\in [0,T]\times\mathbb R^d \times\mathcal P_2(\mathbb R^d).
		\]
		The existence of such a representative is guaranteed by Lemma~\ref{lem:measurable-H-representative}.

		\item[\textnormal{(A2)}]
		The coefficients are uniformly bounded:
		\begin{equation}
			\label{eq:uniform-bound-framework}
			|b(t,x,\mu)|+\|\sigma(t,x,\mu)\|_{\mathrm{HS}}+\|\overline\beta(t,x,\mu)\|_H
			\leq M
		\end{equation}
		for every $(t,x,\mu)\in
		[0,T]\times\mathbb R^d\times\mathcal P_2(\mathbb R^d)$.
		
		\item[\textnormal{(A3)}]
		The coefficients are uniformly Lipschitz continuous in the probability-measure
		variable:
		\begin{equation}
			\label{eq:measure-lipschitz-framework}
			\begin{aligned}
				&|b(t,x,\mu)-b(t,x,\eta)|+\|\sigma(t,x,\mu)-\sigma(t,x,\eta)\|_{\mathrm{HS}}
				\\
				&\qquad+\|\overline\beta(t,x,\mu)-\overline\beta(t,x,\eta)\|_H
				\leq L_\mu W_2(\mu,\eta)
			\end{aligned}
		\end{equation}
		for every $(t,x)\in[0,T]\times\mathbb R^d$ and all $\mu,\eta\in\mathcal P_2(\mathbb R^d)$.
		
		\item[\textnormal{(A4)}]
		The continuous covariance matrix is uniformly non-degenerate:
		\begin{equation}
			\label{eq:uniform-ellipticity-framework}
			\frac12\sigma(t,x,\mu)\sigma(t,x,\mu)^\top\geq\lambda I_d,
		\end{equation}
		for every $(t,x,\mu)\in[0,T]\times\mathbb R^d\times\mathcal P_2(\mathbb R^d)$.
	\end{enumerate}

	Throughout the paper, coefficient classes are represented by admissible representatives for which \textnormal{(A2)}--\textnormal{(A4)} hold pointwise. In particular, the constant triple
	\[
	(0,\sqrt{2\lambda}\,I_d,0)
	\]
	satisfies \textnormal{(A1)}--\textnormal{(A4)}, so the Lipschitz core introduced below is nonempty.

	We next introduce a class of coefficients for which strong well-posedness follows from the standard fixed-point argument.
	
	\begin{definition}
		\label{def:lipschitz-core}
		The Lipschitz core $\mathfrak A_T^{\mathrm{Lip}}$ is the collection	of all coefficient triples $a=(b,\sigma,\overline\beta)$ satisfying	\textnormal{(A1)}--\textnormal{(A4)} and such that there exists a	constant $L_x(a)<\infty$ satisfying
		\begin{equation}
			\label{eq:space-lipschitz-core}
			\begin{aligned}
				&|b(t,x,\mu)-b(t,y,\mu)|
				+\|\sigma(t,x,\mu)-\sigma(t,y,\mu)\|_{\mathrm{HS}}	\\
				&\qquad+\|\overline\beta(t,x,\mu)-\overline\beta(t,y,\mu)\|_H
				\leq L_x(a)|x-y|
			\end{aligned}
		\end{equation}
		for all $t\in[0,T]$, all $x,y\in\mathbb R^d$, and all $\mu\in\mathcal P_2(\mathbb R^d)$.
	\end{definition}
	
	The spatial Lipschitz constant $L_x(a)$ is allowed to depend on the coefficient triple $a$. No common spatial Lipschitz constant is imposed on the whole core.

	\begin{lemma}
		\label{lem-Measurability}
		Let \(X\in\mathcal S_T^2\), and define
		\[
		\mu_t^X:=\mathcal L(X_t), \qquad 0\le t\le T.
		\]
		Then the following statements hold.
		
		\begin{enumerate}
			\item[\textnormal{(i)}]
			The map
			\[
			t\longmapsto \mu_t^X
			\]
			is c\`adl\`ag as a map from \([0,T]\) into \((\mathcal P_2(\mathbb R^d),W_2)\), and its left limit at \(t>0\) is
			\[
			\mu_{t-}^X:=\mathcal L(X_{t-}).
			\]
			In particular, \(t\mapsto \mu_t^X\) is Borel measurable.
			
			\item[\textnormal{(ii)}]
			The drift and diffusion coefficient processes
			\[
			(t,\omega)\longmapsto b(t,X_t(\omega),\mu_t^X)
			\]
			and
			\[
			(t,\omega)\longmapsto \sigma(t,X_t(\omega),\mu_t^X)
			\]
			are progressively measurable.
			
			\item[\textnormal{(iii)}]
			With the convention \(X_{0-}=X_0\), the \(H\)-valued process
			\[
			\Gamma_t^X := \overline\beta(t,X_{t-},\mu_t^X)
			\]
			is predictable. Moreover, the chosen representative
			\[
			\gamma^X(t,\omega,z) := \beta(t,X_{t-}(\omega),\mu_t^X,z)
			\]
			is \(\mathcal P_{\mathrm{pred}}\otimes\mathcal U\)-measurable.
			
			If \textnormal{(A2)} holds, then
			\[
			\mathbb E\int_0^T\int_U |\gamma^X(t,z)|^2\,\nu(dz)\,dt
			=\mathbb E\int_0^T \|\Gamma_t^X\|_H^2\,dt
			\le M^2T.
			\]
		\end{enumerate}
	\end{lemma}
	
	\begin{proof}
		Let \(t_n\downarrow t\). Since \(X\) is c\`adl\`ag,
		\[
		X_{t_n}\longrightarrow X_t
		\qquad\text{almost surely}.
		\]
		Moreover,
		\[
		|X_{t_n}-X_t|^2
		\leq
		4\sup_{0\leq s\leq T}|X_s|^2,
		\]
		and the right-hand side is integrable. Hence, by dominated
		convergence,
		\[
		\mathbb E|X_{t_n}-X_t|^2\longrightarrow0.
		\]
		Using the coupling \((X_{t_n},X_t)\), we obtain
		\[
		W_2(\mu_{t_n}^X,\mu_t^X)^2
		\leq
		\mathbb E|X_{t_n}-X_t|^2
		\longrightarrow0.
		\]
		
		Similarly, if \(t_n\uparrow t\), \(t>0\), then
		\(X_{t_n}\to X_{t-}\) almost surely, and the same dominated
		convergence argument gives
		\[
		X_{t_n}\longrightarrow X_{t-}
		\qquad\text{in }L^2(\Omega;\mathbb R^d).
		\]
		Consequently,
		\[
		W_2\bigl(\mu_{t_n}^X,\mathcal L(X_{t-})\bigr)
		\longrightarrow0.
		\]
		Thus \(t\mapsto\mu_t^X\) is c\`adl\`ag in
		\((\mathcal P_2(\mathbb R^d),W_2)\), with left limit
		\[
		\mu_{t-}^X=\mathcal L(X_{t-}),
		\]
		and is therefore Borel measurable. This proves
		\textnormal{(i)}.
		
		Since \(X\) is adapted and c\`adl\`ag, it is progressively
		measurable, while \(X_-\) is predictable. Since
		\(t\mapsto\mu_t^X\) is deterministic and Borel measurable, the
		process \((t,\omega)\mapsto\mu_t^X\) is both progressively
		measurable and predictable. Hence
		\[
		(t,\omega)\longmapsto
		(t,X_t(\omega),\mu_t^X)
		\]
		is progressively measurable, whereas
		\[
		(t,\omega)\longmapsto
		(t,X_{t-}(\omega),\mu_t^X)
		\]
		is predictable.
		
		Composing the first map with the Borel coefficient maps \(b\) and
		\(\sigma\) proves \textnormal{(ii)}, while composing the second with
		the Borel map \(\overline\beta\) shows that
		\[
		\Gamma_t^X
		=
		\overline\beta(t,X_{t-},\mu_t^X)
		\]
		is predictable. Since the chosen representative
		\[
		\beta:
		[0,T]\times\mathbb R^d\times
		\mathcal P_2(\mathbb R^d)\times U
		\longrightarrow\mathbb R^d
		\]
		is jointly measurable, it also follows that
		\[
		(t,\omega,z)
		\longmapsto
		\beta(t,X_{t-}(\omega),\mu_t^X,z)
		\]
		is
		\(\mathcal P_{\mathrm{pred}}\otimes\mathcal U\)-measurable.
		This proves \textnormal{(iii)}.
		
		Finally, by \textnormal{(A2)},
		\[
		\mathbb E\int_0^T\int_U
		|\gamma^X(t,z)|^2\,\nu(dz)\,dt
		=
		\mathbb E\int_0^T
		\|\Gamma_t^X\|_H^2\,dt
		\leq M^2T.
		\]
	\end{proof}

	\begin{definition}
		\label{def}
		Let \(a=(b,\sigma,\overline\beta)\) satisfy \textnormal{(A1)}--\textnormal{(A4)}, and let \(\beta\) be the chosen jointly measurable representative from \textnormal{(A1)}. A process \(X\in\mathcal S_T^2\) is called a strong solution of \(\mathcal E_T(a,x)\) if
		\[
		\begin{aligned}
			X_t ={}& x+\int_0^t b(s,X_s,\mathcal L(X_s))\,ds
			+\int_0^t \sigma(s,X_s,\mathcal L(X_s))\,dW_s \\
			&+\int_0^t\int_U \beta(s,X_{s-},\mathcal L(X_s),z)\,\widetilde N(ds,dz)
		\end{aligned}
		\]
		holds as an identity of c\`adl\`ag processes up to indistinguishability.
	\end{definition}
		By Lemma~\ref{lem-Measurability} and \textnormal{(A2)}, all three integrals are well-defined.  Strong well-posedness means the existence of a strong solution together with pathwise uniqueness.

	The following is the classical conclusion; see, for example \cite{Cavallazzi-2025}.
	\begin{proposition}
		\label{prop:wellposed-core}
		For every $a\in\mathfrak A_T^{\mathrm{Lip}}$ and every	$x\in\mathbb R^d$, equation $\mathcal E_T(a,x)$ admits a unique	strong solution
		\[
		X^{a,x}\in\mathcal S_T^2.
		\]
		Moreover,
		\begin{equation}
			\label{eq:core-moment-bound}
			\mathbb E\left[\sup_{0\leq t\leq T}|X_t^{a,x}|^2\right]
			\leq C_T(1+|x|^2),
		\end{equation}
		where $C_T=C_{\mathrm{mom}}(T,d,M)$ depends only on	$T$, $d$, and the common bound $M$.
	\end{proposition}
	
	Consequently, the solution map
	\begin{equation}
		\label{eq:core-solution-map}
		\Phi_{T,x}:
		\mathfrak A_T^{\mathrm{Lip}}
		\longrightarrow\mathcal S_T^2,
		\qquad
		\Phi_{T,x}(a):=X^{a,x},
	\end{equation}
	is well defined.

	The topology on the coefficient space is chosen so that convergence of coefficients implies convergence of their values along non-degenerate It\^o--L\'evy processes.
	
	Let
	\[
	r_\ast:=d+1, \qquad q_\ast:=2r_\ast=2(d+1).
	\]

	Fix once and for all a countable dense family
	\[
	\{\mu_j:j\geq 1\} \subset\mathcal P_2(\mathbb R^d)
	\]
	with respect to $W_2$. For $R>0$, let
	\[
	B_R:=\{x\in\mathbb R^d:|x|<R\}, \qquad Q_{T,R}:=[0,T]\times B_R.
	\]

	For two coefficients
	\[
	a=(b,\sigma,\overline\beta), \qquad  \widetilde a=(\widetilde b,\widetilde\sigma,\overline{\widetilde\beta})
	\]
	in $\mathfrak A_T^{\mathrm{Lip}}$, define
	\begin{equation}
		\label{eq:pointwise-coefficient-distance}
		\begin{aligned}
			\Delta(a,\widetilde a)(t,x)
			:=&\sup_{j\geq 1}\Bigl\{|b(t,x,\mu_j)-\widetilde b(t,x,\mu_j)|+\|\sigma(t,x,\mu_j)
			-\widetilde\sigma(t,x,\mu_j)\|_{\mathrm{HS}}	\\
			&+	\|\overline\beta(t,x,\mu_j)	-\overline{\widetilde\beta}(t,x,\mu_j)\|_H\Bigr\}.
		\end{aligned}
	\end{equation}
	
	For every real number $R>0$, set
	\begin{equation}
		\label{eq:local-krylov-distance}
		\rho_{T,R}(a,\widetilde a)
		:=\|\Delta(a,\widetilde a)\|_{L^{q_\ast}(Q_{T,R})}.
	\end{equation}
	The Krylov-stable distance is defined by
	\begin{equation}
		\label{eq:krylov-stable-metric}
		d_T(a,\widetilde a)
		:=
		\sum_{m=1}^{\infty}
		2^{-m}
		\bigl(1\wedge\rho_{T,m}(a,\widetilde a)\bigr).
	\end{equation}
	Because all coefficients satisfy the uniform estimate \eqref{eq:measure-lipschitz-framework}, the maps
	\[
	\mu\longmapsto b(t,x,\mu),\qquad
	\mu\longmapsto \sigma(t,x,\mu),\qquad
	\mu\longmapsto\overline\beta(t,x,\mu)
	\]
	are continuous. It follows that the supremum in  \eqref{eq:pointwise-coefficient-distance} agrees with the supremum over the whole Wasserstein space, thus
	\begin{equation}
		\label{eq:dense-measure-supremum}
		\begin{aligned}
			\Delta(a,\widetilde a)(t,x)
			=&\sup_{\mu\in\mathcal P_2(\mathbb R^d)}\Bigl\{	|b(t,x,\mu)-\widetilde b(t,x,\mu)|
			+\|\sigma(t,x,\mu)-\widetilde\sigma(t,x,\mu)\|_{\mathrm{HS}}\\
			&+\|\overline\beta(t,x,\mu)-\overline{\widetilde\beta}(t,x,\mu)\|_H\Bigr\}.
		\end{aligned}
	\end{equation}
	The countable family $\{\mu_j\}$ is used only to guarantee the measurability of $\Delta(a,\widetilde a)$.
	
	The terminology ``Krylov-stable'' is justified by the identity
	\[
	\|\Delta(a,\widetilde a)^2\|_{L^{r_\ast}(Q_{T,m})}
	=\|\Delta(a,\widetilde a)\|_{L^{q_\ast}(Q_{T,m})}^2.
	\]
	Thus, a Krylov estimate with exponent $r_\ast$ controls the occupation integral of the squared coefficient error by the square of the local distance $\rho_{T,m}$.
	
	We identify coefficients that agree for $dt\,dx$-almost every $(t,x)$, simultaneously for all
	$\mu\in\mathcal P_2(\mathbb R^d)$. Under this identification, $d_T$ is a metric on $\mathfrak A_T^{\mathrm{Lip}}$.

	\begin{definition}
		\label{def:complete-coefficient-space}
		The coefficient space $(\mathfrak A_T,d_T)$ is the metric
		completion of the Lipschitz core:
		\begin{equation}
			\label{eq:completion-definition}
			\mathfrak A_T:=\overline{\mathfrak A_T^{\mathrm{Lip}}}^{\,d_T}.
		\end{equation}
	\end{definition}
	
	By construction, $(\mathfrak A_T,d_T)$ is a complete metric space and hence a Baire space. Moreover, $\mathfrak A_T^{\mathrm{Lip}}$ is dense in $\mathfrak A_T$.
	
	For every $m>0$, we write $(\mathfrak A_m,d_m)$ for the coefficient
	space and the Krylov-stable metric corresponding to the time horizon
	$T=m$.
	
	For a coefficient triple $a\in\mathfrak A_T$ and $x\in\mathbb R^d$,
	let
	\[
	\operatorname{Sol}_T(a,x)
	\]
	denote the set of all strong solutions of $\mathcal E_T(a,x)$ belonging
	to $\mathcal S_T^2$. This set may be empty and, outside the Lipschitz
	core, it need not be a singleton.
	
	Although $\mathfrak A_T$ is introduced abstractly as a completion, its elements can still be represented by genuine measurable coefficient triples.
	
	The following proposition  is an immediate consequence of Theorem~\ref{thm:completion-representation}, proved in Section~\ref{sec:coefficient-space-core}.
	\begin{proposition}
		\label{prop:coefficient-representation}
		Every element of $\mathfrak A_T$ admits a representative
		\[
		a=(b,\sigma,\overline\beta)
		\]
		satisfying \textnormal{(A1)}--\textnormal{(A4)}. More precisely, if	$(a_n)_{n\geq 1}\subset\mathfrak A_T^{\mathrm{Lip}}$ is	$d_T$-Cauchy, then there exist jointly measurable maps	$b$, $\sigma$, and the $H$-valued map $\overline\beta$, together with a jointly measurable pointwise representative $\beta$, such that, after passing to a subsequence,
		\[
		b_n(t,x,\mu)\longrightarrow b(t,x,\mu),
		\]
		\[
		\sigma_n(t,x,\mu)\longrightarrow\sigma(t,x,\mu),
		\]
		and
		\[
		\overline\beta_n(t,x,\mu)
		\longrightarrow\overline\beta(t,x,\mu)
		\quad\text{in }H
		\]
		for almost every $(t,x)$ and every	$\mu\in\mathcal P_2(\mathbb R^d)$.
		
		The limit preserves the uniform bound	\eqref{eq:uniform-bound-framework}, the Wasserstein Lipschitz condition \eqref{eq:measure-lipschitz-framework}, and the	ellipticity condition \eqref{eq:uniform-ellipticity-framework}.	The resulting representative is unique up to the equivalence	relation induced by $d_T$.
	\end{proposition}
	
	In general, elements of $\mathfrak A_T$ need not be continuous in the state variable. Thus, the completed space contains coefficients with genuine spatial singularities. The role of the non-degenerate diffusion is to control such singularities through the Krylov occupation estimate rather than through pointwise spatial regularity.
	
	\subsection{Main Results}
	
	For a topological space $E$, a subset $\mathfrak R\subset E$ is called  residual if it contains a countable intersection of open dense subsets of $E$. A property is called generic in $E$ if it holds on a residual subset of $E$.
	
	The following theorem is the main result of the paper.

	\begin{theorem}
		\label{thm:generic-main-fixed-horizon}
		Fix $T>0$. There exists a dense residual subset
		\[
		\mathfrak R_T\subset\mathfrak A_T
		\]
		that is independent of the initial state and has the following
		properties. For every $a\in\mathfrak R_T$:
		
		\begin{enumerate}
			\item[\textnormal{(i)}]
			For every $x\in\mathbb R^d$, the jump-type McKean--Vlasov
			equation $\mathcal E_T(a,x)$ admits a unique strong solution
			\[
			X^{a,x}\in\mathcal S_T^2.
			\]
			
			\item[\textnormal{(ii)}]
			The solution map
			\[
			\Phi_T(a):\mathbb R^d\longrightarrow\mathcal S_T^2,
			\qquad
			\Phi_T(a)(x):=X^{a,x},
			\]
			is continuous.
			
			\item[\textnormal{(iii)}]
			If
			\[
			a_k\in\mathfrak A_T^{\mathrm{Lip}},
			\qquad
			d_T(a_k,a)\longrightarrow0,
			\]
			then, for every $m\in\mathbb N$,
			\begin{equation}
				\label{eq:generic-uniform-initial-state-convergence}
				\sup_{x\in\overline B_m}
				\|X^{a_k,x}-X^{a,x}\|_{\mathcal S_T^2}
				\longrightarrow0,
			\end{equation}
			where
			\[
			\overline B_m:=\{x\in\mathbb R^d:|x|\leq m\}.
			\]
			
			\item[\textnormal{(iv)}]
			The limit in
			\eqref{eq:generic-uniform-initial-state-convergence}
			is independent of the approximating sequence. More precisely, if
			\[
			a_k,\widetilde a_k\in\mathfrak A_T^{\mathrm{Lip}},
			\qquad
			d_T(a_k,a)+d_T(\widetilde a_k,a)\longrightarrow0,
			\]
			then, for every $m\in\mathbb N$,
			\begin{equation}
				\label{eq:generic-two-uniform-approximations}
				\sup_{x\in\overline B_m}
				\|X^{a_k,x}-X^{\widetilde a_k,x}\|_{\mathcal S_T^2}
				\longrightarrow0.
			\end{equation}
		\end{enumerate}
	\end{theorem}

	For completeness, we also formulate the corresponding global-time
	statement. For integers $1\leq m<\ell$, let
	\[
	r_{m,\ell}:\mathfrak A_\ell\longrightarrow\mathfrak A_m
	\]
	be the restriction map. Its well-definedness and continuity are proved
	in Proposition~\ref{prop:restriction-projective-completeness}. Define
	the projective-limit coefficient space
	\begin{equation}
		\label{eq:local-coefficient-projective-limit}
		\mathfrak A_{\mathrm{loc}}
		:=
		\left\{
		(a_m)_{m\geq1}:
		a_m\in\mathfrak A_m,\quad
		r_{m,\ell}(a_\ell)=a_m
		\text{ whenever }m<\ell
		\right\},
	\end{equation}
	equipped with
	\begin{equation}
		\label{eq:local-global-metric}
		d_{\mathrm{loc}}(a,\widetilde a)
		:=
		\sum_{m=1}^{\infty}
		2^{-m}
		\left(1\wedge d_m(a_m,\widetilde a_m)\right).
	\end{equation}
	We write $\pi_m(a):=a_m$. The locally Lipschitz core is
	\[
	\mathfrak A_{\mathrm{loc}}^{\mathrm{Lip}}
	:=
	\left\{a\in\mathfrak A_{\mathrm{loc}}:
	\pi_m(a)\in\mathfrak A_m^{\mathrm{Lip}}
	\text{ for every }m\geq1
	\right\}.
	\]
	
	\begin{corollary}
		\label{cor:global-generic-main}
		There exists a dense residual subset
		\[
		\mathfrak R^{\mathrm{loc}}
		\subset
		\mathfrak A_{\mathrm{loc}}
		\]
		that is independent of the initial state and such that, for every
		$a\in\mathfrak R^{\mathrm{loc}}$ and every $x\in\mathbb R^d$, the
		global jump-type McKean--Vlasov equation admits a unique strong
		solution
		\[
		X^{a,x}\in\mathcal S_{\mathrm{loc}}^2,
		\qquad
		\mathcal S_{\mathrm{loc}}^2
		:=
		\left\{X:
		X|_{[0,T]}\in\mathcal S_T^2
		\text{ for every }T>0
		\right\}.
		\]
		For every finite $T>0$, the map
		\[
		x\longmapsto X^{a,x}|_{[0,T]}
		\]
		is continuous from $\mathbb R^d$ to $\mathcal S_T^2$.
		
		Moreover, if
		\[
		a_k\in\mathfrak A_{\mathrm{loc}}^{\mathrm{Lip}},
		\qquad
		d_{\mathrm{loc}}(a_k,a)\longrightarrow0,
		\]
		then, for every finite $T>0$ and every $m\in\mathbb N$,
		\begin{equation}
			\label{eq:global-uniform-initial-state-convergence}
			\sup_{x\in\overline B_m}
			\|X^{a_k,x}-X^{a,x}\|_{\mathcal S_T^2}
			\longrightarrow0,
		\end{equation}
		independently of the approximating sequence.
	\end{corollary}
	
	\section{A Stopped Krylov--Melnikov Occupation Estimate for
		It\^o--L\'evy Processes}
	\label{sec:krylov-ito-levy}
	
	The distribution and occupation estimates underlying this section are classical. Krylov proved the bounded-domain exit-time estimate for non-degenerate It\^o processes \cite{Krylov-1969,Krylov-1980}. 	Melnikov and Krylov developed semimartingale versions
	\cite{Melnikov-1983,Krylov-1987}, while jump-process extensions appear	in \cite{Anulova-Pragarauskas-1977}. 
	
	We give a self-contained stopped conditional version in the exact form	used later. The process may have an infinite-activity compensated	Poisson part subject only to a uniform second-moment bound; the time	interval may start and end at random stopping times; the test function
	is a non-negative Borel element of $L^{d+1}$; and the possible jump	overshoot at the exit time is handled directly. The proof also makes the dependence of the constant on the jump second moment explicit. These features allow the estimate to be applied uniformly to coefficient
	approximations and to all candidate solutions in the set-valued	stability argument. The only non-elementary analytic input is the	parabolic Krylov--Aleksandrov construction stated and regularized in	Appendix~\ref{app:aleksandrov-regularization}.

	Whenever a function in $C_c^\infty(0,T)\times B_R)$ is regarded as a function on	$\mathbb R\times\mathbb R^d$, it is extended by zero. Since its support is	compactly contained in $(0,T)\times B_R$, this zero extension is again	smooth.
	
	\subsection{The stopped conditional estimate}
	
	Consider the following $d$-dimensional	It\^o--L\'evy process:
	\begin{equation}
		\label{eq:general-ito-levy-process}
		Y_t	=Y_0+\int_0^t b_s\,ds+\int_0^t\sigma_s\,dW_s+\int_0^t\int_U g_s(z)\,\widetilde N(ds,dz),
		\qquad 0\le t\le T,
	\end{equation}
	where \(b\) and \(\sigma\) are predictable processes with values in	\(\mathbb R^d\) and \(\mathbb R^{d\times d}\), respectively, and
	\[g:[0,T]\times\Omega\times U\longrightarrow\mathbb R^d\]
	is
	\(\mathcal P_{\mathrm{pred}}\otimes\mathcal U\)-measurable. We assume	that
	\[
	\int_U|g_t(z)|^2\,\nu(dz)<\infty
	\]
	for \(dt\,d\mathbb P\)-almost every \((t,\omega)\). Equivalently,
	\[
	\overline g_t:=[g_t(\cdot)]_H
	\]
	is an \(H\)-valued predictable process. Set
	\begin{equation}
		\label{eq:continuous-covariance}
		A_s:=\frac12\sigma_s\sigma_s^\top.
	\end{equation}
	Assume that, for $ds\,d\mathbb P$-almost every $(s,\omega)$,
	\begin{equation}
		\label{eq:ito-levy-uniform-bounds}
		|b_s|\le K_b,\qquad\|\sigma_s\|_{\mathrm{HS}}\le K_\sigma,\qquad	\int_U|g_s(z)|^2\,\nu(dz)\le K_g^2.
	\end{equation}
	Define the exit time
	\[
	\tau_R(Y):=	\inf\{t\in[0,T]:Y_t\notin B_R\}\wedge T.
	\]
	
	We establish a localized conditional version of the classical Aleksandrov–Krylov estimate tailored to the present It\^o–L\'evy setting.
	
	\begin{proposition}
		\label{thm:local-weighted-krylov}
		Assume \eqref{eq:ito-levy-uniform-bounds}. Let $\gamma$ and $\tau$ be	stopping times satisfying
		\[
		0\le\gamma\le\tau\le\tau_R(Y).
		\]
		Then there exists a constant
		\[
		C_{\mathrm K}=C_{\mathrm K}(d,T,R,K_b,K_\sigma,K_g)
		\]
		such that, for every non-negative Borel function
		$f\in L^{d+1}(Q_{T,R})$,
		\begin{equation}
			\label{eq:weighted-conditional-krylov}
			\mathbb E\left[\left.\int_\gamma^\tau(\det A_s)^{1/(d+1)}f(s,Y_s)\,ds\,\right|\mathcal F_\gamma\right]
			\le
			C_{\mathrm K}\|f\|_{L^{d+1}(Q_{T,R})}
		\end{equation}
		almost surely. On the event $\{\gamma=\tau\}$ the left-hand side is understood to be zero.
	\end{proposition}
	
\begin{proof}
	We first prove the estimate for smooth functions and then pass to
	arbitrary \(L^{d+1}\)-functions by identifying the stopped occupation
	measure.

	\textbf{Step 1: smooth non-negative functions and the jump operator.}
	
	Let
	\[
	f\in C_c^\infty((0,T)\times B_R),
	\qquad
	f\geq0.
	\]
	By Lemma~\ref{lem:global-aleksandrov-regularization}, there exist
	\(u^n\in C^{1,2}([0,T]\times\mathbb R^d)\) and non-negative smooth
	functions \(f^n\) such that
	\begin{equation}
		\label{eq:regularized-f-convergence}
		\|f^n-f\|_{L^\infty(\mathbb R\times\mathbb R^d)}
		\longrightarrow0,
	\end{equation}
	\(x\mapsto u^n(t,x)\) is convex, and, for every symmetric
	non-negative definite matrix \(A\),
	\begin{equation}
		\label{eq:regularized-aleksandrov-inequality}
		\partial_tu^n(t,x)
		+\operatorname{tr}\!\left(A D^2u^n(t,x)\right)
		\geq
		\kappa_d(\det A)^{1/(d+1)}f^n(t,x)
	\end{equation}
	on \([0,T]\times B_R\), here \(D^2u\) denotes the Hessian with respect to the spatial variable. Moreover, with
	\[
	L_f
	:=
	C_{d,T,R}\|f\|_{L^{d+1}(Q_{T,R})},
	\]
	one has, uniformly in \(n\),
	\begin{equation}
		\label{eq:regularized-growth-gradient-bounds}
		|u^n(t,x)|
		\leq
		L_f(1+|x|),
		\qquad
		|\nabla u^n(t,x)|
		\leq
		L_f.
	\end{equation}
	
	For fixed \(n\), set
	\begin{equation}
		\label{eq:jump-operator}
		\mathcal J_su^n(t,x)
		:=
		\int_U r_n(t,x,g_s(z))\,\nu(dz),
	\end{equation}
	where
	\[
	r_n(t,x,h)
	:=
	u^n(t,x+h)-u^n(t,x)-\nabla u^n(t,x)\cdot h.
	\]
	Convexity gives
	\begin{equation}
		\label{eq:jump-operator-positive}
		r_n(t,x,h)\geq0,
		\qquad
		\mathcal J_su^n(t,x)\geq0.
	\end{equation}
	
	We next verify that the integral in
	\eqref{eq:jump-operator} is finite along the trajectory before the
	exit time. Put
	\[
	H_{n,R}
	:=
	\sup_{(t,x)\in[0,T]\times\overline B_{R+1}}
	\|D^2u^n(t,x)\|
	<\infty.
	\]
	If \(x\in\overline B_R\) and \(|h|\leq1\), Taylor's formula and
	convexity imply
	\begin{equation}
		\label{eq:small-jump-remainder}
		0
		\leq
		r_n(t,x,h)
		\leq
		\frac12 H_{n,R}|h|^2.
	\end{equation}
	If \(|h|>1\), the global gradient bound in
	\eqref{eq:regularized-growth-gradient-bounds} gives
	\begin{equation}
		\label{eq:large-jump-remainder}
		0
		\leq
		r_n(t,x,h)
		\leq
		2L_f|h|
		\leq
		2L_f|h|^2.
	\end{equation}
	Consequently, for \(x\in\overline B_R\),
	\begin{equation}
		\label{eq:jump-operator-integrability}
		0
		\leq
		\mathcal J_su^n(t,x)
		\leq
		\left(
		\frac12 H_{n,R}+2L_f
		\right)
		\int_U|g_s(z)|^2\,\nu(dz)
		<\infty.
	\end{equation}
	This small-jump/large-jump decomposition is the point at which the
	second-moment assumption on \(g\) is used to control the compensator
	term.

	\textbf{Step 2: the stochastic integrals are true martingales.}
	
	Define
	\[
	M_t^{W,n}
	:=
	\int_0^t
	\nabla u^n(s,Y_{s-})\sigma_s\,dW_s
	\]
	and
	\[
	M_t^{N,n}
	:=
	\int_0^t\int_U
	\bigl[
	u^n(s,Y_{s-}+g_s(z))
	-u^n(s,Y_{s-})
	\bigr]
	\,\widetilde N(ds,dz).
	\]
	By \eqref{eq:regularized-growth-gradient-bounds} and
	\eqref{eq:ito-levy-uniform-bounds},
	\begin{equation}
		\label{eq:brownian-martingale-L2}
		\mathbb E\int_0^T
		|\nabla u^n(s,Y_{s-})\sigma_s|^2\,ds
		\leq
		L_f^2K_\sigma^2T
		<\infty.
	\end{equation}
	Since \(u^n(t,\cdot)\) is globally \(L_f\)-Lipschitz,
	\begin{equation}
		\label{eq:poisson-martingale-L2}
		\begin{aligned}
			&\mathbb E\int_0^T\int_U
			\bigl|
			u^n(s,Y_{s-}+g_s(z))
			-u^n(s,Y_{s-})
			\bigr|^2
			\,\nu(dz)\,ds
			\\
			&\qquad
			\leq
			L_f^2K_g^2T
			<\infty.
		\end{aligned}
	\end{equation}
	Thus \(M^{W,n}\) and \(M^{N,n}\) are square-integrable martingales,
	not merely local martingales. Since \(\gamma\) and \(\tau\) are
	bounded, optional sampling yields
	\begin{equation}
		\label{eq:conditional-martingale-increments-zero}
		\mathbb E\left[
		M_\tau^{W,n}-M_\gamma^{W,n}
		+
		M_\tau^{N,n}-M_\gamma^{N,n}
		\,\middle|\,
		\mathcal F_\gamma
		\right]
		=0.
	\end{equation}

	\textbf{Step 3: It\^o--L\'evy formula and the smooth estimate.}
	
	To justify the It\^o--L\'evy formula without assuming a global bound
	on \(D^2u^n\), choose
	\[
	\chi\in C_c^\infty(\mathbb R^d),
	\qquad
	0\leq\chi\leq1,
	\qquad
	\chi=1\ \text{on }B_1,
	\qquad
	\operatorname{supp}\chi\subset B_2,
	\]
	and set
	\[
	\chi_m(x):=\chi(x/m),
	\qquad
	u_m^n(t,x):=\chi_m(x)u^n(t,x).
	\]
	Then
	\[
	\nabla u_m^n
	=
	\chi_m\nabla u^n
	+
	u^n\nabla\chi_m.
	\]
	Since
	\[
	\nabla\chi_m(x)
	=
	\frac1m\nabla\chi(x/m)
	\]
	and \(\nabla\chi_m\) is supported in
	\(\{m\leq|x|\leq2m\}\), the growth estimate
	\eqref{eq:regularized-growth-gradient-bounds} gives
	\[
	\begin{aligned}
		|u^n(t,x)|\,|\nabla\chi_m(x)|
		&\leq
		L_f(1+2m)
		\frac{\|\nabla\chi\|_\infty}{m}
		\\
		&\leq
		C_\chi L_f
	\end{aligned}
	\]
	on the support of \(\nabla\chi_m\). Hence
	\begin{equation}
		\label{eq:truncated-uniform-gradient}
		\sup_{m\geq1}
		\|\nabla u_m^n\|_{L^\infty(\mathbb R\times\mathbb R^d)}
		\leq
		C_\chi L_f.
	\end{equation}
	
	For
	\[
	r_{m,n}(t,x,h)
	:=
	u_m^n(t,x+h)
	-u_m^n(t,x)
	-\nabla u_m^n(t,x)\cdot h,
	\]
	\eqref{eq:truncated-uniform-gradient} implies
	\[
	|r_{m,n}(t,x,h)|
	\leq
	C_\chi L_f|h|
	\leq
	C_\chi L_f|h|^2,
	\qquad |h|>1.
	\]
	If \(m>R+1\), \(x\in B_R\), and \(|h|\leq1\), then
	\(x,x+h\in B_{R+1}\subset B_m\), and therefore
	\[
	u_m^n(t,x)=u^n(t,x),
	\qquad
	u_m^n(t,x+h)=u^n(t,x+h),
	\qquad
	\nabla u_m^n(t,x)=\nabla u^n(t,x).
	\]
	Consequently,
	\[
	0
	\leq
	r_{m,n}(t,x,h)
	=
	r_n(t,x,h)
	\leq
	\frac12H_{n,R}|h|^2.
	\]
	Thus, for \(m>R+1\) and \(x\in B_R\),
	\begin{equation}
		\label{eq:truncated-jump-dominator}
		|r_{m,n}(t,x,h)|
		\leq
		\frac12H_{n,R}|h|^2
		\mathbf1_{\{|h|\leq1\}}
		+
		C_\chi L_f|h|^2
		\mathbf1_{\{|h|>1\}}.
	\end{equation}
	After substituting \(h=g_s(z)\), the right-hand side is integrable
	with respect to
	\(\nu(dz)\,ds\,d\mathbb P\) by
	\eqref{eq:ito-levy-uniform-bounds}.
	
Apply the It\^o--L\'evy formula to \(u_m^n(t,Y_t)\) on
\([\gamma,\tau]\) and take conditional expectation with respect to
\(\mathcal F_\gamma\). By
\eqref{eq:conditional-martingale-increments-zero}, the Brownian and
compensated Poisson martingale increments have conditional expectation
zero. It remains to pass to the limit \(m\to\infty\) in the remaining
terms.

Since \(\tau\leq\tau_R(Y)\), we have
\[
Y_{s-}\in\overline B_R,
\qquad
\gamma<s<\tau.
\]
If \(m>R+1\), then \(\chi_m\equiv1\) on a neighborhood of
\(\overline B_R\). Hence, on \(\{\gamma<s<\tau\}\),
\[
\partial_su_m^n(s,Y_{s-})
=
\partial_su^n(s,Y_{s-}),
\]
and likewise
\[
\nabla u_m^n(s,Y_{s-})
=
\nabla u^n(s,Y_{s-}),
\qquad
D^2u_m^n(s,Y_{s-})
=
D^2u^n(s,Y_{s-}).
\]
Thus the time-derivative, drift, and continuous second-order terms
already coincide with those for \(u^n\) once \(m>R+1\).

For the jump compensator, for every fixed \((s,\omega,z)\),
\[
r_{m,n}\bigl(s,Y_{s-},g_s(z)\bigr)
\longrightarrow
r_n\bigl(s,Y_{s-},g_s(z)\bigr)
\qquad\text{as }m\to\infty.
\]
Together with the uniform bound
\eqref{eq:truncated-jump-dominator}, dominated convergence yields
\[
\mathcal J_su_m^n(s,Y_{s-})
\longrightarrow
\mathcal J_su^n(s,Y_{s-})
\]
in the corresponding integrable sense, and hence also after integration
over \([\gamma,\tau]\) and conditional expectation.

Finally,
\[
|u_m^n(t,x)|
\leq
|u^n(t,x)|
\leq
L_f(1+|x|)
\]
for all \(m\). Since the boundedness assumptions imply
\(Y_\gamma,Y_\tau\in L^1\), while
\[
u_m^n(\gamma,Y_\gamma)\to u^n(\gamma,Y_\gamma),
\qquad
u_m^n(\tau,Y_\tau)\to u^n(\tau,Y_\tau)
\]
almost surely, conditional dominated convergence applies to the endpoint
terms.

Letting \(m\to\infty\) therefore gives
\begin{equation}
	\label{eq:ito-formula-krylov-revised}
	\begin{aligned}
		&\mathbb E\left[
		u^n(\tau,Y_\tau)
		-u^n(\gamma,Y_\gamma)
		\,\middle|\,
		\mathcal F_\gamma
		\right]
		\\
		=&
		\mathbb E\left[
		\left.
		\int_\gamma^\tau
		\Bigl[
		\partial_su^n(s,Y_{s-})
		+b_s\cdot\nabla u^n(s,Y_{s-})
		\right.\right.
		\\
		&\left.\left.
		+\operatorname{tr}\!\left(
		A_sD^2u^n(s,Y_{s-})
		\right)
		+\mathcal J_su^n(s,Y_{s-})
		\Bigr]\,ds
		\,\right|\,
		\mathcal F_\gamma
		\right].
	\end{aligned}
\end{equation}

	For \(ds\,d\mathbb P\)-almost every \((s,\omega)\) with
	\(\gamma<s<\tau\), one has \(Y_{s-}\in B_R\). Therefore, by
	\eqref{eq:regularized-aleksandrov-inequality},
	\eqref{eq:jump-operator-positive}, and \(|b_s|\leq K_b\),
	\begin{equation}
		\label{eq:generator-lower-bound-revised}
		\begin{aligned}
			&\partial_su^n(s,Y_{s-})
			+b_s\cdot\nabla u^n(s,Y_{s-})
			+\operatorname{tr}\!\left(
			A_sD^2u^n(s,Y_{s-})
			\right)
			+\mathcal J_su^n(s,Y_{s-})
			\\
			\geq&
			\kappa_d(\det A_s)^{1/(d+1)}
			f^n(s,Y_{s-})
			-K_bL_f.
		\end{aligned}
	\end{equation}
	Since a c\`adl\`ag path has at most countably many jumps,
	\(Y_s=Y_{s-}\) for Lebesgue-almost every \(s\). Combining
	\eqref{eq:ito-formula-krylov-revised} and
	\eqref{eq:generator-lower-bound-revised} gives
	\begin{equation}
		\label{eq:preliminary-weighted-bound-revised}
		\begin{aligned}
			&\kappa_d
			\mathbb E\left[
			\left.
			\int_\gamma^\tau
			(\det A_s)^{1/(d+1)}
			f^n(s,Y_s)\,ds
			\,\right|\,
			\mathcal F_\gamma
			\right]
			\\
			\leq&
			\mathbb E\left[
			u^n(\tau,Y_\tau)
			-u^n(\gamma,Y_\gamma)
			\,\middle|\,
			\mathcal F_\gamma
			\right]
			+K_bTL_f.
		\end{aligned}
	\end{equation}
	
	Let
	\[
	E:=\{\gamma<\tau\}\in\mathcal F_\gamma.
	\]
	On \(E\), one has \(\gamma<\tau_R(Y)\), and hence
	\(Y_\gamma\in B_R\), so \(|Y_\gamma|\leq R\).
	Moreover, conditional Cauchy--Schwarz and the isometries for the
	Brownian and compensated Poisson integrals imply
	\begin{equation}
		\label{eq:conditional-terminal-moment-revised}
		\begin{aligned}
			\mathbf1_E
			\mathbb E\left[
			|Y_\tau|
			\,\middle|\,
			\mathcal F_\gamma
			\right]
			&\leq
			\mathbf1_E
			\left(
			|Y_\gamma|
			+K_bT
			+K_\sigma T^{1/2}
			+K_gT^{1/2}
			\right)
			\\
			&\leq
			\mathbf1_E
			\left[
			R+K_bT
			+(K_\sigma+K_g)T^{1/2}
			\right].
		\end{aligned}
	\end{equation}
	In particular,
	\eqref{eq:regularized-growth-gradient-bounds} yields
	\begin{equation}
		\label{eq:terminal-test-function-bound-revised}
		\mathbf1_E
		\left|
		\mathbb E\left[
		u^n(\tau,Y_\tau)
		-u^n(\gamma,Y_\gamma)
		\,\middle|\,
		\mathcal F_\gamma
		\right]
		\right|
		\leq
		C\mathbf1_E
		\|f\|_{L^{d+1}(Q_{T,R})},
	\end{equation}
	where
	\[
	C=C(d,T,R,K_b,K_\sigma,K_g)
	\]
	is independent of \(n\), \(\gamma\), \(\tau\), and \(f\).
	
	On \(E^c=\{\gamma=\tau\}\), the occupation integral is zero.
	Hence \eqref{eq:preliminary-weighted-bound-revised} and
	\eqref{eq:terminal-test-function-bound-revised} imply
	\begin{equation}
		\label{eq:weighted-krylov-regularized}
		\mathbb E\left[
		\left.
		\int_\gamma^\tau
		(\det A_s)^{1/(d+1)}
		f^n(s,Y_s)\,ds
		\,\right|\,
		\mathcal F_\gamma
		\right]
		\leq
		C_{\mathrm K}
		\|f\|_{L^{d+1}(Q_{T,R})}.
	\end{equation}
	The upper bound on \(\sigma\) gives
	\begin{equation}
		\label{eq:determinant-upper-bound}
		(\det A_s)^{1/(d+1)}
		\leq
		C(d,K_\sigma).
	\end{equation}
	Therefore \eqref{eq:regularized-f-convergence} permits passage to the
	limit in \eqref{eq:weighted-krylov-regularized}, and we obtain
	\begin{equation}
		\label{eq:weighted-krylov-smooth-revised}
		\mathbb E\left[
		\left.
		\int_\gamma^\tau
		(\det A_s)^{1/(d+1)}
		f(s,Y_s)\,ds
		\,\right|\,
		\mathcal F_\gamma
		\right]
		\leq
		C_{\mathrm K}
		\|f\|_{L^{d+1}(Q_{T,R})}
	\end{equation}
	for every non-negative
	\(f\in C_c^\infty((0,T)\times B_R)\).

	\textbf{Step 4: extension from smooth functions to	\(L^{d+1}(Q_{T,R})\).}
	
	For \(A\in\mathcal F_\gamma\), define a finite Borel measure	\(\Gamma_A\) on \(Q_{T,R}\) by
	\begin{equation}
		\label{eq:conditional-occupation-measure}
		\Gamma_A(B)	:=
		\mathbb E\left[
		\mathbf1_A
		\int_\gamma^\tau
		(\det A_s)^{1/(d+1)}
		\mathbf1_B(s,Y_s)\,ds
		\right],
		\qquad
		B\in\mathcal B(Q_{T,R}).
	\end{equation}
	Finiteness follows from
	\eqref{eq:determinant-upper-bound}. Moreover,
	\[
	\Gamma_A\bigl(\{0,T\}\times B_R\bigr)=0.
	\]
	Set
	\[
	D_{T,R}:=(0,T)\times B_R.
	\]
	Thus \(\Gamma_A\) is determined by its restriction to \(D_{T,R}\),
	and
	\[
	L^{d+1}(D_{T,R})
	=
	L^{d+1}(Q_{T,R})
	\]
	up to the usual identification of functions on Lebesgue-null sets.
	
	Multiplying \eqref{eq:weighted-krylov-smooth-revised} by
	\(\mathbf1_A\) and taking expectations gives
	\begin{equation}
		\label{eq:occupation-measure-smooth-bound}
		\int_{D_{T,R}}\varphi\,d\Gamma_A
		\leq
		C_{\mathrm K}\mathbb P(A)
		\|\varphi\|_{L^{d+1}(D_{T,R})}
	\end{equation}
	for every non-negative
	\(\varphi\in C_c^\infty(D_{T,R})\).
	
	Since \(D_{T,R}\) is open, every non-negative
	\(\varphi\in C_c(D_{T,R})\) can be uniformly approximated by
	non-negative functions in \(C_c^\infty(D_{T,R})\).
	Consequently, \eqref{eq:occupation-measure-smooth-bound} also holds
	for every non-negative \(\varphi\in C_c(D_{T,R})\).
	Applying this estimate to \(\varphi^+\) and \(\varphi^-\) gives
	\[
	\left|
	\int_{D_{T,R}}\varphi\,d\Gamma_A
	\right|
	\leq
	C_{\mathrm K}\mathbb P(A)
	\|\varphi\|_{L^{d+1}(D_{T,R})},
	\qquad
	\varphi\in C_c(D_{T,R}).
	\]
	
	Since \(C_c^\infty(D_{T,R})\) is dense in
	\(L^{d+1}(D_{T,R})\), the functional
	\[
	\varphi
	\longmapsto
	\int_{D_{T,R}}\varphi\,d\Gamma_A
	\]
	extends uniquely to a bounded linear functional on
	\(L^{d+1}(D_{T,R})\). By
	\(L^p\)--\(L^{p'}\) duality, there exists
	\[
	h_A
	\in
	L^{(d+1)/d}(D_{T,R})
	\]
	such that
	\[
	\int_{D_{T,R}}\varphi\,d\Gamma_A
	=
	\int_{D_{T,R}}
	\varphi(t,x)h_A(t,x)\,dt\,dx
	\]
	for \(\varphi\in C_c(D_{T,R})\), and
	\[
	\|h_A\|_{L^{(d+1)/d}(D_{T,R})}
	\leq
	C_{\mathrm K}\mathbb P(A).
	\]
	The finite Radon measures
	\(\Gamma_A|_{D_{T,R}}\) and
	\(h_A(t,x)\,dt\,dx\) agree on \(C_c(D_{T,R})\).
	By uniqueness in the Riesz representation theorem,
	\[
	\Gamma_A|_{D_{T,R}}(dt,dx)
	=
	h_A(t,x)\,dt\,dx.
	\]
	Since \(\Gamma_A\) gives no mass to
	\(\{0,T\}\times B_R\), we may equivalently write, after extending
	\(h_A\) arbitrarily on this Lebesgue-null set,
	\begin{equation}
		\label{eq:occupation-density}
		\Gamma_A(dt,dx)
		=
		h_A(t,x)\,dt\,dx
		\qquad\text{on }Q_{T,R}.
	\end{equation}
	
	Therefore, for every non-negative Borel function
	\(f\in L^{d+1}(Q_{T,R})\), H\"older's inequality gives
	\[
	\begin{aligned}
		\int_{Q_{T,R}}f\,d\Gamma_A
		&=
		\int_{Q_{T,R}}
		f(t,x)h_A(t,x)\,dt\,dx
		\\
		&\leq
		\|f\|_{L^{d+1}(Q_{T,R})}
		\|h_A\|_{L^{(d+1)/d}(Q_{T,R})}
		\\
		&\leq
		C_{\mathrm K}\mathbb P(A)
		\|f\|_{L^{d+1}(Q_{T,R})}.
	\end{aligned}
	\]
	By the definition of \(\Gamma_A\), first for indicator functions
	and then for non-negative Borel functions by monotone convergence,
	this is precisely
	\begin{equation}
		\label{eq:occupation-measure-general-bound}
		\mathbb E\left[
		\mathbf1_A
		\int_\gamma^\tau
		(\det A_s)^{1/(d+1)}
		f(s,Y_s)\,ds
		\right]
		\leq
		C_{\mathrm K}\mathbb P(A)
		\|f\|_{L^{d+1}(Q_{T,R})}.
	\end{equation}
	
	Since \eqref{eq:occupation-measure-general-bound} holds for every
	\(A\in\mathcal F_\gamma\), the defining property of conditional
	expectation gives
	\eqref{eq:weighted-conditional-krylov}. Indeed, if
	\[
	\mathbb E\left[
	\left.
	\int_\gamma^\tau
	(\det A_s)^{1/(d+1)}
	f(s,Y_s)\,ds
	\,\right|\,
	\mathcal F_\gamma
	\right]
	>
	C_{\mathrm K}
	\|f\|_{L^{d+1}(Q_{T,R})}
	\]
	on some set \(A\in\mathcal F_\gamma\) of positive probability, then
	integrating over \(A\) would contradict
	\eqref{eq:occupation-measure-general-bound}.
\end{proof}

	The weighted estimate becomes an ordinary occupation estimate under	uniform ellipticity.
	
	\begin{corollary}
		\label{cor:unweighted-local-krylov}
		In addition to the assumptions of
		Theorem~\ref{thm:local-weighted-krylov}, suppose that
		\begin{equation}
			\label{eq:krylov-uniform-ellipticity}
			A_s\ge\lambda I_d
		\end{equation}
		for some $\lambda>0$, for $ds\,d\mathbb P$-almost every
		$(s,\omega)$. Then, for every non-negative Borel function
		$f\in L^{d+1}(Q_{T,R})$,
		\begin{equation}
			\label{eq:unweighted-conditional-krylov}
			\mathbb E\left[
			\left.
			\int_\gamma^\tau f(s,Y_s)\,ds
			\,\right|\mathcal F_\gamma
			\right]
			\le
			C_{\mathrm K}^{0}
			\|f\|_{L^{d+1}(Q_{T,R})},
		\end{equation}
		where 
		\[
		C_{\mathrm K}^{0}=	C_{\mathrm K}\lambda^{-d/(d+1)}.
		\]
	\end{corollary}

	\subsection{Occupation estimates for coefficient errors}
	
	We now formulate the consequence needed in the stability arguments. Recall from Section~\ref{sec:framework-main} that	\(q_\ast=2(d+1).\)
	
	Let
	\[
	a=(b,\sigma,\overline\beta),
	\qquad
	\widetilde a
	=(\widetilde b,\widetilde\sigma,\overline{\widetilde\beta})
	\]
	be two coefficient triples satisfying the common structural bounds	of Section~\ref{sec:framework-main}. Recall the pointwise distance
	\[
	\Delta(a,\widetilde a)(t,x)
	\]
	defined in
	\eqref{eq:pointwise-coefficient-distance}. For a measurable	probability-measure-valued process
	\[
	\mu_t\in\mathcal P_2(\mathbb R^d),
	\]
	define
	\begin{equation}
		\label{eq:total-coefficient-error}
		\begin{aligned}
			\mathfrak e_{a,\widetilde a}(t,x,\mu)
			:={}&
			|b(t,x,\mu)-\widetilde b(t,x,\mu)|^2+	\|\sigma(t,x,\mu)-\widetilde\sigma(t,x,\mu)\|_{\mathrm{HS}}^2\\
			&+	\|\overline\beta(t,x,\mu)-\overline{\widetilde\beta}(t,x,\mu)\|_H^2.
		\end{aligned}
	\end{equation}
	By the definition of $\Delta$,
	\begin{equation}
		\label{eq:error-controlled-by-delta}
		\mathfrak e_{a,\widetilde a}(t,x,\mu)
		\leq
		(\Delta(a,\widetilde a)(t,x))^2. 	
	\end{equation}
	By applying Corollary~\ref{cor:unweighted-local-krylov} to
	\[
	f(t,x)
	=
	(\Delta(a,\widetilde a)(t,x))^2,
	\]
	we have 
	\begin{corollary}
		\label{cor:coefficient-error-occupation}
		Assume that $Y$ satisfies
		\eqref{eq:general-ito-levy-process},
		\eqref{eq:ito-levy-uniform-bounds}, and
		\eqref{eq:krylov-uniform-ellipticity}. Let
		\[
		0\leq\gamma\leq\tau\leq\tau_R(Y)
		\]
		be stopping times. Then, for every measurable process	$\mu_t\in\mathcal P_2(\mathbb R^d)$,  one has
		\begin{equation}
			\label{eq:combined-coefficient-error-estimate}
			\begin{aligned}
				\mathbb E\left[
				\left.
				\int_\gamma^\tau
				\mathfrak e_{a,\widetilde a}
				(s,Y_s,\mu_s)\,ds
				\,\right|\,\mathcal F_\gamma
				\right]
				\leq
				C_{\mathrm K}^{0}
				\left\|
				\Delta(a,\widetilde a)
				\right\|_{L^{2d+2}(Q_{T,R})}^2.
			\end{aligned}
		\end{equation}
		Equivalently, in terms of the local metric introduced in
		\eqref{eq:local-krylov-distance},
		\begin{equation}
			\label{eq:combined-error-local-metric}
			\mathbb E\left[
			\left.
			\int_\gamma^\tau
			\mathfrak e_{a,\widetilde a}
			(s,Y_s,\mu_s)\,ds
			\,\right|\,\mathcal F_\gamma
			\right]
			\leq
			C_{\mathrm K}^{0}
			\rho_{T,R}(a,\widetilde a)^2.
		\end{equation}
	\end{corollary}

	In particular, each coefficient component satisfies a separate occupation estimate:
	\begin{equation}
		\label{eq:drift-error-occupation}
		\begin{aligned}
			\mathbb E\left[
			\left.
			\int_\gamma^\tau
			|b(s,Y_s,\mu_s)
			-\widetilde b(s,Y_s,\mu_s)|^2\,ds
			\,\right|\,\mathcal F_\gamma
			\right]
			\leq
			C_{\mathrm K}^{0}
			\|
			\sup_{\mu\in\mathcal P_2(\mathbb R^d)}
			|b-\widetilde b|(\cdot,\cdot,\mu)
			\|_{L^{2d+2}(Q_{T,R})}^2,
		\end{aligned}
	\end{equation}
	\begin{equation}
		\label{eq:diffusion-error-occupation}
		\begin{aligned}
			\mathbb E\left[
			\left.
			\int_\gamma^\tau
			\|\sigma(s,Y_s,\mu_s)
			-\widetilde\sigma(s,Y_s,\mu_s)\|_{\mathrm{HS}}^2\,ds
			\,\right|\,\mathcal F_\gamma
			\right]
			\leq
			C_{\mathrm K}^{0}
			\|
			\sup_{\mu\in\mathcal P_2(\mathbb R^d)}
			\|\sigma-\widetilde\sigma\|_{\mathrm{HS}}
			(\cdot,\cdot,\mu)
			\|_{L^{2d+2}(Q_{T,R})}^2,
		\end{aligned}
	\end{equation}
	and
	\begin{equation}
		\label{eq:jump-error-occupation}
		\begin{aligned}
			\mathbb E\left[
			\left.
			\int_\gamma^\tau
			\|\overline\beta(s,Y_s,\mu_s)
			-\overline{\widetilde\beta}(s,Y_s,\mu_s)\|_H^2\,ds
			\,\right|\,\mathcal F_\gamma
			\right]
			\leq
			C_{\mathrm K}^{0}
			\|
			\sup_{\mu\in\mathcal P_2(\mathbb R^d)}
			\|\overline\beta-\overline{\widetilde\beta}\|_H
			(\cdot,\cdot,\mu)
			\|_{L^{2d+2}(Q_{T,R})}^2.
		\end{aligned}
	\end{equation}
	
	\begin{remark}
		The constant in
		\eqref{eq:combined-coefficient-error-estimate} is uniform over all
		processes satisfying the same bounds
		\[
		K_b,\quad K_\sigma,\quad K_g,\quad\lambda.
		\]
		In particular, it is independent of the spatial regularity of the
		coefficients and of any approximation index. This uniformity is
		essential when the estimate is applied to sequences of equations in
		the complete coefficient space $\mathfrak A_T$.
	\end{remark}
	
	\section{Coefficient space and approximation core}
	\label{sec:coefficient-space-core}
	
	In Section~\ref{sec:framework-main}, the coefficient space
	$(\mathfrak A_T)$ was defined as the completion of the Lipschitz core
	$(\mathfrak A_T^{\mathrm{Lip}})$ under the Krylov-stable metric
	$d_T$. The purpose of the present section is to show that this
	abstract completion can be identified with equivalence classes of
	genuine measurable coefficient triples. We also verify that the
	uniform structural conditions, including ellipticity, are preserved
	under completion.
	
	\subsection{Measurable Representatives and Realization of the Completion}
	
	Throughout this section, all coefficient triples in $\mathfrak A_T^{\mathrm{Lip}}$ satisfy  \textnormal{(A1)}--\textnormal{(A4)} with these same constants.
	
	For notational convenience, introduce the separable Banach space
	\begin{equation}
		\label{eq:coefficient-value-space}
		\mathbb V
		:=
		\mathbb R^d
		\times\mathbb R^{d\times d}
		\times H
	\end{equation}
	with norm
	\begin{equation}
		\label{eq:coefficient-value-norm}
		|(u,S,h)|_{\mathbb V}
		:=
		|u|+\|S\|_{\mathrm{HS}}+\|h\|_H.
	\end{equation}
	For a coefficient triple
	\[
	a=(b,\sigma,\overline\beta),
	\]
	we write
	\[
	\mathbf a(t,x,\mu)
	:=
	\bigl(
	b(t,x,\mu),
	\sigma(t,x,\mu),
	\overline\beta(t,x,\mu)
	\bigr)
	\in\mathbb V.
	\]
	
	All almost-everywhere statements below are understood with respect
	to Lebesgue measure $dt\,dx$ on
	$[0,T]\times\mathbb R^d$. Whenever structural conditions are
	asserted for every
	$\mu\in\mathcal P_2(\mathbb R^d)$, the exceptional
	$dt\,dx$-null set is independent of $\mu$.

	Recall that
	\[
	\{\mu_j:j\geq1\}
	\]
	is a fixed countable dense subset of
	$\mathcal P_2(\mathbb R^d)$ under $W_2$.
	
	\begin{lemma}
		\label{lem:measurable-coefficient-supremum}
		Let $a$ and $\widetilde a$ be two coefficient triples satisfying
		\textnormal{(A1)} and \textnormal{(A3)}. Define
		\begin{equation}
			\label{eq:dense-supremum-section4}
			\Delta_{\mathbb Q}(a,\widetilde a)(t,x)
			:=
			\sup_{j\geq1}
			\left|
			\mathbf a(t,x,\mu_j)
			-
			\widetilde{\mathbf a}(t,x,\mu_j)
			\right|_{\mathbb V}.
		\end{equation}
		Then $\Delta_{\mathbb Q}(a,\widetilde a)$ is Borel measurable and,
		outside a $dt\,dx$-null set,
		\begin{equation}
			\label{eq:full-supremum-section4}
			\Delta_{\mathbb Q}(a,\widetilde a)(t,x)
			=
			\sup_{\mu\in\mathcal P_2(\mathbb R^d)}
			\left|
			\mathbf a(t,x,\mu)
			-
			\widetilde{\mathbf a}(t,x,\mu)
			\right|_{\mathbb V}.
		\end{equation}
		Consequently, the metric $d_T$ does not depend on the particular
		choice of the countable dense family
		$\{\mu_j\}_{j\geq1}$.
	\end{lemma}
	
	\begin{proof}
		For each fixed \(j\), the function
		\[
		(t,x)\longmapsto
		\left|
		\mathbf a(t,x,\mu_j)
		-
		\widetilde{\mathbf a}(t,x,\mu_j)
		\right|_{\mathbb V}
		\]
		is Borel measurable. Hence
		\(\Delta_{\mathbb Q}(a,\widetilde a)\), being the supremum of a
		countable family of Borel functions, is Borel measurable.
		
		For almost every \((t,x)\), assumption \textnormal{(A3)} implies
		that
		\[
		\mu\longmapsto\mathbf a(t,x,\mu),
		\qquad
		\mu\longmapsto\widetilde{\mathbf a}(t,x,\mu)
		\]
		are Lipschitz continuous from
		\((\mathcal P_2(\mathbb R^d),W_2)\) into \(\mathbb V\). Therefore
		\[
		\mu\longmapsto
		\left|
		\mathbf a(t,x,\mu)
		-
		\widetilde{\mathbf a}(t,x,\mu)
		\right|_{\mathbb V}
		\]
		is continuous. Since \(\{\mu_j:j\geq1\}\) is dense in
		\((\mathcal P_2(\mathbb R^d),W_2)\), its supremum over
		\(\{\mu_j\}\) agrees with its supremum over all of
		\(\mathcal P_2(\mathbb R^d)\).
	\end{proof}
	
	In view of Lemma~\ref{lem:measurable-coefficient-supremum}, we
	henceforth write simply $ \Delta(a,\widetilde a)$ for either side of \eqref{eq:full-supremum-section4}.

	The following elementary fact will be used to recover a jointly
	measurable jump coefficient from its $H$-valued version.

	\begin{lemma}
		\label{lem:measurable-H-representative}

		Let \((S,\mathcal S)\) be a measurable space, and let
		\((U,\mathcal U,\nu)\) be a measure space such that
		\[
		H:=L^2(U,\mathcal U,\nu;\mathbb R^d)
		\]
		is separable. If
		\[
		\Phi:(S,\mathcal S)\longrightarrow(H,\mathcal B(H))
		\]
		is measurable, then there exists an
		\(\mathcal S\otimes\mathcal U\)-measurable function
		\[
		\phi:S\times U\longrightarrow\mathbb R^d
		\]
		such that
		\[
		[\phi(s,\cdot)]_H=\Phi(s),
		\qquad s\in S.
		\]
		
	\end{lemma}
	\begin{proof}
		Since \(H\) is separable, choose a dense sequence
		\[
		(h_j)_{j\geq1}\subset H,
		\]
		and for each \(j\) fix an \(\mathcal U\)-measurable representative
		\[
		\widehat h_j:U\longrightarrow\mathbb R^d
		\]
		of \(h_j\).
		
		For \(n\geq1\) and \(s\in S\), define
		\[
		k_n(s)
		:=
		\min\left\{
		j\geq1:
		\|\Phi(s)-h_j\|_H<2^{-n}
		\right\}.
		\]
		By density, \(k_n(s)\) is well defined. Moreover, \(k_n\) is
		\(\mathcal S\)-measurable, since
		\[
		\{k_n=j\}
		=
		\{\|\Phi-h_j\|_H<2^{-n}\}
		\cap
		\cap_{i<j}
		\{\|\Phi-h_i\|_H\geq2^{-n}\}.
		\]
		Set
		\[
		\Phi_n(s):=h_{k_n(s)},
		\qquad
		\phi_n(s,z):=\widehat h_{k_n(s)}(z).
		\]
		Then \(\phi_n\) is
		\(\mathcal S\otimes\mathcal U\)-measurable,
		\[
		[\phi_n(s,\cdot)]_H=\Phi_n(s),
		\]
		and
		\[
		\|\Phi_n(s)-\Phi(s)\|_H<2^{-n}.
		\]
		Hence, for every \(s\in S\),
		\[
		\sum_{n=1}^\infty
		\|\Phi_{n+1}(s)-\Phi_n(s)\|_H
		<\infty.
		\]
		
		Let
		\[
		d_n(s,z):=\phi_{n+1}(s,z)-\phi_n(s,z).
		\]
		For fixed \(s\), Minkowski's inequality gives
		\[
		\big\|
		\sum_{n=1}^N |d_n(s,\cdot)|
		\big\|_{L^2(U,\nu)}
		\leq
		\sum_{n=1}^N
		\|\Phi_{n+1}(s)-\Phi_n(s)\|_H.
		\]
		The right-hand side is bounded uniformly in \(N\). Therefore,
		by monotone convergence,
		\[
		\sum_{n=1}^\infty |d_n(s,z)|<\infty
		\qquad
		\text{for \(\nu\)-almost every }z.
		\]
		
		Define
		\[
		D
		:=
		\left\{
		(s,z)\in S\times U:
		\sum_{n=1}^\infty |d_n(s,z)|<\infty
		\right\}.
		\]
		Since each \(d_n\) is
		\(\mathcal S\otimes\mathcal U\)-measurable, \(D\) is measurable.
		Now set
		\[
		\phi(s,z)
		:=
		\mathbf1_D(s,z)
		\left(
		\phi_1(s,z)
		+
		\sum_{n=1}^\infty d_n(s,z)
		\right).
		\]
		Then \(\phi\) is
		\(\mathcal S\otimes\mathcal U\)-measurable.
		
		For every fixed \(s\), one has
		\[
		\phi_n(s,z)\longrightarrow\phi(s,z)
		\qquad
		\text{for \(\nu\)-almost every }z.
		\]
		On the other hand,
		\[
		\|\Phi_n(s)-\Phi(s)\|_H\longrightarrow0.
		\]
		Since \([\phi_n(s,\cdot)]_H=\Phi_n(s)\), the \(L^2\)-limit is unique,
		and therefore
		\[
		[\phi(s,\cdot)]_H=\Phi(s),
		\qquad s\in S.
		\]
		This proves the assertion.
	\end{proof}

	We now prove the representation statement announced in
	Proposition~\ref{prop:coefficient-representation}.

	\begin{theorem}
		\label{thm:completion-representation}
		Let
		\[
		(a_n)_{n\geq 1}\subset \mathfrak A_T^{\mathrm{Lip}}
		\]
		be a $d_T$-Cauchy sequence. Then there exists a coefficient triple
		\[
		a=(b,\sigma,\overline\beta),
		\qquad
		\overline\beta:
		[0,T]\times\mathbb R^d\times
		\mathcal P_2(\mathbb R^d)\longrightarrow H,
		\]
		such that the following assertions hold.
		
		\begin{enumerate}
			\item[\textnormal{(i)}]
			The map
			\[
			\mathbf a:=(b,\sigma,\overline\beta):
			[0,T]\times\mathbb R^d\times
			\mathcal P_2(\mathbb R^d)
			\longrightarrow\mathbb V
			\]
			is jointly Borel measurable. Moreover, $\overline\beta$ admits an
			$\mathcal B([0,T])\otimes\mathcal B(\mathbb R^d)
			\otimes\mathcal B(\mathcal P_2(\mathbb R^d))
			\otimes\mathcal U$-measurable representative
			\[
			\beta:
			[0,T]\times\mathbb R^d\times
			\mathcal P_2(\mathbb R^d)\times U
			\longrightarrow\mathbb R^d
			\]
			satisfying
			\[
			[\beta(t,x,\mu,\cdot)]_H
			=
			\overline\beta(t,x,\mu).
			\]
			
			\item[\textnormal{(ii)}]
			For every $(t,x,\mu)$,
			\begin{equation}
				\label{eq:completion-uniform-bound}
				|b(t,x,\mu)|
				+
				\|\sigma(t,x,\mu)\|_{\mathrm{HS}}
				+
				\|\overline\beta(t,x,\mu)\|_H
				\leq M.
			\end{equation}
			
			\item[\textnormal{(iii)}]
			For every $(t,x,\mu,\eta)$,
			\begin{equation}
				\label{eq:completion-measure-lipschitz}
				\left|
				\mathbf a(t,x,\mu)-\mathbf a(t,x,\eta)
				\right|_{\mathbb V}
				\leq
				L_\mu W_2(\mu,\eta).
			\end{equation}
			
			\item[\textnormal{(iv)}]
			For every $(t,x,\mu)$,
			\begin{equation}
				\label{eq:completion-ellipticity}
				\frac12
				\sigma(t,x,\mu)\sigma(t,x,\mu)^\top
				\geq
				\lambda I_d.
			\end{equation}
			
			\item[\textnormal{(v)}]
			The full sequence converges to $a$:
			\begin{equation}
				\label{eq:completion-metric-convergence}
				d_T(a_n,a)\longrightarrow 0.
			\end{equation}
		\end{enumerate}
		
		The limit is unique modulo the equivalence relation
		\[
		a\sim\widetilde a
		\quad\Longleftrightarrow\quad
		d_T(a,\widetilde a)=0.
		\]
	\end{theorem}
	\begin{proof}
		Set
		\[
		q_\ast:=2(d+1),
		\qquad
		\sigma_\ast:=\sqrt{2\lambda}\,I_d,
		\qquad
		\mathbf a_\ast:=(0,\sigma_\ast,0).
		\]
		By the compatibility condition
		\[
		M>\sqrt{2\lambda d},
		\]
		the reference coefficient \(\mathbf a_\ast\) satisfies both the
		uniform bound and the ellipticity condition.
		
		Changing the representatives of the \(a_n\) on a common Borel
		\(dt\,dx\)-null set, if necessary, does not affect any of the
		quantities \(\rho_{T,m}\) or \(d_T\). We may therefore assume
		throughout the proof that the uniform bound, the measure-Lipschitz
		condition, and ellipticity hold for every \((t,x,\mu)\).

		\textbf{Step 1. Construction of pointwise limits on a dense family of laws.}
		
		Let \((\mu_j)_{j\geq1}\) be the fixed dense sequence in
		\(\mathcal P_2(\mathbb R^d)\) used in the definition of \(d_T\).
		Since \((a_n)\) is \(d_T\)-Cauchy, it is
		\(\rho_{T,m}\)-Cauchy for every \(m\). Hence we may choose an
		increasing sequence \((n_k)_{k\geq1}\) such that
		\begin{equation}
			\label{eq:rapid-cauchy-subsequence}
			\rho_{T,k}(a_{n_{k+1}},a_{n_k})
			\leq
			2^{-k},
			\qquad k\geq1.
		\end{equation}
		
		Define
		\[
		\delta_k(t,x)
		:=
		\sup_{j\geq1}
		\left|
		\mathbf a_{n_{k+1}}(t,x,\mu_j)
		-
		\mathbf a_{n_k}(t,x,\mu_j)
		\right|_{\mathbb V}.
		\]
		Then \(\delta_k\) is Borel measurable and
		\[
		\|\delta_k\|_{L^{q_\ast}(Q_{T,k})}
		\leq
		2^{-k}.
		\]
		Thus, for every fixed \(m\) and every \(k\geq m\),
		\[
		\|\delta_k\|_{L^{q_\ast}(Q_{T,m})}
		\leq
		2^{-k}.
		\]
		By Minkowski's inequality,
		\[
		\big\|
		\sum_{k=m}^{N}\delta_k
		\big\|_{L^{q_\ast}(Q_{T,m})}
		\leq
		\sum_{k=m}^{N}2^{-k}
		\leq
		2^{1-m}.
		\]
		Letting \(N\to\infty\) and using monotone convergence yields
		\[
		\sum_{k=m}^{\infty}\delta_k(t,x)
		<\infty
		\]
		for almost every \((t,x)\in Q_{T,m}\).
		
		Taking the union of the exceptional sets over \(m\), and then a
		Borel hull, we obtain a Borel \(dt\,dx\)-null set
		\[
		\mathcal N
		\subset
		[0,T]\times\mathbb R^d
		\]
		such that, for every \((t,x)\notin\mathcal N\) and every \(j\),
		the sequence
		\[
		\bigl(\mathbf a_{n_k}(t,x,\mu_j)\bigr)_{k\geq1}
		\]
		is Cauchy in \(\mathbb V\). Define
		\[
		\mathbf a^j(t,x)
		:=
		\begin{cases}
			\displaystyle
			\lim_{k\to\infty}
			\mathbf a_{n_k}(t,x,\mu_j),
			&(t,x)\notin\mathcal N,\\[1mm]
			\mathbf a_\ast,
			&(t,x)\in\mathcal N.
		\end{cases}
		\]
		Each \(\mathbf a^j\) is Borel measurable. Moreover, for every
		\(\ell\geq m\),
		\begin{equation}
			\label{eq:uniform-dense-tail}
			\sup_{j\geq1}
			\left|
			\mathbf a_{n_\ell}(t,x,\mu_j)
			-
			\mathbf a^j(t,x)
			\right|_{\mathbb V}
			\leq
			\sum_{k=\ell}^{\infty}\delta_k(t,x)
		\end{equation}
		for almost every \((t,x)\in Q_{T,m}\).
		
		For every \(k\) and all \(i,j\geq1\), assumption
		\textnormal{(A3)} gives
		\[
		\left|
		\mathbf a_{n_k}(t,x,\mu_i)
		-
		\mathbf a_{n_k}(t,x,\mu_j)
		\right|_{\mathbb V}
		\leq
		L_\mu W_2(\mu_i,\mu_j).
		\]
		For \((t,x)\notin\mathcal N\), letting \(k\to\infty\) yields
		\begin{equation}
			\label{eq:dense-limit-lipschitz}
			\left|
			\mathbf a^i(t,x)-\mathbf a^j(t,x)
			\right|_{\mathbb V}
			\leq
			L_\mu W_2(\mu_i,\mu_j).
		\end{equation}

		\textbf{Step 2. Extension to all probability measures and joint measurability.}
		
		Fix \((t,x)\notin\mathcal N\) and
		\(\mu\in\mathcal P_2(\mathbb R^d)\). Choose a sequence
		\((\mu_{j_r})_{r\geq1}\) from the dense family such that
		\[
		W_2(\mu_{j_r},\mu)\longrightarrow0.
		\]
		By \eqref{eq:dense-limit-lipschitz},
		\(\bigl(\mathbf a^{j_r}(t,x)\bigr)_{r\geq1}\) is Cauchy in
		\(\mathbb V\). Define
		\[
		\mathbf a(t,x,\mu)
		:=
		\lim_{r\to\infty}
		\mathbf a^{j_r}(t,x).
		\]
		
		This definition is independent of the approximating sequence.
		Indeed, if
		\[
		\mu_{j_r}\longrightarrow\mu,
		\qquad
		\mu_{i_r}\longrightarrow\mu
		\]
		in \(W_2\), then
		\[
		\begin{aligned}
			\left|
			\mathbf a^{j_r}(t,x)
			-
			\mathbf a^{i_r}(t,x)
			\right|_{\mathbb V}
			&\leq
			L_\mu W_2(\mu_{j_r},\mu_{i_r})
			\\
			&\leq
			L_\mu
			\bigl(
			W_2(\mu_{j_r},\mu)
			+
			W_2(\mu,\mu_{i_r})
			\bigr)
			\longrightarrow0.
		\end{aligned}
		\]
		Thus \(\mathbf a(t,x,\mu)\) is well defined.
		
		Similarly, if
		\(\mu_{j_r}\to\mu\) and \(\mu_{i_r}\to\eta\), then
		\eqref{eq:dense-limit-lipschitz} gives
		\[
		\left|
		\mathbf a^{j_r}(t,x)
		-
		\mathbf a^{i_r}(t,x)
		\right|_{\mathbb V}
		\leq
		L_\mu W_2(\mu_{j_r},\mu_{i_r}).
		\]
		Letting \(r\to\infty\) yields
		\begin{equation}
			\label{eq:completion-measure-lipschitz}
			\left|
			\mathbf a(t,x,\mu)
			-
			\mathbf a(t,x,\eta)
			\right|_{\mathbb V}
			\leq
			L_\mu W_2(\mu,\eta).
		\end{equation}
		On \(\mathcal N\), define
		\[
		\mathbf a(t,x,\mu)
		:=
		\mathbf a_\ast
		\qquad
		\text{for every }\mu.
		\]
		Then \eqref{eq:completion-measure-lipschitz} also holds on
		\(\mathcal N\).
		
		It remains to verify joint measurability. For \(r\geq1\), define
		\[
		j_r(\mu)
		:=
		\min
		\left\{
		j\geq1:
		W_2(\mu,\mu_j)<r^{-1}
		\right\}.
		\]
		The map
		\[
		j_r:
		\mathcal P_2(\mathbb R^d)
		\longrightarrow
		\mathbb N
		\]
		is Borel measurable, since
		\[
		\{j_r=j\}
		=
		\{W_2(\mu,\mu_j)<r^{-1}\}
		\cap
		\cap_{i<j}
		\{W_2(\mu,\mu_i)\geq r^{-1}\}.
		\]
		Define
		\[
		\mathbf a^{(r)}(t,x,\mu)
		:=
		\begin{cases}
			\mathbf a^{j_r(\mu)}(t,x),
			&(t,x)\notin\mathcal N,\\
			\mathbf a_\ast,
			&(t,x)\in\mathcal N.
		\end{cases}
		\]
		Each \(\mathbf a^{(r)}\) is jointly Borel measurable. Moreover,
		by \eqref{eq:completion-measure-lipschitz},
		\[
		\left|
		\mathbf a^{(r)}(t,x,\mu)
		-
		\mathbf a(t,x,\mu)
		\right|_{\mathbb V}
		\leq
		L_\mu r^{-1}.
		\]
		Hence
		\[
		\mathbf a^{(r)}(t,x,\mu)
		\longrightarrow
		\mathbf a(t,x,\mu)
		\]
		pointwise, and therefore \(\mathbf a\) is jointly Borel measurable.
		
		In particular, its \(H\)-valued component
		\[
		\overline\beta:
		[0,T]\times\mathbb R^d\times
		\mathcal P_2(\mathbb R^d)
		\longrightarrow H
		\]
		is Borel measurable. Applying
		Lemma~\ref{lem:measurable-H-representative} with
		\[
		S
		=
		[0,T]\times\mathbb R^d\times
		\mathcal P_2(\mathbb R^d)
		\]
		yields a jointly measurable representative
		\[
		\beta(t,x,\mu,z)
		\]
		satisfying
		\[
		[\beta(t,x,\mu,\cdot)]_H
		=
		\overline\beta(t,x,\mu).
		\]

		\textbf{Step 3. Preservation of the uniform bound and ellipticity.}
		
		For every \(j\) and every
		\((t,x)\notin\mathcal N\), passage to the limit in the uniform
		bound gives
		\[
		|\mathbf a^j(t,x)|_{\mathbb V}
		\leq
		M.
		\]
		Approximating an arbitrary
		\(\mu\in\mathcal P_2(\mathbb R^d)\) by the dense family and using
		the continuity implied by
		\eqref{eq:completion-measure-lipschitz}, we obtain
		\begin{equation}
			\label{eq:completion-uniform-bound}
			|\mathbf a(t,x,\mu)|_{\mathbb V}
			\leq
			M.
		\end{equation}
		On \(\mathcal N\), the same inequality holds because
		\[
		|\mathbf a_\ast|_{\mathbb V}
		=
		\sqrt{2\lambda d}
		<
		M.
		\]
		
		To prove ellipticity, fix
		\((t,x)\notin\mathcal N\) and \(j\geq1\).
		Since
		\[
		\sigma_{n_k}(t,x,\mu_j)
		\longrightarrow
		\sigma(t,x,\mu_j)
		\]
		in Hilbert--Schmidt norm, while
		\[
		\frac12
		\sigma_{n_k}(t,x,\mu_j)
		\sigma_{n_k}(t,x,\mu_j)^\top
		\geq
		\lambda I_d
		\]
		for every \(k\), and since the set
		\[
		\left\{
		S\in\mathbb R^{d\times d}:
		\frac12SS^\top\geq\lambda I_d
		\right\}
		\]
		is closed in the Hilbert--Schmidt topology, we obtain
		\[
		\frac12
		\sigma(t,x,\mu_j)
		\sigma(t,x,\mu_j)^\top
		\geq
		\lambda I_d.
		\]
		Approximating an arbitrary \(\mu\) by the dense family and using
		\eqref{eq:completion-measure-lipschitz} yields
		\begin{equation}
			\label{eq:completion-ellipticity}
			\frac12
			\sigma(t,x,\mu)\sigma(t,x,\mu)^\top
			\geq
			\lambda I_d.
		\end{equation}
		On \(\mathcal N\), the same inequality follows from
		\[
		\frac12
		\sigma_\ast\sigma_\ast^\top
		=
		\lambda I_d.
		\]

		\textbf{Step 4. Convergence in the Krylov-stable metric.}
		
		Fix \(m\geq1\) and \(\ell\geq m\). By
		\eqref{eq:uniform-dense-tail},
		\[
		\begin{aligned}
			\rho_{T,m}(a_{n_\ell},a)
			\leq	\big\|
			\sum_{k=\ell}^{\infty}\delta_k\big\|_{L^{q_\ast}(Q_{T,m})}	
			\leq\sum_{k=\ell}^{\infty}\|\delta_k\|_{L^{q_\ast}(Q_{T,m})}
			\leq\sum_{k=\ell}^{\infty}2^{-k}.
		\end{aligned}
		\]
		Therefore
		\[
		\rho_{T,m}(a_{n_\ell},a)
		\longrightarrow0
		\qquad
		\text{for every }m.
		\]
		By the definition of \(d_T\) and dominated convergence for its
		defining series,
		\[
		d_T(a_{n_\ell},a)
		\longrightarrow0.
		\]
		
		Since the original sequence \((a_n)\) is \(d_T\)-Cauchy, this
		convergence extends to the full sequence. Indeed, given
		\(\varepsilon>0\), choose \(\ell\) sufficiently large so that
		\[
		d_T(a_{n_\ell},a)<\frac{\varepsilon}{2}.
		\]
		Then, since \((a_n)\) is \(d_T\)-Cauchy, for all sufficiently large
		\(n\),
		\[
		d_T(a_n,a_{n_\ell})<\frac{\varepsilon}{2}.
		\]
		Hence
		\[
		d_T(a_n,a)
		\leq
		d_T(a_n,a_{n_\ell})
		+
		d_T(a_{n_\ell},a)
		<
		\varepsilon.
		\]
		Thus
		\begin{equation}
			\label{eq:completion-full-sequence-convergence}
			d_T(a_n,a)
			\longrightarrow0.
		\end{equation}
		
		\textbf{Step 5. Uniqueness of the representative modulo \(d_T\).}
		
		Let \(\widetilde a\) be another admissible coefficient triple such
		that
		\[
		d_T(a_n,\widetilde a)
		\longrightarrow0.
		\]
		By the triangle inequality and
		\eqref{eq:completion-full-sequence-convergence},
		\[
		d_T(a,\widetilde a)=0.
		\]
		Hence, for every \(m\),
		\[
		\rho_{T,m}(a,\widetilde a)=0,
		\]
		and therefore
		\[
		\sup_{j\geq1}
		\left|
		\mathbf a(t,x,\mu_j)
		-
		\widetilde{\mathbf a}(t,x,\mu_j)
		\right|_{\mathbb V}
		=0
		\]
		for almost every \((t,x)\in Q_{T,m}\).
		
		Taking the countable union of the exceptional sets over \(m\), we
		obtain a single \(dt\,dx\)-null set outside which
		\[
		\mathbf a(t,x,\mu_j)
		=
		\widetilde{\mathbf a}(t,x,\mu_j)
		\qquad
		\text{for every }j\geq1.
		\]
		Since both representatives are Lipschitz continuous in the law
		variable with respect to \(W_2\), density of
		\(\{\mu_j:j\geq1\}\) implies
		\[
		\mathbf a(t,x,\mu)
		=
		\widetilde{\mathbf a}(t,x,\mu)
		\]
		for every
		\(\mu\in\mathcal P_2(\mathbb R^d)\) outside that null set.
		Thus the representative \(a\) is unique modulo the equivalence
		relation induced by \(d_T\).
	\end{proof}

	Theorem~\ref{thm:completion-representation} identifies the abstract
	metric completion $\mathfrak A_T$ with a space of measurable
	coefficient triples. We shall use this identification without
	further comment.

	\begin{corollary}
		\label{cor}
		Every element of $\mathfrak A_T$ admits a jointly measurable
		representative satisfying \textnormal{(A1)}--\textnormal{(A4)} at
		every
		\[
		(t,x,\mu)\in
		[0,T]\times\mathbb R^d\times\mathcal P_2(\mathbb R^d).
		\]
		In particular, if $a\in\mathfrak A_T$ and $X$ is an adapted
		c\'adl\'ag process, then, for any fixed  c\'adl\'ag version of $X$,
		\[
		\frac12
		\sigma\bigl(t,X_t(\omega),\mathcal L(X_t)\bigr)
		\sigma\bigl(t,X_t(\omega),\mathcal L(X_t)\bigr)^\top
		\geq
		\lambda I_d
		\]
		for every $(t,\omega)\in[0,T]\times\Omega$.
	\end{corollary}

	\begin{proposition}
		\label{prop:rep-equivalence}
		Let $a$ and $\widetilde a$ be two admissible representatives of the
		same element of $\mathfrak A_T$. Then, for every $x\in\mathbb R^d$,
		\[
		\operatorname{Sol}_T(a,x)
		=
		\operatorname{Sol}_T(\widetilde a,x).
		\]
		Consequently, the strong-solution relation is well defined on
		$\mathfrak A_T$ and is independent of the chosen admissible
		representative.
	\end{proposition}
	
	\begin{proof}
		Let \(X\in\operatorname{Sol}_T(a,x)\) and set
		\[
		\mu_t:=\mathcal L(X_t),
		\qquad
		\tau_R:=\inf\{t\in[0,T]:|X_t|\geq R\}\wedge T,
		\qquad R>0.
		\]
		Since \(a\) and \(\widetilde a\) are representatives of the same element
		of \(\mathfrak A_T\), Theorem~\ref{thm:completion-representation}
		yields a \((dt,dx)\)-null set \(\mathcal N\), independent of
		\(\mu\), such that
		\[
		\mathbf a(t,y,\mu)
		=
		\widetilde{\mathbf a}(t,y,\mu)
		\]
		for every \(\mu\in\mathcal P_2(\mathbb R^d)\) and every
		\((t,y)\notin\mathcal N\).
		
		By the local occupation estimate applied to the non-degenerate
		solution \(X\),
		\[
		\int_0^{\tau_R}
		\mathbf 1_{\mathcal N}(t,X_t)\,dt
		=0
		\qquad\text{almost surely}.
		\]
		Since a c\`adl\`ag path has at most countably many jump times,
		\[
		X_t=X_{t-}
		\qquad
		\text{for }dt\,d\mathbb P\text{-almost every }(t,\omega).
		\]
		Consequently,
		\[
		\mathbf 1_{\{t\leq\tau_R\}}
		\bigl|b(t,X_t,\mu_t)-\widetilde b(t,X_t,\mu_t)\bigr|
		=0,
		\]
		\[
		\mathbf 1_{\{t\leq\tau_R\}}
		\bigl\|\sigma(t,X_t,\mu_t)
		-\widetilde\sigma(t,X_t,\mu_t)\bigr\|_{\mathrm{HS}}
		=0,
		\]
		and
		\[
		\mathbf 1_{\{t\leq\tau_R\}}
		\bigl\|
		\overline\beta(t,X_{t-},\mu_t)
		-\overline{\widetilde\beta}(t,X_{t-},\mu_t)
		\bigr\|_H
		=0
		\]
		for \(dt\,d\mathbb P\)-almost every \((t,\omega)\).
		
		It follows that the stopped drift integrals coincide. Moreover, by
		the It\^o isometry,
		\[
		\int_0^{\,\cdot\wedge\tau_R}
		\sigma(s,X_s,\mu_s)\,dW_s
		=
		\int_0^{\,\cdot\wedge\tau_R}
		\widetilde\sigma(s,X_s,\mu_s)\,dW_s
		\]
		up to indistinguishability.
		
		For the jump terms, the preceding \(H\)-valued equality implies
		\[
		\begin{aligned}
			&\int_0^T
			\mathbf 1_{\{s\leq\tau_R\}}
			\int_U
			\bigl|
			\beta(s,X_{s-},\mu_s,z)
			-\widetilde\beta(s,X_{s-},\mu_s,z)
			\bigr|^2
			\,\nu(dz)\,ds
			\\
			=&
			\int_0^T
			\mathbf 1_{\{s\leq\tau_R\}}
			\bigl\|
			\overline\beta(s,X_{s-},\mu_s)
			-\overline{\widetilde\beta}(s,X_{s-},\mu_s)
			\bigr\|_H^2
			\,ds
			=0
		\end{aligned}
		\]
		almost surely. Hence, by the isometry for compensated Poisson
		integrals,
		\[
		\int_0^{\,\cdot\wedge\tau_R}\int_U
		\beta(s,X_{s-},\mu_s,z)\,\widetilde N(ds,dz)
		=
		\int_0^{\,\cdot\wedge\tau_R}\int_U
		\widetilde\beta(s,X_{s-},\mu_s,z)\,\widetilde N(ds,dz)
		\]
		up to indistinguishability.
		
		Substituting these identities into the equation satisfied by \(X\)
		gives, for every \(t\in[0,T]\),
		\[
		\begin{aligned}
			X_{t\wedge\tau_R}
			={}&x
			+\int_0^{t\wedge\tau_R}
			\widetilde b(s,X_s,\mu_s)\,ds
			\\
			&+\int_0^{t\wedge\tau_R}
			\widetilde\sigma(s,X_s,\mu_s)\,dW_s
			\\
			&+\int_0^{t\wedge\tau_R}\int_U
			\widetilde\beta(s,X_{s-},\mu_s,z)\,
			\widetilde N(ds,dz),
		\end{aligned}
		\]
		where
		\[
		\mu_s=\mathcal L(X_s)
		\]
		is the marginal law of the original, unstopped process.
		
		Finally, take \(R\in\mathbb N\). Outside a single null set, the
		preceding stopped identity holds simultaneously for every
		\(R\in\mathbb N\). Since every c\`adl\`ag path is bounded on
		\([0,T]\), for each such sample path there exists \(R\in\mathbb N\)
		large enough such that
		\[
		\tau_R=T.
		\]
		Hence \(X\) satisfies the full McKean--Vlasov equation with coefficient
		\(\widetilde a\), and therefore
		\[
		X\in\operatorname{Sol}_T(\widetilde a,x).
		\]
		Thus
		\[
		\operatorname{Sol}_T(a,x)
		\subset
		\operatorname{Sol}_T(\widetilde a,x).
		\]
		The reverse inclusion follows by symmetry.
	\end{proof}

	\begin{remark}
		After Proposition~\ref{prop:rep-equivalence}, all
		expressions involving
		\[
		\mathcal E_T(a,x)
		\quad\text{or}\quad
		\operatorname{Sol}_T(a,x),
		\qquad a\in\mathfrak A_T,
		\]
		are understood at the level of the metric-space element rather than
		at the level of a particular coefficient representative.
	\end{remark}
	
	\subsection{Approximation by the Lipschitz Core}

	By definition,
	\[
	\mathfrak A_T
	=
	\overline{\mathfrak A_T^{\mathrm{Lip}}}^{d_T}.
	\]
	Hence, for every $a\in\mathfrak A_T$, there exists
	$(a_n)_{n\geq 1}\subset\mathfrak A_T^{\mathrm{Lip}}$ such that
	\[
	d_T(a_n,a)<2^{-n},
	\qquad n\geq 1.
	\]
	Conversely, an admissible measurable coefficient triple
	\[
	a=(b,\sigma,\overline\beta)
	\]
	represents an element of $\mathfrak A_T$ whenever there exists
	$(a_n)_{n\geq 1}\subset\mathfrak A_T^{\mathrm{Lip}}$ such that, for
	every $m\geq 1$,
	\[
	\|
	\sup_{\mu\in\mathcal P_2(\mathbb R^d)}
	\left|
	\mathbf a_n(\cdot,\cdot,\mu)
	-
	\mathbf a(\cdot,\cdot,\mu)
	\right|_{\mathbb V}
	\|_{L^{q_\ast}(Q_{T,m})}
	\longrightarrow 0.
	\]
	Indeed, this convergence implies
	$\rho_{T,m}(a_n,a)\to 0$ for every $m$, and hence
	$d_T(a_n,a)\to 0$. Uniform ellipticity is preserved because
	\[
	\left\{
	S\in\mathbb R^{d\times d}:
	\frac12SS^\top\geq\lambda I_d
	\right\}
	\]
	is closed in the Hilbert--Schmidt topology.
	
	Thus $\mathfrak A_T$ is not the class of all measurable coefficients
	satisfying \textnormal{(A1)}--\textnormal{(A4)}; membership also
	requires approximation by the Lipschitz core in the Krylov-stable
	metric.

	\subsection{Natural Subclasses of the Krylov-Stable Completion}
	\label{subsec:natural-subclasses-completion}
	
	The definition
	\[
	\mathfrak A_T
	=
	\overline{\mathfrak A_T^{\mathrm{Lip}}}^{\,d_T}
	\]
	is abstract. We next give an intrinsic sufficient condition for membership
	in \(\mathfrak A_T\) and derive several natural subclasses. The condition is
	formulated as uniform local translation continuity, with respect to the law
	variable, in the exponent dictated by the Krylov estimate. 
	
	The formulas defining \(\Delta\), \(\rho_{T,R}\), and \(d_T\) extend
	 to any two coefficient triples satisfying \textnormal{(A1)} and
	\textnormal{(A3)}. In the arguments below, this extended pseudodistance is
	used to identify an external measurable coefficient with the element of the
	metric completion generated by its Lipschitz approximations.
	
	For a coefficient triple
	\[
	a=(b,\sigma,\overline\beta)
	\]
	satisfying \textnormal{(A1)} and \textnormal{(A3)}, let
	\[
	\mathbf a(t,x,\mu)
	:=
	\bigl(
	b(t,x,\mu),
	\sigma(t,x,\mu),
	\overline\beta(t,x,\mu)
	\bigr)
	\in\mathbb V.
	\]
	For \(h\in\mathbb R^d\), define
	\begin{equation}
		\label{eq}
		D_{a,h}(t,x)
		:=
		\sup_{j\geq1}
		\left|
		\mathbf a(t,x+h,\mu_j)
		-
		\mathbf a(t,x,\mu_j)
		\right|_{\mathbb V}.
	\end{equation}
	By \textnormal{(A3)} and the density of
	\(\{\mu_j:j\geq1\}\) in \((\mathcal P_2(\mathbb R^d),W_2)\),
	\begin{equation}
		\label{eq:spatial-translation-error-full-supremum}
		D_{a,h}(t,x)
		=
		\sup_{\mu\in\mathcal P_2(\mathbb R^d)}
		\left|
		\mathbf a(t,x+h,\mu)
		-
		\mathbf a(t,x,\mu)
		\right|_{\mathbb V}.
	\end{equation}
	In particular, \(D_{a,h}\) is Borel measurable. For \(R>0\), set
	\begin{equation}
		\label{eq:spatial-translation-modulus}
		\omega_{a,R}(h)
		:=
		\|D_{a,h}\|_{L^{q_\ast}(Q_{T,R})}.
	\end{equation}
	
	We first record the elementary translation fact used repeatedly below.
	
	\begin{lemma}
		\label{lem:local-bochner-translation}
		Let \(E\) be a separable Banach space, let \(1\leq p<\infty\), and let
		\[
		f\in L^p_{\mathrm{loc}}
		\bigl([0,T]\times\mathbb R^d;E\bigr).
		\]
		Then, for every \(R>0\),
		\begin{equation}
			\label{eq:local-bochner-translation}
			\|f(t,x+h)-f(t,x)\|_{L^p([0,T]\times B_R;E)}
			\longrightarrow0
			\qquad\text{as }h\longrightarrow0.
		\end{equation}
	\end{lemma}
	
	\begin{proof}
		Fix \(R>0\) and choose
		\(\chi\in C_c^\infty(\mathbb R^d)\) such that
		\(\chi=1\) on \(B_{R+1}\). Set
		\[
		g(t,x):=\chi(x)f(t,x).
		\]
		Then
		\[
		g\in L^p([0,T]\times\mathbb R^d;E).
		\]
		For \(|h|<1\), one has, on \([0,T]\times B_R\),
		\[
		f(t,x+h)-f(t,x)=g(t,x+h)-g(t,x).
		\]
		Translations are strongly continuous on the Bochner space
		\(L^p([0,T]\times\mathbb R^d;E)\), see, for example, \cite{Hytonen-2003}. Hence
		\[
		\|g(t,x+h)-g(t,x)\|_{L^p([0,T]\times\mathbb R^d;E)}
		\longrightarrow0,
		\]
		which proves \eqref{eq:local-bochner-translation}.
	\end{proof}
	
	\begin{proposition}
		\label{prop:translation-continuity-completion}
		Let
		\[
		a=(b,\sigma,\overline\beta)
		\]
		satisfy \textnormal{(A1)}--\textnormal{(A3)} with the constants
		\(M\) and \(L_\mu\). Let
		\(\mathcal C_\lambda\subset\mathbb R^{d\times d}\) be a nonempty closed
		convex set such that
		\begin{equation}
			\label{eq:convex-elliptic-matrix-class}
			S\in\mathcal C_\lambda
			\quad\Longrightarrow\quad
			\frac12SS^\top\geq\lambda I_d.
		\end{equation}
		Assume that
		\begin{equation}
			\label{eq:diffusion-in-convex-elliptic-class}
			\sigma(t,x,\mu)\in\mathcal C_\lambda
		\end{equation}
		for every \((t,x,\mu)\), and that, for every \(R>0\),
		\begin{equation}
			\label{eq:translation-continuity-assumption}
			\omega_{a,R}(h)\longrightarrow0
			\qquad\text{as }h\longrightarrow0.
		\end{equation}
		Then \(a\) represents an element of \(\mathfrak A_T\). More precisely,
		there exists
		\[
		a^\varepsilon\in\mathfrak A_T^{\mathrm{Lip}},
		\qquad
		\varepsilon\in(0,1),
		\]
		such that
		\begin{equation}
			\label{eq:mollified-coefficient-convergence}
			d_T(a^\varepsilon,a)\longrightarrow0
			\qquad\text{as }\varepsilon\downarrow0.
		\end{equation}
	\end{proposition}
	
\begin{proof}
	Choose a standard mollifier
	\[
	\rho\in C_c^\infty(\mathbb R^d),
	\qquad
	\rho\geq0,
	\qquad
	\operatorname{supp}\rho\subset B_1,
	\qquad
	\int_{\mathbb R^d}\rho(h)\,dh=1,
	\]
	and set
	\[
	\rho_\varepsilon(h)
	:=
	\varepsilon^{-d}\rho(h/\varepsilon).
	\]
	Define the \(\mathbb V\)-valued spatial convolution
	\begin{equation}
		\label{eq:spatially-mollified-coefficient}
		\mathbf a^\varepsilon(t,x,\mu)
		:=
		\int_{\mathbb R^d}
		\rho_\varepsilon(h)
		\mathbf a(t,x-h,\mu)\,dh.
	\end{equation}
	Since \(\mathbb V\) is separable and \(\mathbf a\) is bounded and
	jointly Borel measurable, the Bochner integral in
	\eqref{eq:spatially-mollified-coefficient} is well defined and
	\(\mathbf a^\varepsilon\) is jointly Borel measurable. Write
	\[
	\mathbf a^\varepsilon
	=
	\bigl(
	b^\varepsilon,
	\sigma^\varepsilon,
	\overline\beta^\varepsilon
	\bigr).
	\]
	Applying Lemma~\ref{lem:measurable-H-representative} to
	\(\overline\beta^\varepsilon\) yields a jointly measurable pointwise
	representative. Thus \textnormal{(A1)} holds.
	
	By the triangle inequality for Bochner integrals and
	\textnormal{(A2)},
	\begin{equation}
		\label{eq:mollification-preserves-uniform-bound}
		\left|
		\mathbf a^\varepsilon(t,x,\mu)
		\right|_{\mathbb V}
		\leq
		\int_{\mathbb R^d}
		\rho_\varepsilon(h)
		\left|
		\mathbf a(t,x-h,\mu)
		\right|_{\mathbb V}\,dh
		\leq
		M.
	\end{equation}
	Similarly, by \textnormal{(A3)},
	\begin{equation}
		\label{eq:mollification-preserves-law-lipschitz}
		\left|
		\mathbf a^\varepsilon(t,x,\mu)
		-
		\mathbf a^\varepsilon(t,x,\eta)
		\right|_{\mathbb V}
		\leq
		L_\mu W_2(\mu,\eta).
	\end{equation}
	
	Since \(\mathcal C_\lambda\) is closed and convex, the Bochner average
	of an \(\mathcal C_\lambda\)-valued function again belongs to
	\(\mathcal C_\lambda\). Therefore
	\[
	\sigma^\varepsilon(t,x,\mu)
	\in
	\mathcal C_\lambda,
	\]
	and \eqref{eq:convex-elliptic-matrix-class} yields
	\begin{equation}
		\label{eq:mollification-preserves-ellipticity}
		\frac12
		\sigma^\varepsilon(t,x,\mu)
		\sigma^\varepsilon(t,x,\mu)^\top
		\geq
		\lambda I_d.
	\end{equation}
	
	To prove spatial Lipschitz continuity, rewrite
	\eqref{eq:spatially-mollified-coefficient} as
	\[
	\mathbf a^\varepsilon(t,x,\mu)
	=
	\int_{\mathbb R^d}
	\rho_\varepsilon(x-z)
	\mathbf a(t,z,\mu)\,dz.
	\]
	Then, for \(x,y\in\mathbb R^d\),
	\begin{align}
		\left|
		\mathbf a^\varepsilon(t,x,\mu)
		-
		\mathbf a^\varepsilon(t,y,\mu)
		\right|_{\mathbb V}
		&\leq
		M
		\int_{\mathbb R^d}
		\left|
		\rho_\varepsilon(x-z)
		-
		\rho_\varepsilon(y-z)
		\right|\,dz
		\nonumber\\
		&\leq
		M
		\|\nabla\rho_\varepsilon\|_{L^1}
		|x-y|
		\nonumber\\
		&=
		M\varepsilon^{-1}
		\|\nabla\rho\|_{L^1}
		|x-y|.
		\label{eq:mollified-spatial-lipschitz}
	\end{align}
	Hence
	\[
	a^\varepsilon
	\in
	\mathfrak A_T^{\mathrm{Lip}}.
	\]
	
	It remains to prove convergence. For every \((t,x)\),
	\begin{align}
		\Delta(a^\varepsilon,a)(t,x)
		&=
		\sup_{\mu\in\mathcal P_2(\mathbb R^d)}
		\left|
		\int_{\mathbb R^d}
		\rho_\varepsilon(h)
		\bigl[
		\mathbf a(t,x-h,\mu)
		-
		\mathbf a(t,x,\mu)
		\bigr]\,dh
		\right|_{\mathbb V}
		\nonumber\\
		&\leq
		\int_{\mathbb R^d}
		\rho_\varepsilon(h)
		D_{a,-h}(t,x)\,dh.
		\label{eq:mollified-error-pointwise}
	\end{align}
	Minkowski's integral inequality therefore gives
	\begin{align}
		\rho_{T,R}(a^\varepsilon,a)
		&\leq
		\int_{\mathbb R^d}
		\rho_\varepsilon(h)
		\omega_{a,R}(-h)\,dh
		\nonumber\\
		&=
		\int_{B_1}
		\rho(u)
		\omega_{a,R}(-\varepsilon u)\,du.
		\label{eq:mollified-error-minkowski}
	\end{align}
	By \textnormal{(A2)},
	\[
	0
	\leq
	\omega_{a,R}(h)
	\leq
	2M(T|B_R|)^{1/q_\ast}.
	\]
	Hence
	\eqref{eq:translation-continuity-assumption} and dominated
	convergence imply
	\[
	\rho_{T,R}(a^\varepsilon,a)
	\longrightarrow0
	\qquad
	\text{for every }R>0
	\]
	as \(\varepsilon\downarrow0\).
	Dominated convergence in the defining series for \(d_T\) then yields
	\begin{equation}
		\label{eq:mollified-coefficient-convergence}
		d_T(a^\varepsilon,a)
		\longrightarrow0
		\qquad
		\text{as }\varepsilon\downarrow0,
	\end{equation}
	where \(d_T(a^\varepsilon,a)\) is understood in the extended
	pseudodistance between admissible measurable coefficient triples.
	
	In particular, for any sequence
	\(\varepsilon_n\downarrow0\),
	\((a^{\varepsilon_n})_{n\geq1}\) is \(d_T\)-Cauchy in
	\(\mathfrak A_T^{\mathrm{Lip}}\). Indeed,
	\[
	d_T(a^{\varepsilon_n},a^{\varepsilon_m})
	\leq
	d_T(a^{\varepsilon_n},a)
	+
	d_T(a,a^{\varepsilon_m})
	\longrightarrow0
	\qquad
	\text{as }n,m\to\infty.
	\]
	Thus \((a^{\varepsilon_n})_{n\geq1}\) defines an element of the
	completion \(\mathfrak A_T\). Since
	\[
	d_T(a^{\varepsilon_n},a)
	\longrightarrow0,
	\]
	the measurable coefficient triple \(a\) is an admissible
	representative of this element. In particular,
	\(a\) represents an element of \(\mathfrak A_T\).
\end{proof}
	
\begin{remark}
	\label{rem:nonconvex-general-diffusion}
	The closed convex-set assumption ensures that spatial convolution
	preserves ellipticity. A basic example is
	\[
	\mathcal C_\lambda^+
	:=
	\left\{
	S\in\mathbb S^d:
	S\geq\sqrt{2\lambda}\,I_d
	\right\},
	\]
	but Proposition~\ref{prop:translation-continuity-completion} applies
	to any closed convex subset of
	\[
	\left\{
	S\in\mathbb R^{d\times d}:
	\frac12SS^\top\geq\lambda I_d
	\right\}.
	\]
	Since the full uniformly elliptic set is not convex, convolution of
	a general matrix-valued diffusion need not preserve
	\textnormal{(A4)}. One may instead assume directly that the
	mollified coefficients satisfy \textnormal{(A4)}, but this condition
	is less intrinsic.
\end{remark}
	
	We next show that coefficients may be pasted across completely arbitrary
	fixed Borel interfaces. No finite-perimeter assumption is needed for
	membership in the completion; finite perimeter only provides a quantitative
	rate.
	
	\begin{proposition}
		\label{prop:finite-patching-completion}
		Let
		\[
		\mathbb R^d
		=
		D_1\mathbin{\dot\cup}\cdots\mathbin{\dot\cup}D_N
		\]
		be a finite Borel partition. For every \(i\), let
		\[
		a_i=(b_i,\sigma_i,\overline\beta_i)
		\]
		satisfy \textnormal{(A1)}--\textnormal{(A3)} with the same constants
		\((M,L_\mu)\), assume
		\[
		\sigma_i(t,x,\mu)\in\mathcal C_\lambda,
		\]
		where \(\mathcal C_\lambda\) is as in
		Proposition~\ref{prop:translation-continuity-completion}, and suppose that
		\begin{equation}
			\label{eq:branch-translation-continuity}
			\omega_{a_i,R}(h)\longrightarrow0
			\qquad\text{as }h\longrightarrow0
		\end{equation}
		for every \(R>0\) and every \(i\). Define
		\begin{equation}
			\label{eq:piecewise-lipschitz-coefficient}
			\mathbf a(t,x,\mu)
			:=
			\sum_{i=1}^N
			\mathbf1_{D_i}(x)\mathbf a_i(t,x,\mu).
		\end{equation}
		Then \(a\) represents an element of \(\mathfrak A_T\).
	\end{proposition}
	
	\begin{proof}
		Since the partition is finite and Borel, the resulting piecewise-defined
		function is jointly Borel measurable.  Since exactly one
		branch is selected at every \(x\), conditions \textnormal{(A2)} and
		\textnormal{(A3)} hold with the same constants, and
		\[
		\sigma(t,x,\mu)\in\mathcal C_\lambda.
		\]
		
		For \(h\in\mathbb R^d\), set
		\[
		I_h(x)
		:=
		\sum_{i=1}^N
		|\mathbf1_{D_i}(x+h)-\mathbf1_{D_i}(x)|.
		\]
		A direct decomposition gives
		\begin{equation}
			\label{eq:finite-patching-translation-bound}
			D_{a,h}(t,x)
			\leq
			\sum_{i=1}^N
			\mathbf1_{D_i}(x+h)D_{a_i,h}(t,x)
			+
			M I_h(x).
		\end{equation}
		Consequently,
		\begin{align}
			\omega_{a,R}(h)
			&\leq
			\sum_{i=1}^N\omega_{a_i,R}(h)
			+
			MT^{1/q_\ast}
			\sum_{i=1}^N
			\|\mathbf1_{D_i}(\cdot+h)-\mathbf1_{D_i}\|_{L^{q_\ast}(B_R)}.
			\label{eq:finite-patching-modulus}
		\end{align}
		For every Borel set \(D_i\), the indicator
		\(\mathbf1_{D_i}\) belongs to \(L^{q_\ast}_{\mathrm{loc}}(\mathbb R^d)\).
		Lemma~\ref{lem:local-bochner-translation}, applied with \(E=\mathbb R\),
		therefore gives
		\[
		\|\mathbf1_{D_i}(\cdot+h)-\mathbf1_{D_i}\|_{L^{q_\ast}(B_R)}
		\longrightarrow0.
		\]
		Together with \eqref{eq:branch-translation-continuity}, this proves
		\(\omega_{a,R}(h)\to0\). The assertion follows from
		Proposition~\ref{prop:translation-continuity-completion}.
	\end{proof}
	
	\begin{corollary}
		\label{cor:piecewise-lipschitz-completion}
		Let \(D_1,\ldots,D_N\) be an arbitrary finite Borel partition of
		\(\mathbb R^d\). For each \(i\), let
		\[
		a_i\in\mathfrak A_T^{\mathrm{Lip}}
		\]
		have the common structural constants \((M,L_\mu,\lambda)\), and assume
		that the diffusion matrices of all branches take values in one fixed closed
		convex set \(\mathcal C_\lambda\) satisfying
		\eqref{eq:convex-elliptic-matrix-class}. Then the pasted coefficient
		\[
		\mathbf a(t,x,\mu)
		=
		\sum_{i=1}^N
		\mathbf1_{D_i}(x)\mathbf a_i(t,x,\mu)
		\]
		belongs to \(\mathfrak A_T\).
	\end{corollary}
	
	\begin{proof}
		If \(L_i=L_x(a_i)\), then
		\[
		\omega_{a_i,R}(h)
		\leq
		L_i|h|(T|B_R|)^{1/q_\ast}.
		\]
		Hence Proposition~\ref{prop:finite-patching-completion} applies.
	\end{proof}
	
	\begin{remark}
		\label{rem:finite-perimeter-rate}
		If, in addition, every \(D_i\) has locally finite perimeter, write
		\(P(D_i;B)\) for the perimeter of \(D_i\) relative to the Borel set \(B\).
		Then, for
		\(|h|\leq1\),
		\[
		\int_{B_R}
		|\mathbf1_{D_i}(x+h)-\mathbf1_{D_i}(x)|\,dx
		\leq
		C_d|h|P(D_i;B_{R+1}).
		\]
		Since the indicator difference takes values in \(\{0,1\}\), this gives an
		explicit \(O(|h|^{1/q_\ast})\) contribution to
		\(\omega_{a,R}(h)\). Thus finite perimeter is useful for rates, but is not
		required for membership in \(\mathfrak A_T\).
	\end{remark}
	
The next result shows that the bounded-variation (BV) assumption suggested by the preceding examples is also unnecessary when the dependence on the law is through finitely many Lipschitz functionals.

		Let \(F_0\equiv1\), and let
	\[
	F_j:\mathcal P_2(\mathbb R^d)\longrightarrow\mathbb R,
	\qquad j=1,\ldots,N,
	\]
	be bounded \(W_2\)-Lipschitz functions. Set
	\[
	K_j:=\|F_j\|_\infty,
	\qquad
	\ell_j:=\sup_{\mu\neq\eta}	\frac{|F_j(\mu)-F_j(\eta)|}	{W_2(\mu,\eta)},
	\qquad
	K_0:=1.
	\]
	For \(j=0,\ldots,N\), let
	\[
	\mathbf v_j:[0,T]\times\mathbb R^d\longrightarrow\mathbb V
	\]
	be jointly Borel measurable and satisfy
	\begin{equation}
		\label{eq:finite-rank-profile-bound}
		\sup_{(t,x)}|\mathbf v_j(t,x)|_{\mathbb V}
		\leq M_j<\infty.
	\end{equation}
	Define
	\begin{equation}
		\label{eq:finite-rank-law-coefficient}
		\mathbf a(t,x,\mu):=\sum_{j=0}^N F_j(\mu)\mathbf v_j(t,x).
	\end{equation}
	
	\begin{proposition}
		\label{prop:finite-rank-measurable-completion}
		Assume that
		\begin{equation}
			\label{eq:finite-rank-structural-bounds}
			\sum_{j=0}^N K_jM_j\leq M,
			\qquad
			\sum_{j=1}^N\ell_jM_j\leq L_\mu,
		\end{equation}
		and that the diffusion component of \(\mathbf a\) takes values in a fixed
		closed convex set \(\mathcal C_\lambda\) satisfying
		\eqref{eq:convex-elliptic-matrix-class}. Then the coefficient triple  \(a=(b,\sigma,\overline\beta)\) represents an element
		of \(\mathfrak A_T\).
	\end{proposition}
	
	\begin{proof}
		The map \(\mathbf a\) is jointly Borel measurable, and its \(H\)-valued
		component admits a jointly measurable pointwise representative by
		Lemma~\ref{lem:measurable-H-representative}. The first inequality in
		\eqref{eq:finite-rank-structural-bounds} gives \textnormal{(A2)}, while
		\begin{align*}
			|\mathbf a(t,x,\mu)-\mathbf a(t,x,\eta)|_{\mathbb V}
			&\leq
			\sum_{j=1}^N
			|F_j(\mu)-F_j(\eta)|\,|\mathbf v_j(t,x)|_{\mathbb V}
			\\
			&\leq
			\left(\sum_{j=1}^N\ell_jM_j\right)W_2(\mu,\eta)
			\leq
			L_\mu W_2(\mu,\eta).
		\end{align*}
		Thus \textnormal{(A3)} holds, and the convex matrix-class assumption gives
		\textnormal{(A4)}.
		
		For every \(h\in\mathbb R^d\),
		\begin{equation}
			\label{eq:finite-rank-translation-pointwise}
			D_{a,h}(t,x)
			\leq
			\sum_{j=0}^N
			K_j
			|\mathbf v_j(t,x+h)-\mathbf v_j(t,x)|_{\mathbb V}.
		\end{equation}
		Each bounded measurable profile \(\mathbf v_j\) belongs to
		\(L^{q_\ast}_{\mathrm{loc}}([0,T]\times\mathbb R^d;\mathbb V)\).
		Lemma~\ref{lem:local-bochner-translation} and Minkowski's inequality imply
		\[
		\omega_{a,R}(h)\longrightarrow0
		\qquad\text{for every }R>0.
		\]
		Proposition~\ref{prop:translation-continuity-completion} completes the
		proof.
	\end{proof}

	\begin{corollary}
		\label{cor:finite-rank-bv-rate}
		Under the assumptions of
		Proposition~\ref{prop:finite-rank-measurable-completion}, suppose in addition
		that, for almost every \(t\),
		\[
		x\longmapsto\mathbf v_j(t,x)
		\in BV_{\mathrm{loc}}(\mathbb R^d;\mathbb V),
		\]
		where Banach-valued BV is understood in the distributional sense and
		\(|D_x\mathbf v_j(t,\cdot)|\) denotes the total-variation measure,
		and that, for every \(R>0\),
		\[
		V_{j,R}
		:=
		\int_0^T
		|D_x\mathbf v_j(t,\cdot)|(B_R)\,dt
		<\infty.
		\]
		Then, for \(|h|\leq1\),
		\begin{equation}
			\label{eq:finite-rank-bv-translation-bound}
			\omega_{a,R}(h)^{q_\ast}
			\leq
			C_d(N+1)^{q_\ast-1}|h|
			\sum_{j=0}^N
			K_j^{q_\ast}(2M_j)^{q_\ast-1}V_{j,R+1}.
		\end{equation}
	\end{corollary}
	
	\begin{proof}
		From \eqref{eq:finite-rank-translation-pointwise}, convexity gives
		\[
		D_{a,h}^{q_\ast}
		\leq
		(N+1)^{q_\ast-1}
		\sum_{j=0}^N
		K_j^{q_\ast}
		|\mathbf v_j(t,x+h)-\mathbf v_j(t,x)|_{\mathbb V}^{q_\ast}.
		\]
		By \eqref{eq:finite-rank-profile-bound},
		\[
		|\mathbf v_j(t,x+h)-\mathbf v_j(t,x)|_{\mathbb V}^{q_\ast}
		\leq
		(2M_j)^{q_\ast-1}
		|\mathbf v_j(t,x+h)-\mathbf v_j(t,x)|_{\mathbb V}.
		\]
		Since \(\mathbb V\) is a finite product of finite-dimensional spaces and the Hilbert space \(H\), the standard Bochner integration and \(L^p\)-duality results used below apply to \(\mathbb V\)-valued functions.
	For almost every \(t\), the standard local translation estimate for
	\(\mathbb V\)-valued BV maps gives, for \(|h|\leq 1\),
	\[
	\int_{B_R}
	\left|\mathbf v_j(t,x+h)-\mathbf v_j(t,x)\right|_{\mathbb V}\,dx
	\leq
	C_d |h| \, |D_x\mathbf v_j(t,\cdot)|(B_{R+1}).
	\]
	Here \(B_{R+1}\) appears because
	\(x+\theta h\in B_{R+1}\) for \(x\in B_R\), \(0\leq\theta\leq1\).

	Combining this with the preceding two estimates and integrating over
	\([0,T]\times B_R\), we obtain
	\[
	\begin{aligned}
		\omega_{a,R}(h)^{q_\ast}
		&=
		\int_0^T\int_{B_R}D_{a,h}(t,x)^{q_\ast}\,dx\,dt
		\\
		&\leq
		C_d(N+1)^{q_\ast-1}|h|
		\sum_{j=0}^N
		K_j^{q_\ast}(2M_j)^{q_\ast-1}
		\int_0^T
		|D_x\mathbf v_j(t,\cdot)|(B_{R+1})\,dt
		\\
		&=
		C_d(N+1)^{q_\ast-1}|h|
		\sum_{j=0}^N
		K_j^{q_\ast}(2M_j)^{q_\ast-1}V_{j,R+1},
	\end{aligned}
	\]
	which is \eqref{eq:finite-rank-bv-translation-bound}.
	\end{proof}
	
	We finally treat coefficients defined by integration against the law. The
	following measurability fact is included to make the construction explicit.
	
	\begin{lemma}
		\label{lem:parameterized-bochner-integration}
		Let \(S\) be a standard Borel space, let \(Y\) be a Polish space, let
		\(E\) be a separable Banach space, and let
		\[
		G:S\times Y\longrightarrow E
		\]
		be bounded and Borel measurable. Then
		\[
		(s,\mu)
		\longmapsto
		\int_Y G(s,y)\,\mu(dy)
		\]
		is Borel measurable from \(S\times\mathcal P(Y)\) into \(E\), where
		\(\mathcal P(Y)\) carries the Borel \(\sigma\)-field of the weak topology.
	\end{lemma}
	
	\begin{proof}
		For bounded scalar Borel functions, the assertion follows from the monotone
		class theorem, starting with indicators of Borel rectangles. Since \(E\) is
		separable, fix a dense sequence \((e_k)_{k\geq1}\subset E\). Since \(G\)
		is bounded, the approximants below may be chosen with a common uniform
		bound. For each \(n\), choose Borel sets
		\((A_{n,k})_{k\geq1}\) forming a partition of \(S\times Y\) such that the
		countably valued Borel map
		\[
		G_n:=\sum_{k=1}^{\infty}e_k\mathbf1_{A_{n,k}}
		\]
		satisfies \(\|G_n-G\|_E\leq n^{-1}\) pointwise. For each finite partial
		sum, the parameterized integral is Borel measurable by the scalar result.
		Because the selected values of \(G_n\) are uniformly bounded, the partial
		sums of the integral converge in \(E\). Hence
		\[
		(s,\mu)\longmapsto\int_YG_n(s,y)\,\mu(dy)
		\]
		is Borel measurable. Finally,
		\[
		\big\|
		\int_YG_n(s,y)\,\mu(dy)
		-
		\int_YG(s,y)\,\mu(dy)
		\big\|_E
		\leq n^{-1},
		\]
		uniformly in (s). Hence the desired map is the uniform limit of
		Borel measurable \(E\)-valued maps, and is therefore Borel measurable.
	\end{proof}
	
	\begin{corollary}
		\label{cor:integral-interaction-completion}
		Let
		\[
		\mathbf A:
		[0,T]\times\mathbb R^d\times\mathbb R^d
		\longrightarrow\mathbb V
		\]
		be jointly Borel measurable, and write
		\[
		\mathbf A=(A_b,A_\sigma,A_\beta)
		\]
		for its drift, matrix, and $H$-valued components, respectively. Assume that
		\begin{equation}
			\label{eq:integral-interaction-uniform-bound}
			\sup_{y\in\mathbb R^d}
			|\mathbf A(t,x,y)|_{\mathbb V}
			\leq M
		\end{equation}
		for every \((t,x)\), and that
		\begin{equation}
			\label{eq:kernel-law-lipschitz}
			|\mathbf A(t,x,y)-\mathbf A(t,x,y')|_{\mathbb V}
			\leq
			L_\mu|y-y'|
		\end{equation}
		for all \(t,x,y,y'\). Assume also that the matrix component
		\(A_\sigma(t,x,y)\) belongs to a fixed closed convex set
		\(\mathcal C_\lambda\) satisfying
		\eqref{eq:convex-elliptic-matrix-class}.
		
		Fix a countable dense sequence \((y_r)_{r\geq1}\) in \(\mathbb R^d\), and
		suppose that, for every \(R>0\),
		\begin{equation}
			\label{eq:kernel-uniform-spatial-translation}
			\Theta_R(h)
			:=
			\big\|
			\sup_{r\geq1}
			|\mathbf A(t,x+h,y_r)-\mathbf A(t,x,y_r)|_{\mathbb V}
			\big\|_{L^{q_\ast}(Q_{T,R})}
			\longrightarrow0
		\end{equation}
		as \(h\to0\). Define
		\begin{equation}
			\label{eq:integral-interaction-coefficient}
			\mathbf a(t,x,\mu)	:=	\int_{\mathbb R^d}\mathbf A(t,x,y)\,\mu(dy).
		\end{equation}
		Then \(a=(b,\sigma,\overline\beta)\) represents an element of \(\mathfrak A_T\).
	\end{corollary}
	
\begin{proof}
	Lemma~\ref{lem:parameterized-bochner-integration}, applied with
	\(E=\mathbb V\), shows that \(\mathbf a\) is jointly Borel measurable on
	\[
	[0,T]\times\mathbb R^d\times\mathcal P(\mathbb R^d).
	\]
	Since the embedding
	\[
	(\mathcal P_2(\mathbb R^d),W_2)
	\hookrightarrow
	\mathcal P(\mathbb R^d)
	\]
	into the space of probability measures endowed with the weak topology
	is continuous, the restriction of \(\mathbf a\) to
	\[
	[0,T]\times\mathbb R^d\times\mathcal P_2(\mathbb R^d)
	\]
	is jointly Borel measurable. Lemma~\ref{lem:measurable-H-representative}
	then provides a jointly measurable pointwise representative of the
	\(H\)-valued component. Thus \textnormal{(A1)} holds, while
	\eqref{eq:integral-interaction-uniform-bound} gives
	\textnormal{(A2)}.
	
	Let \(\mu,\eta\in\mathcal P_2(\mathbb R^d)\) and
	\(\pi\in\Pi(\mu,\eta)\). By
	\eqref{eq:kernel-law-lipschitz},
	\begin{align*}
		\left|
		\mathbf a(t,x,\mu)
		-
		\mathbf a(t,x,\eta)
		\right|_{\mathbb V}
		&=
		\left|
		\int
		\bigl[
		\mathbf A(t,x,y)
		-
		\mathbf A(t,x,y')
		\bigr]
		\,\pi(dy,dy')
		\right|_{\mathbb V}
		\\
		&\leq
		L_\mu
		\int
		|y-y'|
		\,\pi(dy,dy').
	\end{align*}
	Taking the infimum over
	\(\pi\in\Pi(\mu,\eta)\) gives
	\[
	\left|
	\mathbf a(t,x,\mu)
	-
	\mathbf a(t,x,\eta)
	\right|_{\mathbb V}
	\leq
	L_\mu W_1(\mu,\eta)
	\leq
	L_\mu W_2(\mu,\eta).
	\]
	Thus \textnormal{(A3)} holds.
	
	Since \(\mathcal C_\lambda\) is closed and convex,
	\[
	\sigma(t,x,\mu)
	=
	\int
	A_\sigma(t,x,y)\,\mu(dy)
	\in
	\mathcal C_\lambda.
	\]
	Hence \textnormal{(A4)} follows from
	\eqref{eq:convex-elliptic-matrix-class}.
	
	Finally, by \eqref{eq:kernel-law-lipschitz}, the map
	\[
	y
	\longmapsto
	\mathbf A(t,x+h,y)
	-
	\mathbf A(t,x,y)
	\]
	is \(2L_\mu\)-Lipschitz as a \(\mathbb V\)-valued map. Therefore,
	by the density of \((y_r)_{r\geq1}\),
	\[
	\sup_{r\geq1}
	\left|
	\mathbf A(t,x+h,y_r)
	-
	\mathbf A(t,x,y_r)
	\right|_{\mathbb V}
	=
	\sup_{y\in\mathbb R^d}
	\left|
	\mathbf A(t,x+h,y)
	-
	\mathbf A(t,x,y)
	\right|_{\mathbb V}.
	\]
	Moreover, for every \((t,x)\),
	\begin{align*}
		D_{a,h}(t,x)
		&\leq
		\sup_{\mu\in\mathcal P_2(\mathbb R^d)}
		\int
		\left|
		\mathbf A(t,x+h,y)
		-
		\mathbf A(t,x,y)
		\right|_{\mathbb V}
		\,\mu(dy)
		\\
		&\leq
		\sup_{y\in\mathbb R^d}
		\left|
		\mathbf A(t,x+h,y)
		-
		\mathbf A(t,x,y)
		\right|_{\mathbb V}.
	\end{align*}
	Consequently,
	\[
	\omega_{a,R}(h)
	\leq
	\Theta_R(h)
	\longrightarrow0
	\qquad
	\text{as }h\to0.
	\]
	Proposition~\ref{prop:translation-continuity-completion} therefore
	yields the claim.
\end{proof}
	
\begin{remark}
	\label{rem:interaction-kernel-examples}
	For coefficients of the form
	\[
	\mathbf a(t,x,\mu)
	=
	\mathbf a_0(t,x)
	+
	\int_{\mathbb R^d}
	\mathbf K(t,x,y)\,\mu(dy),
	\]
	Corollary~\ref{cor:integral-interaction-completion} applies with
	\[
	\mathbf A(t,x,y)
	:=
	\mathbf a_0(t,x)+\mathbf K(t,x,y).
	\]
	In particular,
	\[
	\sup_{y\in\mathbb R^d}
	|\mathbf a_0(t,x)+\mathbf K(t,x,y)|_{\mathbb V}
	\leq M,
	\]
	and \eqref{eq:kernel-law-lipschitz} is necessary as well as sufficient
	for the law-Lipschitz estimate, as seen by testing on Dirac masses.
	
	Condition \eqref{eq:kernel-uniform-spatial-translation} holds, for example,
	when \(\mathbf A\) is uniformly continuous in \(x\), uniformly in
	\((t,y)\), or when it is piecewise defined on a fixed finite Borel
	partition in \(x\) with uniformly translation-continuous branches. It
	also covers finite-rank kernels of the form
	\[
	\mathbf A(t,x,y)
	=
	\sum_{r=1}^J
	\psi_r(y)\mathbf v_r(t,x),
	\]
	with bounded Lipschitz \(\psi_r\) and bounded Borel
	\(\mathbb V\)-valued profiles \(\mathbf v_r\).
	
	For convolution kernels
	\[
	\mathbf A(t,x,y)
	=
	\boldsymbol\kappa(t,x-y),
	\]
	the condition reduces to
	\[
	\big\|
	\sup_{z\in\mathbb R^d}
	|\boldsymbol\kappa(t,z+h)-\boldsymbol\kappa(t,z)|_{\mathbb V}
	\big\|_{L^{q_\ast}(0,T)}
	\longrightarrow0.
	\]
	Thus bounded spatially uniformly continuous kernels are covered, whereas
	discontinuous kernels with moving interfaces need not be.
\end{remark}
	
	We next give an element of $(\mathfrak A_T)$ with a genuine spatial
	discontinuity. The additional strict inequality imposed below is used
	only for this example and is not required in the main results.
	
	Assume, for this example, that
	\begin{equation}
		\label{eq:strict-bound-margin-example}
		M>\sqrt{2\lambda d}.
	\end{equation}
	Choose a unit vector $e\in\mathbb R^d$, an element $h_0\in H$ with	\(\|h_0\|_H=1\)
	and a $\mathcal U$-measurable representative
	\(\widehat h_0:U\longrightarrow\mathbb R^d\)
	of $h_0$. Choose also a non-zero $1$-Lipschitz function
	\[
	\varphi:\mathbb R^d\longrightarrow[0,1].
	\]
	Set
	\[
	F(\mu)
	:=
	\int_{\mathbb R^d}\varphi(y)\,\mu(dy),
	\qquad
	\chi(x)
	:=
	\mathbf{1}_{\{x_1>0\}}.
	\]
	Then
	\begin{equation}
		\label{eq:F-Wasserstein-Lipschitz}
		|F(\mu)-F(\eta)|
		\leq
		W_1(\mu,\eta)
		\leq
		W_2(\mu,\eta).
	\end{equation}
	
	Let
	\[
	\delta
	:=
	\min\left\{
	M-\sqrt{2\lambda d},
	L_\mu
	\right\}>0.
	\]
	Choose $(c_b,c_\sigma,c_\beta>0)$ such that
	\begin{equation}
		\label{eq:example-small-constant-condition}
		c_b+\sqrt d\,c_\sigma+c_\beta
		\leq
		\delta.
	\end{equation}
	For example, one may take
	\[
	c_b=c_\beta=\frac{\delta}{4},
	\qquad
	c_\sigma=\frac{\delta}{4\sqrt d}.
	\]
	It follows that
	\begin{equation}
		\label{eq:example-constant-conditions}
		c_b+\sqrt d\bigl(\sqrt{2\lambda}+c_\sigma\bigr)+c_\beta
		\leq M,
		\qquad
		c_b+\sqrt d\,c_\sigma+c_\beta
		\leq L_\mu.
	\end{equation}
	
	Define
	\begin{align}
		b(t,x,\mu)
		&:=
		c_b\chi(x)F(\mu)e,
		\label{eq:discontinuous-drift-example}\\
		\sigma(t,x,\mu)
		&:=
		\bigl(
		\sqrt{2\lambda}
		+c_\sigma\chi(x)F(\mu)
		\bigr)I_d,
		\label{eq:discontinuous-diffusion-example}\\
		\overline\beta(t,x,\mu)
		&:=
		c_\beta\chi(x)F(\mu)h_0.
		\label{eq:discontinuous-jump-example}
	\end{align}
	A jointly measurable pointwise representative of the jump coefficient
	is given by
	\begin{equation}
		\label{eq:discontinuous-jump-representative}
		\beta(t,x,\mu,z)
		:=
		c_\beta\chi(x)F(\mu)\widehat h_0(z).
	\end{equation}
	
	\begin{proposition}
		\label{prop:discontinuous-example-in-completion}
		The coefficient triple
		\[
		a=(b,\sigma,\overline\beta)
		\]
		defined by
		\eqref{eq:discontinuous-drift-example}--
		\eqref{eq:discontinuous-jump-example}
		belongs to $(\mathfrak A_T)$ and satisfies
		\textnormal{(A1)}--\textnormal{(A4)}. Moreover, its equivalence
		class contains no admissible representative whose coefficient
		components are continuous in the state variable. In fact, the
		spatial discontinuity of each of the three coefficient components
		cannot be removed by modifying the coefficients on a
		$(dt,dx)$-null set.
	\end{proposition}
	
	\begin{proof}
		The map $F:\mathcal P_2(\mathbb R^d)\to\mathbb R$ is Lipschitz
		continuous by \eqref{eq:F-Wasserstein-Lipschitz}, and $\chi$ is
		Borel measurable. Hence $b,\sigma$, and $\overline\beta$ are
		jointly Borel measurable. Moreover, the map
		\[
		(t,x,\mu,z)
		\longmapsto
		\beta(t,x,\mu,z)
		=
		c_\beta\chi(x)F(\mu)\widehat h_0(z)
		\]
		is jointly measurable and satisfies
		\[
		[\beta(t,x,\mu,\cdot)]_H
		=
		c_\beta\chi(x)F(\mu)h_0
		=
		\overline\beta(t,x,\mu).
		\]
		Thus \textnormal{(A1)} holds.
		
		Since
		\[
		0\leq\chi\leq1,
		\qquad
		0\leq F\leq1,
		\]
		we have
		\[
		\begin{aligned}
			&|b(t,x,\mu)|
			+
			\|\sigma(t,x,\mu)\|_{\mathrm{HS}}
			+
			\|\overline\beta(t,x,\mu)\|_H
			\\
			&\qquad=
			c_b\chi(x)F(\mu)
			+
			\sqrt d\,
			\bigl(
			\sqrt{2\lambda}
			+c_\sigma\chi(x)F(\mu)
			\bigr)
			+
			c_\beta\chi(x)F(\mu)
			\\
			&\qquad\leq
			c_b
			+
			\sqrt d\bigl(\sqrt{2\lambda}+c_\sigma\bigr)
			+
			c_\beta
			\\
			&\qquad\leq M
		\end{aligned}
		\]
		by \eqref{eq:example-constant-conditions}. Hence
		\textnormal{(A2)} holds.
		
		For every $\mu,\eta\in\mathcal P_2(\mathbb R^d)$,
		\[
		\begin{aligned}
			&|b(t,x,\mu)-b(t,x,\eta)|
			+
			\|\sigma(t,x,\mu)-\sigma(t,x,\eta)\|_{\mathrm{HS}}
			\\
			&\qquad+
			\|\overline\beta(t,x,\mu)
			-\overline\beta(t,x,\eta)\|_H
			\\
			&\leq
			\chi(x)
			\bigl(
			c_b+\sqrt d\,c_\sigma+c_\beta
			\bigr)
			|F(\mu)-F(\eta)|
			\\
			&\leq
			\bigl(
			c_b+\sqrt d\,c_\sigma+c_\beta
			\bigr)
			W_2(\mu,\eta)
			\\
			&\leq
			L_\mu W_2(\mu,\eta).
		\end{aligned}
		\]
		Thus \textnormal{(A3)} holds.
		
		Furthermore,
		\[
		\begin{aligned}
			\frac12
			\sigma(t,x,\mu)\sigma(t,x,\mu)^\top
			&=
			\frac12
			\bigl(
			\sqrt{2\lambda}
			+c_\sigma\chi(x)F(\mu)
			\bigr)^2I_d
			\\
			&\geq
			\lambda I_d,
		\end{aligned}
		\]
		so \textnormal{(A4)} holds.
		
		We next prove that $a$ belongs to the completion of the
		Lipschitz core. For $n\geq1$, define
		\[
		\chi_n(x)
		:=
		\begin{cases}
			0,
			&x_1\leq0,\\[1mm]
			nx_1,
			&0<x_1<n^{-1},\\[1mm]
			1,
			&x_1\geq n^{-1},
		\end{cases}
		\]
		and let
		\[
		a_n=(b_n,\sigma_n,\overline\beta_n)
		\]
		be obtained by replacing $\chi$ with $\chi_n$ in
		\eqref{eq:discontinuous-drift-example}--
		\eqref{eq:discontinuous-jump-example}. More explicitly,
		\[
		\begin{aligned}
			b_n(t,x,\mu)
			&=
			c_b\chi_n(x)F(\mu)e,\\
			\sigma_n(t,x,\mu)
			&=
			\bigl(
			\sqrt{2\lambda}
			+c_\sigma\chi_n(x)F(\mu)
			\bigr)I_d,\\
			\overline\beta_n(t,x,\mu)
			&=
			c_\beta\chi_n(x)F(\mu)h_0.
		\end{aligned}
		\]
		The corresponding pointwise jump representative is
		\[
		\beta_n(t,x,\mu,z)
		=
		c_\beta\chi_n(x)F(\mu)\widehat h_0(z).
		\]
		
		Since $\chi_n$ is globally $n$-Lipschitz and $0\leq F\leq1$,
		\[
		\begin{aligned}
			&|b_n(t,x,\mu)-b_n(t,y,\mu)|
			+
			\|\sigma_n(t,x,\mu)-\sigma_n(t,y,\mu)\|_{\mathrm{HS}}
			\\
			&\qquad+
			\|\overline\beta_n(t,x,\mu)
			-\overline\beta_n(t,y,\mu)\|_H
			\\
			&\leq
			n\bigl(
			c_b+\sqrt d\,c_\sigma+c_\beta
			\bigr)|x-y|.
		\end{aligned}
		\]
		Consequently,
		\[
		a_n\in\mathfrak A_T^{\mathrm{Lip}}.
		\]
		Moreover, each \(a_n\) satisfies \textnormal{(A2)}--\textnormal{(A4)}
		with the same constants \(M\), \(L_\mu\), and \(\lambda\),
		independently of \(n\).

		For every $(t,x)$, the definition of the pointwise coefficient
		distance gives
		\[
		\begin{aligned}
			\Delta(a_n,a)(t,x)
			\leq&
			\sup_{\mu\in\mathcal P_2(\mathbb R^d)}
			\Bigl[
			c_b|\chi_n(x)-\chi(x)|F(\mu)
			+
			\sqrt d\,c_\sigma
			|\chi_n(x)-\chi(x)|F(\mu)
			\\
			&+
			c_\beta|\chi_n(x)-\chi(x)|F(\mu)
			\Bigr]
			\\
			\leq&
			C_0|\chi_n(x)-\chi(x)|,
		\end{aligned}
		\]
		where
		\[
		C_0:=c_b+\sqrt d\,c_\sigma+c_\beta.
		\]
		The right-hand side is bounded by $C_0$ and vanishes outside
		\[
		S_n
		:=
		\{x\in\mathbb R^d:0<x_1<n^{-1}\}.
		\]
		Therefore, for every $m\geq1$,
		\begin{equation}
			\label{eq:example-local-convergence}
			\begin{aligned}
				\rho_{T,m}(a_n,a)^{q_\ast}
				&=
				\int_0^T\int_{B_m}
				\Delta(a_n,a)(t,x)^{q_\ast}\,dx\,dt
				\\
				&\leq
				C_0^{q_\ast}T
				|B_m\cap S_n|
				\\
				&\leq
				\frac{C_{m,T}}{n}.
			\end{aligned}
		\end{equation}
		Hence
		\[
		\rho_{T,m}(a_n,a)\longrightarrow0
		\qquad
		\text{for every }m\geq1.
		\]
		Since the \(m\)-th summand in the definition of \(d_T\) is bounded by
		\(2^{-m}\), and \(\sum_{m\geq1}2^{-m}<\infty\), dominated convergence yields
		\[
		d_T(a_n,a)\longrightarrow0.
		\]
		Thus
		\[
		a\in\mathfrak A_T.
		\]
		
		It remains to prove that the spatial discontinuities cannot be
		removed by choosing another admissible representative. Since
		$\varphi$ is non-zero and non-negative, there exists
		$y_0\in\mathbb R^d$ such that
		\[
		\varphi(y_0)>0.
		\]
		Set
		\[
		\mu_0:=\delta_{y_0},
		\qquad
		F_0:=F(\mu_0)=\varphi(y_0)>0.
		\]
		
		Suppose, for contradiction, that the equivalence class of $a$
		contains an admissible representative
		\[
		\widetilde a
		=
		(\widetilde b,\widetilde\sigma,
		\overline{\widetilde\beta})
		\]
		such that all three maps
		\[
		x\longmapsto\widetilde b(t,x,\mu),
		\qquad
		x\longmapsto\widetilde\sigma(t,x,\mu),
		\qquad
		x\longmapsto
		\overline{\widetilde\beta}(t,x,\mu)
		\]
		are continuous for every
		\[
		(t,\mu)\in[0,T]\times\mathcal P_2(\mathbb R^d).
		\]
		
		Since $a$ and $\widetilde a$ represent the same element of
		$\mathfrak A_T$, there exists a Borel $dt\,dx$-null set
		\[
		\mathcal N\subset[0,T]\times\mathbb R^d
		\]
		such that, for every
		$\mu\in\mathcal P_2(\mathbb R^d)$ and every
		$(t,x)\notin\mathcal N$,
		\[
		\widetilde b(t,x,\mu)=b(t,x,\mu),
		\]
		\[
		\widetilde\sigma(t,x,\mu)=\sigma(t,x,\mu),
		\]
		and
		\[
		\overline{\widetilde\beta}(t,x,\mu)
		=
		\overline\beta(t,x,\mu).
		\]
		
		By Fubini's theorem, there exists $t_0\in[0,T]$ such that
		\[
		\mathcal N_{t_0}
		:=
		\{x\in\mathbb R^d:(t_0,x)\in\mathcal N\}
		\]
		has Lebesgue measure zero. Define
		\[
		H_-:=\{x\in\mathbb R^d:x_1<0\},
		\qquad
		H_+:=\{x\in\mathbb R^d:x_1>0\}.
		\]
		
		For almost every $x\in H_-$,
		\[
		\widetilde b(t_0,x,\mu_0)=0,
		\]
		\[
		\widetilde\sigma(t_0,x,\mu_0)
		=
		\sqrt{2\lambda}\,I_d,
		\]
		and
		\[
		\overline{\widetilde\beta}(t_0,x,\mu_0)=0.
		\]
		Since $H_-\setminus\mathcal N_{t_0}$ is dense in $H_-$,
		continuity implies that these identities hold for every
		$x\in H_-$.
		
		Similarly, for almost every $x\in H_+$,
		\[
		\widetilde b(t_0,x,\mu_0)
		=
		c_bF_0e,
		\]
		\[
		\widetilde\sigma(t_0,x,\mu_0)
		=
		\bigl(
		\sqrt{2\lambda}+c_\sigma F_0
		\bigr)I_d,
		\]
		and
		\[
		\overline{\widetilde\beta}(t_0,x,\mu_0)
		=
		c_\beta F_0h_0.
		\]
		Continuity therefore implies that these identities hold for every
		$x\in H_+$.
		
		Let $x^\ast\in\mathbb R^d$ satisfy $x_1^\ast=0$, and choose
		sequences
		\[
		x_n^-\in H_-,
		\qquad
		x_n^+\in H_+,
		\]
		such that
		\[
		x_n^-\longrightarrow x^\ast,
		\qquad
		x_n^+\longrightarrow x^\ast.
		\]
		Continuity of the drift would imply simultaneously
		\[
		\widetilde b(t_0,x^\ast,\mu_0)=0
		\]
		and
		\[
		\widetilde b(t_0,x^\ast,\mu_0)
		=
		c_bF_0e,
		\]
		which is impossible because $c_b>0$, $F_0>0$, and $|e|=1$.
		
		Likewise, continuity of the diffusion coefficient would imply
		simultaneously
		\[
		\widetilde\sigma(t_0,x^\ast,\mu_0)
		=
		\sqrt{2\lambda}\,I_d
		\]
		and
		\[
		\widetilde\sigma(t_0,x^\ast,\mu_0)
		=
		\bigl(
		\sqrt{2\lambda}+c_\sigma F_0
		\bigr)I_d,
		\]
		which is impossible because $c_\sigma F_0>0$.
		
		Finally, continuity of the $H$-valued jump coefficient would imply
		simultaneously
		\[
		\overline{\widetilde\beta}(t_0,x^\ast,\mu_0)=0
		\]
		and
		\[
		\overline{\widetilde\beta}(t_0,x^\ast,\mu_0)
		=
		c_\beta F_0h_0,
		\]
		which is impossible because $c_\beta F_0>0$ and
		$\|h_0\|_H=1$.
		
		Thus the equivalence class of $a$ contains no admissible
		representative whose coefficient components are continuous in the
		state variable.
	\end{proof}

	\section{Stability near the Lipschitz core}
	\label{sec:stability-core}
	
	In this section, we prove that strong solutions are stable at every
	coefficient in the Lipschitz core. The nearby coefficient is allowed
	to belong to the completed space $\mathfrak A_T$, and the
	corresponding equation is not assumed to be well posed. More
	precisely, every strong solution of a nearby equation must remain
	close to the unique solution associated with the core coefficient.

	\subsection{Moment and Localization Estimates}
	
	For an adapted c\`adl\`ag process $X$ and $R>0$, define
	\[
	\tau_R(X)
	:=
	\inf\{t\in[0,T]:|X_t|\geq R\}\wedge T.
	\]
	For two such processes $X$ and $Y$, set
	\[
	\tau_R(X,Y)
	:=
	\tau_R(X)\wedge\tau_R(Y).
	\]
	
	\begin{proposition}
		\label{prop:moment-localization-estimates}
		The following assertions hold.
		
		\begin{enumerate}
			\item[\textnormal{(i)}]
			Let $a\in\mathfrak A_T$, $x\in\mathbb R^d$, and
			$X\in\operatorname{Sol}_T(a,x)$. Then
			\begin{equation}
				\label{eq:uniform-second-moment}
				\mathbb E\big[
				\sup_{0\leq t\leq T}|X_t|^2
				\big]
				\leq
				C_{\mathrm{mom}}(1+|x|^2),
			\end{equation}
			where $C_{\mathrm{mom}}=C_{\mathrm{mom}}(T,M)$ is independent
			of $a$, $x$, and $X$. Consequently,
			\begin{equation}
				\label{eq:uniform-exit-probability}
				\mathbb P\bigl(\tau_R(X)<T\bigr)
				\leq
				\frac{C_{\mathrm{mom}}(1+|x|^2)}{R^2}.
			\end{equation}
			
			\item[\textnormal{(ii)}]
			Let $a,\widetilde a\in\mathfrak A_T$,
			$y\in\mathbb R^d$, and
			$Y\in\operatorname{Sol}_T(a,y)$. Set
			\[
			\nu_t:=\mathcal L(Y_t).
			\]
			Then, for every stopping time $\theta$ satisfying
			\[
			0\leq\theta\leq\tau_R(Y),
			\]
			we have
			\begin{equation}
				\label{eq:stopped-coefficient-error}
				\mathbb E\int_0^\theta
				\mathfrak e_{a,\widetilde a}
				(s,Y_s,\nu_s)\,ds
				\leq
				C_{\mathrm K}(R)
				\rho_{T,R}(a,\widetilde a)^2,
			\end{equation}
			where
			\[
			C_{\mathrm K}(R)
			=
			C_{\mathrm K}(d,T,R,M,\lambda).
			\]
			In particular, for any adapted c\`adl\`ag process $X$,
			\begin{equation}
				\label{eq:common-stopped-error}
				\mathbb E\int_0^{\tau_R(X,Y)}
				\mathfrak e_{a,\widetilde a}
				(s,Y_s,\nu_s)\,ds
				\leq
				C_{\mathrm K}(R)
				\rho_{T,R}(a,\widetilde a)^2.
			\end{equation}
		\end{enumerate}
	\end{proposition}
	
	\begin{proof}
		Let $X\in\operatorname{Sol}_T(a,x)$ and write $\mu_t:=\mathcal L(X_t)$. By the Burkholder--Davis--Gundy	inequalities,	and the uniform bound \textnormal{(A2)},
		\[
		\begin{aligned}
			\mathbb E\sup_{0\leq t\leq T}|X_t|^2
			&\leq
			4|x|^2
			+
			4T\mathbb E\int_0^T
			|b(s,X_s,\mu_s)|^2\,ds
			\\
			&\quad+
			C\mathbb E\int_0^T
			\left(
			\|\sigma(s,X_s,\mu_s)\|_{\mathrm{HS}}^2
			+
			\|\overline\beta(s,X_s,\mu_s)\|_H^2
			\right)ds
			\\
			&\leq
			4|x|^2+C(T+T^2)M^2.
		\end{aligned}
		\]
		This proves \eqref{eq:uniform-second-moment}. Since
		\[
		\{\tau_R(X)<T\}
		\subset
		\left\{
		\sup_{0\leq t\leq T}|X_t|\geq R
		\right\},
		\]
		\eqref{eq:uniform-exit-probability} follows from Markov's
		inequality.
		
		For \textnormal{(ii)}, the solution $Y$ is an It\^o--L\'evy
		process whose coefficients satisfy the boundedness and uniform
		ellipticity assumptions required by
		Corollary~\ref{cor:coefficient-error-occupation}. Applying that
		corollary with the stopping time $\theta\leq\tau_R(Y)$ gives
		\eqref{eq:stopped-coefficient-error}. Finally,
		\[
		\tau_R(X,Y)\leq\tau_R(Y),
		\]
		so \eqref{eq:common-stopped-error} follows by taking
		$\theta=\tau_R(X,Y)$.
	\end{proof}

	\subsection{Quantitative Stability and Its Consequences}
	
	Fix
	\[
	a^0=(b^0,\sigma^0,\overline\beta^0)
	\in\mathfrak A_T^{\mathrm{Lip}},
	\]
	and let
	\[
	L_0:=L_x(a^0)+L_\mu.
	\]
	Then
	\begin{equation}
		\label{eq:full-core-lipschitz}
		\left|
		\mathbf a^0(t,x,\mu)
		-
		\mathbf a^0(t,y,\eta)
		\right|_{\mathbb V}
		\leq
		L_0\bigl(|x-y|+W_2(\mu,\eta)\bigr).
	\end{equation}
	
	\begin{theorem}
		\label{thm:quantitative-core-stability}
		Let $a\in\mathfrak A_T$, $x,y\in\mathbb R^d$, and let
		\[
		X:=X^{a^0,x}, \qquad	Y\in\operatorname{Sol}_T(a,y).
		\]
		Then, for every $R>|x|\vee|y|$,
		\begin{equation}
			\label{eq:quantitative-core-stability}
			\mathbb E\big[\sup_{0\leq t\leq T}|Y_t-X_t|^2\big]
			\leq	C_0\left[|y-x|^2+C_{\mathrm K}(R)\rho_{T,R}(a,a^0)^2+\frac{1+|x|^2+|y|^2}{R^2}
			\right],
		\end{equation}
		where \(C_0=C_0(T,M,L_0)\) is independent of $a,x,y,Y$, and $R$.
	\end{theorem}
	
	\begin{proof}
		Set
		\[
		\mu_t:=\mathcal L(X_t), \qquad	\nu_t:=\mathcal L(Y_t),
		\qquad	Z_t:=Y_t-X_t,
		\]
		and
		\[
		u(t):=	\mathbb E\big[\sup_{0\leq r\leq t}|Z_r|^2\big].
		\]
		The Cauchy--Schwarz and  Burkholder--Davis--Gundy inequalities give
		\[
		u(t)\leq C|y-x|^2+C\mathbb E\int_0^t\mathfrak d_s\,ds,
		\]
		where
		\[
		\begin{aligned}
			\mathfrak d_s
			:={}&|b(s,Y_s,\nu_s)-b^0(s,X_s,\mu_s)|^2\\
			&+\|\sigma(s,Y_s,\nu_s)-\sigma^0(s,X_s,\mu_s)\|_{\mathrm{HS}}^2\\
			&+\|\overline\beta(s,Y_s,\nu_s)-\overline\beta^0(s,X_s,\mu_s)\|_H^2.
		\end{aligned}
		\]
		Here the left limits in the jump term may be replaced by present
		values in the $ds$-integral, since c\`adl\`ag paths have at most
		countably many jumps.
		
		Inserting $\mathbf a^0(s,Y_s,\nu_s)$ and using
		\eqref{eq:full-core-lipschitz}, we obtain
		\[
		\mathfrak d_s\leq C\mathfrak e_{a,a^0}(s,Y_s,\nu_s)	+CL_0^2\left(|Z_s|^2+W_2(\nu_s,\mu_s)^2\right).
		\]
		Since the joint law of $(X_s,Y_s)$ is a coupling of
		$\mu_s$ and $\nu_s$,
		\[
		W_2(\nu_s,\mu_s)^2\leq\mathbb E|Y_s-X_s|^2\leq u(s).
		\]
		Consequently,
		\begin{equation}
			\label{eq:core-stability-prelocalization}
			u(t)
			\leq
			C|y-x|^2
			+
			C\mathbb E\int_0^t
			\mathfrak e_{a,a^0}(s,Y_s,\nu_s)\,ds
			+
			CL_0^2\int_0^tu(s)\,ds.
		\end{equation}
		
		Let
		\[
		\tau_R:=\tau_R(X,Y).
		\]
		By Proposition~\ref{prop:moment-localization-estimates},
		\[
		\mathbb E\int_0^{t\wedge\tau_R}
		\mathfrak e_{a,a^0}(s,Y_s,\nu_s)\,ds
		\leq
		C_{\mathrm K}(R)\rho_{T,R}(a,a^0)^2.
		\]
		Moreover, the uniform coefficient bound gives
		\[
		\mathfrak e_{a,a^0}(s,x,\mu)\leq CM^2,
		\]
		and hence
		\[
		\begin{aligned}
			\mathbb E\int_{t\wedge\tau_R}^{t}
			\mathfrak e_{a,a^0}(s,Y_s,\nu_s)\,ds
			&\leq
			CM^2T\,\mathbb P(\tau_R<T)
			\\
			&\leq
			\frac{C(T,M)(1+|x|^2+|y|^2)}{R^2},
		\end{aligned}
		\]
		where the last inequality follows from the uniform exit estimate in
		Proposition~\ref{prop:moment-localization-estimates}. Substituting these
		bounds into \eqref{eq:core-stability-prelocalization} and applying
		Gronwall's inequality proves
		\eqref{eq:quantitative-core-stability}.
	\end{proof}
	
	The estimate holds for every
	$Y\in\operatorname{Sol}_T(a,y)$; no uniqueness assumption is imposed
	on the nearby equation.
	
	\begin{corollary}
		\label{cor:core-initial-state-continuity}
		Let
		\[
		c\in\mathfrak A_T^{\mathrm{Lip}}.
		\]
		Then there exists a constant
		\[
		C_{\mathrm{init}}(c)
		=
		C_{\mathrm{init}}
		\bigl(T,M,L_x(c)+L_\mu\bigr)
		<\infty
		\]
		such that
		\begin{equation}
			\label{eq:core-initial-state-continuity}
			\|X^{c,x}-X^{c,y}\|_{\mathcal S_T^2}
			\leq
			C_{\mathrm{init}}(c)|x-y|
		\end{equation}
		for all $x,y\in\mathbb R^d$. Consequently, the map
		\[
		\Phi_T(c):
		\mathbb R^d\longrightarrow\mathcal S_T^2,
		\qquad
		\Phi_T(c)(x):=X^{c,x},
		\]
		is globally Lipschitz continuous. In particular, for every
		$m\in\mathbb N$,
		\[
		\Phi_T(c)|_{\overline B_m}
		\in
		C(\overline B_m;\mathcal S_T^2),
		\qquad
		\overline B_m:=\{x\in\mathbb R^d:|x|\leq m\}.
		\]
	\end{corollary}
	
	\begin{proof}
		Fix $x,y\in\mathbb R^d$. Apply
		Theorem~\ref{thm:quantitative-core-stability} with
		\[
		a^0=c,
		\qquad
		a=c,
		\qquad
		X=X^{c,x},
		\qquad
		Y=X^{c,y}.
		\]
		Since
		\[
		\rho_{T,R}(c,c)=0,
		\]
		the theorem gives, for every $R>|x|\vee|y|$,
		\[
		\|X^{c,y}-X^{c,x}\|_{\mathcal S_T^2}^2
		\leq
		C_0(c)
		\left[
		|x-y|^2
		+
		\frac{1+|x|^2+|y|^2}{R^2}
		\right],
		\]
		where
		\[
		C_0(c)=C_0\bigl(T,M,L_x(c)+L_\mu\bigr)
		\]
		is independent of $x,y$, and $R$. Letting $R\to\infty$ yields
		\[
		\|X^{c,y}-X^{c,x}\|_{\mathcal S_T^2}^2
		\leq
		C_0(c)|x-y|^2.
		\]
		Thus \eqref{eq:core-initial-state-continuity} holds with
		\[
		C_{\mathrm{init}}(c):=C_0(c)^{1/2}.
		\]
	\end{proof}
	
	\begin{corollary}
		\label{cor:uniform-stability-compact-initial-set}
		Let
		\[
		c\in\mathfrak A_T^{\mathrm{Lip}},
		\qquad
		m,n\in\mathbb N,
		\]
		and set
		\begin{equation}
			\label{eq:uniform-generic-epsilon}
			\varepsilon_n:=2^{-n-3}.
		\end{equation}
		Then there exists a radius
		\[
		r_{m,n}(c)>0
		\]
		such that, whenever
		\[
		a\in\mathfrak A_T,
		\qquad
		d_T(a,c)<r_{m,n}(c),
		\]
		one has
		\begin{equation}
			\label{eq:uniform-stability-compact-conclusion}
			\|Y-X^{c,x}\|_{\mathcal S_T^2}
			<
			\varepsilon_n
		\end{equation}
		for every $x\in\overline B_m$ and every
		\[
		Y\in\operatorname{Sol}_T(a,x).
		\]
	\end{corollary}
	
	\begin{proof}
		Let
		\[
		C_0(c)=C_0\bigl(T,M,L_x(c)+L_\mu\bigr)
		\]
		be the constant in
		Theorem~\ref{thm:quantitative-core-stability} when the reference
		coefficient is $a^0=c$. Choose an integer $J>m$ sufficiently large
		that
		\begin{equation}
			\label{eq:uniform-stability-tail-choice}
			C_0(c)\frac{1+2m^2}{J^2}
			<
			\frac{\varepsilon_n^2}{2}.
		\end{equation}
		Next choose $\delta\in(0,1)$ sufficiently small that
		\begin{equation}
			\label{eq:uniform-stability-local-choice}
			C_0(c)C_{\mathrm K}(J)\delta^2
			<
			\frac{\varepsilon_n^2}{2}.
		\end{equation}
		Define
		\begin{equation}
			\label{eq:uniform-stability-radius}
			r_{m,n}(c):=2^{-J}\delta.
		\end{equation}
		
		Suppose that $d_T(a,c)<r_{m,n}(c)$. By the definition of $d_T$,
		the $J$-th summand satisfies
		\[
		2^{-J}\bigl(1\wedge\rho_{T,J}(a,c)\bigr)
		\leq
		d_T(a,c)
		<
		2^{-J}\delta.
		\]
		Consequently,
		\[
		1\wedge\rho_{T,J}(a,c)<\delta.
		\]
		Since $\delta<1$, this implies
		\begin{equation}
			\label{eq:uniform-stability-local-distance}
			\rho_{T,J}(a,c)<\delta.
		\end{equation}
		
		Fix $x\in\overline B_m$ and
		$Y\in\operatorname{Sol}_T(a,x)$. Applying
		Theorem~\ref{thm:quantitative-core-stability} with
		\[
		a^0=c,
		\qquad
		y=x,
		\qquad
		R=J,
		\]
		gives
		\[
		\begin{aligned}
			\|Y-X^{c,x}\|_{\mathcal S_T^2}^2
			&\leq
			C_0(c)
			\left[
			C_{\mathrm K}(J)\rho_{T,J}(a,c)^2
			+
			\frac{1+2|x|^2}{J^2}
			\right]
			\\
			&\leq
			C_0(c)
			\left[
			C_{\mathrm K}(J)\delta^2
			+
			\frac{1+2m^2}{J^2}
			\right]
			<
			\varepsilon_n^2.
		\end{aligned}
		\]
		This proves \eqref{eq:uniform-stability-compact-conclusion}.
	\end{proof}
	
	\begin{theorem}
		\label{thm:sequential-core-stability}
		Let
		\[
		a_n\in\mathfrak A_T,
		\qquad
		d_T(a_n,a^0)\longrightarrow0,
		\qquad
		x_n\longrightarrow x,
		\]
		and choose arbitrary solutions
		\[
		X^n\in\operatorname{Sol}_T(a_n,x_n).
		\]
		Then
		\begin{equation}
			\label{eq:sequential-core-convergence}
			\|X^n-X^{a^0,x}\|_{\mathcal S_T^2}
			\longrightarrow0.
		\end{equation}
	\end{theorem}
	
	\begin{proof}
		The sequence $(x_n)$ is bounded. For each fixed $R>|x|$ and all
		sufficiently large $n$, Theorem~\ref{thm:quantitative-core-stability}
		gives
		\[
		\begin{aligned}
			\|X^n-X^{a^0,x}\|_{\mathcal S_T^2}^2
			\leq
			C_0\left[
			|x_n-x|^2
			+
			C_{\mathrm K}(R)\rho_{T,R}(a_n,a^0)^2
			+
			\frac{1+|x|^2+|x_n|^2}{R^2}
			\right].
		\end{aligned}
		\]
		For fixed $R$, convergence in $d_T$ implies
		$\rho_{T,R}(a_n,a^0)\to0$. Hence
		\[
		\limsup_{n\to\infty}
		\|X^n-X^{a^0,x}\|_{\mathcal S_T^2}^2
		\leq
		\frac{
			C\bigl(1+|x|^2+\sup_n|x_n|^2\bigr)
		}{R^2}.
		\]
		Letting $R\to\infty$ proves the assertion.
	\end{proof}
	
	\begin{corollary}
		\label{cor:core-solution-map-continuity}
		Fix
		\[
		a^0\in\mathfrak A_T^{\mathrm{Lip}},
		\qquad
		x\in\mathbb R^d,
		\qquad
		X^0:=X^{a^0,x}.
		\]
		Then the following assertions hold.
		
		\begin{enumerate}
			\item[\textnormal{(i)}]
			If
			\[
			a_n\in\mathfrak A_T^{\mathrm{Lip}},
			\qquad
			d_T(a_n,a^0)\to0,
			\qquad
			x_n\to x,
			\]
			then
			\[
			\|X^{a_n,x_n}-X^0\|_{\mathcal S_T^2}
			\longrightarrow0.
			\]
			In particular, the core solution map is continuous in the
			relative $d_T$-topology.
			
			\item[\textnormal{(ii)}]
			For every $\varepsilon>0$, there exists
			$r=r(a^0,x,\varepsilon)>0$ such that
			\[
			d_T(a,a^0)<r,
			\qquad
			|y-x|<r
			\]
			imply
			\begin{equation}
				\label{eq:stable-neighborhood-conclusion}
				\|Y-X^0\|_{\mathcal S_T^2}<\varepsilon
			\end{equation}
			for every
			\[
			Y\in\operatorname{Sol}_T(a,y).
			\]
			
			\item[\textnormal{(iii)}]
			For the same $r$, if $d_T(a,a^0)<r$, then
			\begin{equation}
				\label{eq:solution-set-contained-ball}
				\operatorname{Sol}_T(a,x)
				\subset
				B_{\mathcal S_T^2}(X^0,\varepsilon),
			\end{equation}
		where \(B_{\mathcal S_T^2}(X^0,\varepsilon):=\left\{X\in\mathcal S_T^2:\|X-X^0\|_{\mathcal S_T^2}<\varepsilon\right\}.\)	Consequently, whenever
			$\operatorname{Sol}_T(a,x)\neq\varnothing$,
			\begin{equation}
				\label{eq:nearby-solution-diameter}
				\operatorname{diam}_{\mathcal S_T^2}
				\operatorname{Sol}_T(a,x)
				\leq2\varepsilon,
			\end{equation}
			where 
			\[
			\operatorname{diam}_{\mathcal S_T^2}(A)
			:=
			\sup\left\{ \|X-Y\|_{\mathcal S_T^2} : X,Y\in A \right\}
			\]
			for every nonempty \(A\subset\mathcal S_T^2\).
		\end{enumerate}
	\end{corollary}
	
	\begin{proof}
		Part \textnormal{(i)} follows from
		Theorem~\ref{thm:sequential-core-stability} and the well-posedness of
		core equations.
		
		If \textnormal{(ii)} were false, there would exist
		$a_n\to a^0$, $y_n\to x$, and
		$Y^n\in\operatorname{Sol}_T(a_n,y_n)$ such that
		\[
		\|Y^n-X^0\|_{\mathcal S_T^2}\geq\varepsilon
		\]
		for every $n$, contradicting
		Theorem~\ref{thm:sequential-core-stability}. Part \textnormal{(iii)}
		is obtained from \textnormal{(ii)} by taking $y=x$; the diameter
		bound follows from the triangle inequality.
	\end{proof}
	
	The stable-neighborhood property is the local input used in the Baire
	category argument: near each coefficient in the dense Lipschitz core,
	all solutions of every solvable nearby equation remain close to the
	corresponding core solution.

	\section{Closedness and Identification of Limits}
	\label{sec:closedness-limits}
	
	In this section, we prove that the strong-solution relation is closed
	under simultaneous convergence of coefficients and solutions. The main
	difficulty is that the limiting coefficient may be merely measurable in
	the state variable, so pointwise convergence of the coefficient
	integrands cannot be expected. We overcome this difficulty by combining
	Krylov occupation estimates with approximation by the Lipschitz core.
	
	\subsection{Krylov Control and Convergence of the Integrands}
	
	Let
	\[
	a_n=(b_n,\sigma_n,\overline\beta_n)
	\in\mathfrak A_T^{\mathrm{Lip}},
	\qquad
	a=(b,\sigma,\overline\beta)\in\mathfrak A_T,
	\]
	and assume that
	\begin{equation}
		\label{eq:coefficient-convergence-closedness}
		d_T(a_n,a)\longrightarrow0.
	\end{equation}
	Let $x_n\to x$ in $\mathbb R^d$, and set
	\[
	X^n:=X^{a_n,x_n},
	\qquad
	\mu_t^n:=\mathcal L(X_t^n).
	\]
	Suppose that, for some $X\in\mathcal S_T^2$,
	\begin{equation}
		\label{eq:solution-convergence-closedness}
		\|X^n-X\|_{\mathcal S_T^2}
		\longrightarrow0,
	\end{equation}
	and write
	\[
	\mu_t:=\mathcal L(X_t).
	\]
	Since $X^n$ and $X$ are defined on the same probability space,
	\begin{equation}
		\label{eq:uniform-law-convergence}
		\sup_{0\leq t\leq T}	W_2(\mu_t^n,\mu_t)^2\leq\mathbb E\big[\sup_{0\leq t\leq T}|X_t^n-X_t|^2\big]
		\longrightarrow0.
	\end{equation}
	Moreover, Proposition~\ref{prop:moment-localization-estimates} and
	\eqref{eq:solution-convergence-closedness} imply
	\begin{equation}
		\label{eq:uniform-moment-closedness}
		\sup_{n\geq1}\mathbb E\big[\sup_{0\leq t\leq T}|X_t^n|^2\big]
		+\mathbb E\big[\sup_{0\leq t\leq T}|X_t|^2\big]	<\infty.
	\end{equation}
	
	Although $X$ is not yet known to solve an equation, it inherits the
	Krylov occupation estimate from the approximating solutions.
	
	\begin{proposition}
		\label{prop:limit-krylov-estimate}
		For every $R>0$ and every non-negative Borel function
		$f\in L^{d+1}(Q_{T,R})$,
		\begin{equation}
			\label{eq:limit-krylov-estimate}
			\mathbb E\int_0^{\tau_R(X)}
			f(t,X_t)\,dt
			\leq
			C_{\mathrm K}^{\mathrm{lim}}(R)
			\|f\|_{L^{d+1}(Q_{T,R})},
		\end{equation}
		where
		\[
		C_{\mathrm K}^{\mathrm{lim}}(R)
		=
		C_{\mathrm K}(R+1)
		\]
		depends only on $d,T,R,M$, and $\lambda$.
	\end{proposition}
	
	\begin{proof}
		We first prove the estimate for
		\[
		f\in C_c^\infty((0,T)\times B_R),
		\qquad f\geq0,
		\]
		extended by zero outside $Q_{T,R}$. Set
		\[
		\tau:=\tau_R(X),
		\qquad
		\tau_n:=\tau\wedge\tau_{R+1}(X^n).
		\]
		Since $\tau_n\leq\tau_{R+1}(X^n)$, the stopped Krylov estimate for
		$X^n$ gives
		\begin{equation}
			\label{eq:limit-krylov-approximate}
			\mathbb E\int_0^{\tau_n}
			f(t,X_t^n)\,dt
			\leq
			C_{\mathrm K}(R+1)
			\|f\|_{L^{d+1}(Q_{T,R})}.
		\end{equation}
		
		Let
		\[
		A_n:=
		\left\{
		\sup_{0\leq t\leq T}|X_t^n-X_t|<1
		\right\}.
		\]
		By \eqref{eq:solution-convergence-closedness},
		$\mathbb P(A_n^c)\to0$. On $A_n$, for every $t<\tau$,
		\[
		|X_t|<R
		\quad\Longrightarrow\quad
		|X_t^n|<R+1.
		\]
		Therefore,
		\[
		\tau_{R+1}(X^n)\geq\tau
		\]
		on $A_n$, and consequently
		\begin{equation}
			\tau_n=\tau \quad\text{on }A_n.
		\end{equation}
		Therefore,
		\[
		\begin{aligned}
			&
			\left|
			\mathbb E\int_0^\tau f(t,X_t)\,dt
			-
			\mathbb E\int_0^{\tau_n}f(t,X_t^n)\,dt
			\right|
			\\
			&\quad\leq
			\mathbb E\left[
			\mathbf 1_{A_n}
			\int_0^\tau
			|f(t,X_t)-f(t,X_t^n)|\,dt
			\right]
			+
			2T\|f\|_\infty\mathbb P(A_n^c).
		\end{aligned}
		\]
		The right-hand side tends to zero because $f$ is uniformly continuous,
		$X^n\to X$ uniformly in probability, and the integrands are uniformly
		bounded. Passing to the limit in
		\eqref{eq:limit-krylov-approximate} yields
		\[
		\mathbb E\int_0^{\tau_R(X)}
		f(t,X_t)\,dt
		\leq
		C_{\mathrm K}(R+1)
		\|f\|_{L^{d+1}(Q_{T,R})}.
		\]
		
		Define the occupation measure
		\[
		\Gamma_R(B)
		:=
		\mathbb E\int_0^{\tau_R(X)}
		\mathbf 1_B(t,X_t)\,dt,
		\qquad
		B\in\mathcal B(Q_{T,R}).
		\]
		The preceding inequality shows that
		\[
		f\longmapsto\int_{Q_{T,R}}f\,d\Gamma_R
		\]
		is bounded on $C_c^\infty((0,T)\times B_R)$ with respect to the
		$L^{d+1}$ norm. By density and $L^p$ duality, there exists
		\[
		g_R\in L^{(d+1)/d}(Q_{T,R})
		\]
		such that
		\[
		\int_{Q_{T,R}}f\,d\Gamma_R
		=
		\int_{Q_{T,R}}f(t,x)g_R(t,x)\,dt\,dx
		\]
		for every $f\in C_c^\infty((0,T)\times B_R)$, with
		\[
		\|g_R\|_{L^{(d+1)/d}(Q_{T,R})}
		\leq
		C_{\mathrm K}(R+1).
		\]
		By uniqueness of Radon measures,
		\[
		d\Gamma_R=g_R\,dt\,dx.
		\]
		Consequently, \eqref{eq:limit-krylov-estimate} holds for every
		non-negative Borel function
		$f\in L^{d+1}(Q_{T,R})$.
	\end{proof}
	
	The preceding estimate immediately gives a coefficient-error bound
	along the limiting process.
	
	\begin{corollary}
		\label{cor:limit-process-coefficient-error}
		For every $c,\widetilde c\in\mathfrak A_T$,
		\begin{equation}
			\label{eq:limit-process-error-estimate}
			\mathbb E\int_0^{\tau_R(X)}
			\mathfrak e_{c,\widetilde c}
			(t,X_t,\mu_t)\,dt
			\leq
			C_{\mathrm K}^{\mathrm{lim}}(R)
			\rho_{T,R}(c,\widetilde c)^2.
		\end{equation}
	\end{corollary}
	
	\begin{proof}
		By \eqref{eq:error-controlled-by-delta},
		\[
		\mathfrak e_{c,\widetilde c}(t,x,\mu)
		\leq
		\Delta(c,\widetilde c)(t,x)^2.
		\]
		Applying Proposition~\ref{prop:limit-krylov-estimate} with
		\[
		f(t,x)=\Delta(c,\widetilde c)(t,x)^2
		\]
		and using
		\[
		\|\Delta(c,\widetilde c)^2\|_{L^{d+1}(Q_{T,R})}
		=
		\rho_{T,R}(c,\widetilde c)^2
		\]
		proves the assertion.
	\end{proof}
	
	Choose
	\begin{equation}
		\label{eq:intermediate-core-sequence}
		c^k=(b^k,\sigma^k,\overline\beta^k)
		\in\mathfrak A_T^{\mathrm{Lip}}
	\end{equation}
	such that
	\begin{equation}
		\label{eq:intermediate-core-convergence}
		d_T(c^k,a)\longrightarrow0.
	\end{equation}
	For each $k$, let $L_k$ satisfy
	\begin{equation}
		\label{eq:intermediate-core-lipschitz}
		\begin{aligned}
			&
			|b^k(t,x,\mu)-b^k(t,y,\eta)|
			+
			\|\sigma^k(t,x,\mu)-\sigma^k(t,y,\eta)\|_{\mathrm{HS}}
			\\
			&\quad+
			\|\overline\beta^k(t,x,\mu)
			-\overline\beta^k(t,y,\eta)\|_H
			\\
			&\leq
			L_k\bigl(|x-y|+W_2(\mu,\eta)\bigr).
		\end{aligned}
	\end{equation}
	No uniform bound on $L_k$ is required.
	
	Define
	\begin{align}
		\mathfrak E_n(t)
		:={}&
		|b_n(t,X_t^n,\mu_t^n)-b(t,X_t,\mu_t)|^2
		\nonumber\\
		&+
		\|\sigma_n(t,X_t^n,\mu_t^n)
		-\sigma(t,X_t,\mu_t)\|_{\mathrm{HS}}^2
		\nonumber\\
		&+
		\|\overline\beta_n(t,X_{t-}^n,\mu_t^n)
		-\overline\beta(t,X_{t-},\mu_t)\|_H^2.
		\label{eq:full-integrand-error}
	\end{align}
	Since c\`adl\`ag paths have at most countably many discontinuities,
	\[
	X_{t-}^n=X_t^n,
	\qquad
	X_{t-}=X_t
	\]
	for $dt\,d\mathbb P$-almost every $(t,\omega)$. Thus left limits may
	be replaced by present values in all time-integrated $H$-norm
	estimates.
	
	\begin{proposition}
		\label{prop:coefficient-integrand-convergence}
		Under
		\eqref{eq:coefficient-convergence-closedness} and
		\eqref{eq:solution-convergence-closedness},
		\begin{equation}
			\label{eq:coefficient-integrand-L2-convergence}
			\mathbb E\int_0^T\mathfrak E_n(t)\,dt
			\longrightarrow0.
		\end{equation}
		In particular,
		\begin{equation}
			\label{eq:drift-integrand-L2}
			\mathbb E\int_0^T
			|b_n(t,X_t^n,\mu_t^n)-b(t,X_t,\mu_t)|^2\,dt
			\longrightarrow0,
		\end{equation}
		\begin{equation}
			\label{eq:diffusion-integrand-L2}
			\mathbb E\int_0^T
			\|\sigma_n(t,X_t^n,\mu_t^n)
			-\sigma(t,X_t,\mu_t)\|_{\mathrm{HS}}^2\,dt
			\longrightarrow0,
		\end{equation}
		and
		\begin{equation}
			\label{eq:jump-integrand-L2}
			\mathbb E\int_0^T
			\|\overline\beta_n(t,X_{t-}^n,\mu_t^n)
			-\overline\beta(t,X_{t-},\mu_t)\|_H^2\,dt
			\longrightarrow0.
		\end{equation}
	\end{proposition}
	
	\begin{proof}
		Fix $R>0$ and define
		\[
		\theta_R^n
		:=
		\tau_R(X^n)\wedge\tau_R(X).
		\]
		We first estimate the error on $[0,\theta_R^n]$.
		
		For each coefficient component, insert the intermediate coefficient
		$c^k$. For example,
		\[
		\begin{aligned}
			&
			b_n(t,X_t^n,\mu_t^n)-b(t,X_t,\mu_t)\\
			=&[b_n-b^k](t,X_t^n,\mu_t^n)+\bigl[	b^k(t,X_t^n,\mu_t^n)
			-b^k(t,X_t,\mu_t)\bigr]
			+[b^k-b](t,X_t,\mu_t),
		\end{aligned}
		\]
		and the same decomposition is used for $\sigma$ and
		$\overline\beta$. Hence,
		\begin{align}
			\mathbb E\int_0^{\theta_R^n}
			\mathfrak E_n(t)\,dt\leq&
			C\mathbb E\int_0^{\theta_R^n}
			\mathfrak e_{a_n,c^k}(t,X_t^n,\mu_t^n)\,dt
			+
			C L_k^2
			\int_0^T
			\left[
			\mathbb E|X_t^n-X_t|^2
			+
			W_2(\mu_t^n,\mu_t)^2
			\right]dt
			\nonumber\\
			&+
			C\mathbb E\int_0^{\theta_R^n}
			\mathfrak e_{c^k,a}(t,X_t,\mu_t)\,dt.
			\label{eq:stopped-three-error}
		\end{align}
		
		By Proposition~\ref{prop:moment-localization-estimates},
		\[
		\mathbb E\int_0^{\theta_R^n}
		\mathfrak e_{a_n,c^k}(t,X_t^n,\mu_t^n)\,dt
		\leq
		C_{\mathrm K}(R)
		\rho_{T,R}(a_n,c^k)^2.
		\]
		By Corollary~\ref{cor:limit-process-coefficient-error},
		\[
		\mathbb E\int_0^{\theta_R^n}
		\mathfrak e_{c^k,a}(t,X_t,\mu_t)\,dt
		\leq
		C_{\mathrm K}^{\mathrm{lim}}(R)
		\rho_{T,R}(c^k,a)^2.
		\]
		Furthermore, \eqref{eq:uniform-law-convergence} gives
		\[
		\begin{aligned}
			&
			L_k^2
			\int_0^T
			\left[
			\mathbb E|X_t^n-X_t|^2
			+
			W_2(\mu_t^n,\mu_t)^2
			\right]dt
			\\
			\leq&
			2TL_k^2
			\|X^n-X\|_{\mathcal S_T^2}^2
			\longrightarrow0
		\end{aligned}
		\]
		for every fixed $k$.
		
		Since
		\[
		\rho_{T,R}(a_n,c^k)
		\leq
		\rho_{T,R}(a_n,a)
		+
		\rho_{T,R}(a,c^k),
		\]
		we obtain, for fixed $k$ and $R$,
		\[
		\limsup_{n\to\infty}
		\rho_{T,R}(a_n,c^k)
		\leq
		\rho_{T,R}(a,c^k).
		\]
		It follows from \eqref{eq:stopped-three-error} that
		\[
		\begin{aligned}
			\limsup_{n\to\infty}
			\mathbb E\int_0^{\theta_R^n}
			\mathfrak E_n(t)\,dt
			\leq
			C\left[
			C_{\mathrm K}(R)
			+
			C_{\mathrm K}^{\mathrm{lim}}(R)
			\right]
			\rho_{T,R}(c^k,a)^2.
		\end{aligned}
		\]
		Since $c^k\to a$ in $d_T$,
		\[
		\rho_{T,R}(c^k,a)\longrightarrow0
		\]
		for every fixed $R$. Letting $k\to\infty$ gives
		\begin{equation}
			\label{eq:stopped-integrand-convergence}
			\mathbb E\int_0^{\theta_R^n}
			\mathfrak E_n(t)\,dt
			\longrightarrow0
		\end{equation}
		for every fixed $R$.
		
		It remains to remove the stopping time. Assumption
		\textnormal{(A2)} implies
		\[
		\mathfrak E_n(t)\leq CM^2
		\]
		for $dt\,d\mathbb P$-almost every $(t,\omega)$. Therefore,
		\[
		\begin{aligned}
			\mathbb E\int_{\theta_R^n}^T
			\mathfrak E_n(t)\,dt
			&\leq
			CM^2T\,\mathbb P(\theta_R^n<T)
			\\
			&\leq
			CM^2T
			\left[
			\mathbb P(\tau_R(X^n)<T)
			+
			\mathbb P(\tau_R(X)<T)
			\right].
		\end{aligned}
		\]
		By \eqref{eq:uniform-moment-closedness} and Markov's inequality,
		\[
		\sup_{n\geq1}
		\mathbb P(\tau_R(X^n)<T)
		+
		\mathbb P(\tau_R(X)<T)
		\leq
		\frac{C}{R^2}.
		\]
		Combining this estimate with
		\eqref{eq:stopped-integrand-convergence} yields
		\[
		\limsup_{n\to\infty}
		\mathbb E\int_0^T\mathfrak E_n(t)\,dt
		\leq
		\frac{C}{R^2}.
		\]
		Letting $R\to\infty$ proves
		\eqref{eq:coefficient-integrand-L2-convergence}. The componentwise
		convergences
		\eqref{eq:drift-integrand-L2}--
		\eqref{eq:jump-integrand-L2}
		follow immediately.
	\end{proof}

	\subsection{Identification of the Limiting Equation}
	
	\begin{proposition}
		\label{prop:integral-term-convergence}
		Under the assumptions of	Proposition~\ref{prop:coefficient-integrand-convergence}, the drift,	Brownian, and compensated Poisson integral processes converge in	$\mathcal S_T^2$. More precisely,
		\begin{equation}
			\label{eq:drift-integral-convergence}
			\mathbb E\left[
			\sup_{0\leq t\leq T}
			\left|
			\int_0^t
			\bigl[
			b_n(s,X_s^n,\mu_s^n)
			-
			b(s,X_s,\mu_s)
			\bigr]\,ds
			\right|^2
			\right]
			\longrightarrow0,
		\end{equation}
		\begin{equation}
			\label{eq:brownian-integral-convergence}
			\mathbb E\left[
			\sup_{0\leq t\leq T}
			\left|
			\int_0^t
			\bigl[
			\sigma_n(s,X_s^n,\mu_s^n)
			-
			\sigma(s,X_s,\mu_s)
			\bigr]\,dW_s
			\right|^2
			\right]
			\longrightarrow0,
		\end{equation}
		and
		\begin{equation}
			\label{eq:poisson-integral-convergence}
			\mathbb E\left[
			\sup_{0\leq t\leq T}
			\left|
			\int_0^t\int_U
			\Gamma_s^n(z)\,
			\widetilde N(ds,dz)
			\right|^2
			\right]
			\longrightarrow0,
		\end{equation}
		where
		\[
		\Gamma_s^n(z)
		:=
		\beta_n(s,X_{s-}^n,\mu_s^n,z)
		-
		\beta(s,X_{s-},\mu_s,z).
		\]
	\end{proposition}
	
	\begin{proof}
		By the Cauchy--Schwarz inequality,
		\[
		\begin{aligned}
			&
			\mathbb E\sup_{0\leq t\leq T}
			\left|
			\int_0^t
			\bigl[
			b_n(s,X_s^n,\mu_s^n)
			-
			b(s,X_s,\mu_s)
			\bigr]\,ds
			\right|^2
			\\
			\leq&
			T\mathbb E\int_0^T
			|b_n(s,X_s^n,\mu_s^n)
			-b(s,X_s,\mu_s)|^2\,ds,
		\end{aligned}
		\]
		so \eqref{eq:drift-integral-convergence} follows from
		\eqref{eq:drift-integrand-L2}.
		
		Similarly, the BDG inequality gives
		\[
		\begin{aligned}
			&
			\mathbb E\sup_{0\leq t\leq T}
			\left|
			\int_0^t
			\bigl[
			\sigma_n(s,X_s^n,\mu_s^n)
			-
			\sigma(s,X_s,\mu_s)
			\bigr]\,dW_s
			\right|^2
			\\
			\leq&
			C\mathbb E\int_0^T
			\|\sigma_n(s,X_s^n,\mu_s^n)
			-\sigma(s,X_s,\mu_s)\|_{\mathrm{HS}}^2\,ds,
		\end{aligned}
		\]
		which proves \eqref{eq:brownian-integral-convergence} by
		\eqref{eq:diffusion-integrand-L2}.
		
		By Lemma~\ref{lem-Measurability},
		\(\Gamma^n\) is
		\(\mathcal P_{\mathrm{pred}}\otimes\mathcal U\)-measurable.
		Moreover,
		\[
		\begin{aligned}
			&\mathbb E\int_0^T\int_U
			|\Gamma_s^n(z)|^2\,\nu(dz)\,ds
			\\
			=&
			\mathbb E\int_0^T
			\|
			\overline\beta_n(s,X_{s-}^n,\mu_s^n)
			-
			\overline\beta(s,X_{s-},\mu_s)
			\|_H^2\,ds
			<\infty.
		\end{aligned}
		\]
		Therefore the BDG inequality for compensated Poisson integrals yields
		\[
		\begin{aligned}
			&
			\mathbb E\sup_{0\leq t\leq T}
			\left|
			\int_0^t\int_U
			\Gamma_s^n(z)\,
			\widetilde N(ds,dz)
			\right|^2
			\\
			\leq&
			C\mathbb E\int_0^T
			\|\overline\beta_n(s,X_{s-}^n,\mu_s^n)
			-\overline\beta(s,X_{s-},\mu_s)\|_H^2\,ds.
		\end{aligned}
		\]
		This proves \eqref{eq:poisson-integral-convergence} by
		\eqref{eq:jump-integrand-L2}.
	\end{proof}

	We can now identify the limiting process.
	
	\begin{theorem}
		\label{thm:closed-graph-solution-relation}
		Let
		\[
		a_n\in\mathfrak A_T^{\mathrm{Lip}},
		\qquad
		a\in\mathfrak A_T,
		\qquad
		x_n,x\in\mathbb R^d.
		\]
		Assume that
		\[
		d_T(a_n,a)\longrightarrow0,
		\qquad
		x_n\longrightarrow x,
		\]
		and that
		\[
		X^{a_n,x_n}\longrightarrow X
		\quad\text{in }\mathcal S_T^2
		\]
		for some $X\in\mathcal S_T^2$. Then
		\[
		X\in\operatorname{Sol}_T(a,x).
		\]
		Equivalently, with $\mu_s=\mathcal L(X_s)$,
		\[
		\begin{aligned}
			X_t
			={}&
			x
			+\int_0^t b(s,X_s,\mu_s)\,ds
			+\int_0^t\sigma(s,X_s,\mu_s)\,dW_s
			\\
			&+
			\int_0^t\int_U
			\beta(s,X_{s-},\mu_s,z)\,
			\widetilde N(ds,dz).
		\end{aligned}
		\]
	\end{theorem}
	
	\begin{proof}
		Set
		\[
		X^n:=X^{a_n,x_n},
		\qquad
		\mu_t^n:=\mathcal L(X_t^n),
		\qquad
		\mu_t:=\mathcal L(X_t).
		\]
		For every $n$,
		\begin{align}
			X_t^n
			={}&
			x_n
			+
			\int_0^t b_n(s,X_s^n,\mu_s^n)\,ds
			+
			\int_0^t\sigma_n(s,X_s^n,\mu_s^n)\,dW_s
			\nonumber\\
			&+
			\int_0^t\int_U
			\beta_n(s,X_{s-}^n,\mu_s^n,z)\,
			\widetilde N(ds,dz).
			\label{eq:approximating-equation-closedness}
		\end{align}

		Because $X$ is adapted and c\`adl\`ag, $X$ is progressively
		measurable and $X_{-}$ is predictable. Moreover,
		\[
		W_2\bigl(\mathcal L(X_t),\mathcal L(X_s)\bigr)^2
		\leq
		\mathbb E|X_t-X_s|^2.
		\]
		Hence $t\mapsto\mu_t$ is Borel measurable.	By Lemma~\ref{lem-Measurability}\textnormal{(ii)--(iii)}, the maps
		\[
		(t,\omega)
		\longmapsto
		b(t,X_t(\omega),\mu_t)
		\]
		and
		\[
		(t,\omega)
		\longmapsto
		\sigma(t,X_t(\omega),\mu_t)
		\]
		are progressively measurable. Moreover,
		\[
		(t,\omega,z)
		\longmapsto
		\beta(t,X_{t-}(\omega),\mu_t,z)
		\]
		is
		\(\mathcal P_{\mathrm{pred}}\otimes\mathcal U\)-measurable and
		satisfies
		\[
		\mathbb E\int_0^T\int_U
		|\beta(t,X_{t-},\mu_t,z)|^2\,\nu(dz)\,dt
		\leq M^2T.
		\]
		Thus all limiting integrals are well defined.

		By Proposition~\ref{prop:integral-term-convergence}, the three
		integral processes in
		\eqref{eq:approximating-equation-closedness} converge in
		$\mathcal S_T^2$ to their counterparts associated with
		$(a,X,\mu)$. Together with
		\[
		X^n\longrightarrow X
		\quad\text{in }\mathcal S_T^2,
		\qquad
		x_n\longrightarrow x,
		\]
		this yields
		\begin{equation}
			\label{eq:limit-equation-S2}
			\begin{aligned}
				X_\cdot
				={}&
				x
				+\int_0^\cdot b(s,X_s,\mu_s)\,ds
				+\int_0^\cdot\sigma(s,X_s,\mu_s)\,dW_s
				\\
				&+
				\int_0^\cdot\int_U
				\beta(s,X_{s-},\mu_s,z)\,
				\widetilde N(ds,dz)
			\end{aligned}
		\end{equation}
		in $\mathcal S_T^2$.
		
		Since both sides admit c\`adl\`ag versions, the identity holds up to
		indistinguishability, simultaneously for all $t\in[0,T]$. Therefore,
		\[
		X\in\operatorname{Sol}_T(a,x).
		\]
	\end{proof}
	
	\begin{corollary}
		\label{cor:closed-core-solution-map}
		For fixed $x\in\mathbb R^d$, let
		\[
		\Phi_{T,x}(c)=X^{c,x},
		\qquad
		c\in\mathfrak A_T^{\mathrm{Lip}}.
		\]
		If
		\[
		c_n\in\mathfrak A_T^{\mathrm{Lip}},
		\qquad
		c_n\longrightarrow a\in\mathfrak A_T
		\quad\text{in }d_T,
		\]
		and
		\[
		\Phi_{T,x}(c_n)\longrightarrow X
		\quad\text{in }\mathcal S_T^2,
		\]
		then
		\[
		X\in\operatorname{Sol}_T(a,x).
		\]
		Thus the closure of the graph of $\Phi_{T,x}$ is contained in the
		graph of the strong-solution relation
		\[
		a\longmapsto\operatorname{Sol}_T(a,x).
		\]
	\end{corollary}
	
	\begin{remark}
		The closed-graph property yields existence once a sequence of core
		solutions is shown to converge. It does not imply uniqueness of the
		limiting equation, which will instead follow from the
		stable-neighborhood property and the Baire category argument.
	\end{remark}

	\section{Generic Strong Well-Posedness}
	\label{sec:generic-strong-wellposedness}
	
	We first isolate the abstract Baire-category mechanism used in the
	proof. The result is included to distinguish the general topological
	extension argument from the stochastic estimates needed to verify its
	hypotheses.
	
	\subsection{A generic canonical-extension principle for the McKean--Vlasov solution relation}
	
	\begin{theorem}
		\label{thm:abstract-generic-canonical-extension}
		Let $(E,d_E)$ and $(Z,d_Z)$ be complete metric spaces, let
		$D\subset E$ be dense, and let $\mathsf X$ be a metric space admitting
		a proper compact exhaustion
		\[
		K_1\subset K_2\subset\cdots,
		\qquad
		K_m\subset\operatorname{int}_{\mathsf X}K_{m+1}
		\quad(m\geq1),
		\qquad
		\cup_{m\geq1}K_m=\mathsf X.
		\]
		Here $C_{\mathrm{loc}}(\mathsf X;Z)$ denotes the space of continuous
		maps from $\mathsf X$ to $Z$, endowed with the topology of uniform
		convergence on compact subsets (equivalently, on each $K_m$ of the
		fixed exhaustion). Suppose that
		\[
		\Phi:D\longrightarrow C_{\mathrm{loc}}(\mathsf X;Z)
		\]
		is a single-valued map and that
		\[
		\mathfrak S:E\times\mathsf X\rightrightarrows Z
		\]
		is a possibly empty and possibly set-valued relation. Let
		$(\varepsilon_n)_{n\geq1}$ be a positive sequence with
		$\varepsilon_n\downarrow0$. Assume the following.
		
		\begin{enumerate}
			\item[\textnormal{(G1)}]
			For every $c\in D$ and $x\in\mathsf X$,
			\[
			\mathfrak S(c,x)=\big\{\Phi(c)(x)\big\}.
			\]
			
			\item[\textnormal{(G2)}]
			For every $c\in D$ and $m,n\in\mathbb N$, there exists
			$r_{m,n}(c)>0$ such that, for every $a\in E$, $x\in K_m$, and
			$z\in\mathfrak S(a,x)$,
			\[
			d_E(a,c)<r_{m,n}(c)
			\]
			imply
			\[
			d_Z\bigl(z,\Phi(c)(x)\bigr)<\varepsilon_n.
			\]
			
			\item[\textnormal{(G3)}]
			If $c_k\in D$, $c_k\to a$ in $E$, $x\in\mathsf X$, and
			\[
			\Phi(c_k)(x)\longrightarrow z
			\qquad\text{in }Z,
			\]
			then
			\[
			z\in\mathfrak S(a,x).
			\]
		\end{enumerate}
		
		For $m,n\in\mathbb N$, define
		\begin{equation}
			\label{eq:abstract-generic-open-dense-set}
			\mathcal G_{m,n}
			:=
			\cup_{c\in D}
			B_{d_E}\left(c,\frac12r_{m,n}(c)\right),
		\end{equation}
		and set
		\begin{equation}
			\label{eq:abstract-generic-residual-set}
			\mathcal R
			:=
			\cap_{m=1}^{\infty}
			\cap_{n=1}^{\infty}
			\mathcal G_{m,n}.
		\end{equation}
		Then $\mathcal R$ is dense and residual in $E$. Moreover, for every
		$a\in\mathcal R$ there exists a unique map
		\[
		\overline\Phi(a)
		\in
		C_{\mathrm{loc}}(\mathsf X;Z)
		\]
		such that the following statements hold.
		
		\begin{enumerate}
			\item[\textnormal{(i)}]
			For every sequence $c_k\in D$ with $c_k\to a$ in $E$ and every
			$m\in\mathbb N$,
			\[
			\sup_{x\in K_m}
			d_Z\bigl(\Phi(c_k)(x),\overline\Phi(a)(x)\bigr)
			\longrightarrow0.
			\]
			
			\item[\textnormal{(ii)}]
			For every $x\in\mathsf X$,
			\[
			\mathfrak S(a,x)=\big\{\overline\Phi(a)(x)\big\}.
			\]
		\end{enumerate}
		Thus $\Phi$ has a unique sequential extension at every point of
		$\mathcal R$, locally uniformly in the parameter $x$.
	\end{theorem}
	
\begin{proof}
	Each \(\mathcal G_{m,n}\) is open and contains \(D\), hence is dense.
	Since \(E\) is complete, the Baire category theorem implies that
	\(\mathcal R\) is dense and residual.
	
	Fix \(a\in\mathcal R\) and a sequence \(c_k\in D\) with
	\(c_k\to a\). Fix \(m\in\mathbb N\). We first show that
	\((\Phi(c_k)|_{K_m})_{k\geq1}\) is Cauchy in \(C(K_m;Z)\).
	Let \(\eta>0\) and choose \(n\) so large that
	\(2\varepsilon_n<\eta\). Since \(a\in\mathcal G_{m,n}\), there exists
	\(c_{m,n}\in D\) such that
	\[
	d_E(a,c_{m,n})
	<
	\frac12 r_{m,n}(c_{m,n}).
	\]
	For all sufficiently large \(k\),
	\[
	d_E(c_k,c_{m,n})
	<
	r_{m,n}(c_{m,n}).
	\]
	By \textnormal{(G1)},
	\(\Phi(c_k)(x)\in\mathfrak S(c_k,x)\), and hence
	\textnormal{(G2)} gives
	\[
	\sup_{x\in K_m}
	d_Z\bigl(\Phi(c_k)(x),\Phi(c_{m,n})(x)\bigr)
	<
	\varepsilon_n
	\]
	for all sufficiently large \(k\). Therefore, for sufficiently large
	\(k,\ell\),
	\[
	\sup_{x\in K_m}
	d_Z\bigl(\Phi(c_k)(x),\Phi(c_\ell)(x)\bigr)
	<
	2\varepsilon_n
	<
	\eta.
	\]
	
	Since \(Z\) is complete, \(C(K_m;Z)\) is complete under the uniform
	metric. Thus there exists
	\[
	\Phi_m^a\in C(K_m;Z)
	\]
	such that
	\[
	\Phi(c_k)|_{K_m}
	\longrightarrow
	\Phi_m^a
	\]
	uniformly. If \(m\leq\ell\), then
	\(\Phi_m^a=\Phi_\ell^a|_{K_m}\), since both are limits of
	\(\Phi(c_k)|_{K_m}\). Hence the family
	\((\Phi_m^a)_{m\geq1}\) defines a map
	\[
	\overline\Phi(a):\mathsf X\longrightarrow Z.
	\]
	
	This map is continuous. Indeed, if \(x\in\mathsf X\), choose \(j\)
	with \(x\in K_j\). By the proper-exhaustion condition,
	\[
	x\in\operatorname{int}_{\mathsf X}K_{j+1},
	\]
	and on this neighborhood
	\(\overline\Phi(a)=\Phi_{j+1}^a\). Therefore
	\[
	\overline\Phi(a)\in C_{\mathrm{loc}}(\mathsf X;Z).
	\]
	
	Fix \(x\in\mathsf X\) and choose \(m\) with \(x\in K_m\). Then
	\[
	\Phi(c_k)(x)
	\longrightarrow
	\overline\Phi(a)(x).
	\]
	By \textnormal{(G3)},
	\[
	\overline\Phi(a)(x)\in\mathfrak S(a,x),
	\]
	so \(\mathfrak S(a,x)\) is nonempty.
	
	To prove uniqueness, let
	\(y,z\in\mathfrak S(a,x)\), and fix \(m\) such that \(x\in K_m\).
	For every \(n\), since \(a\in\mathcal G_{m,n}\), choose
	\(c_{m,n}\in D\) satisfying
	\[
	d_E(a,c_{m,n})
	<
	\frac12r_{m,n}(c_{m,n}).
	\]
	Assumption \textnormal{(G2)} then yields
	\[
	d_Z\bigl(y,\Phi(c_{m,n})(x)\bigr)
	<
	\varepsilon_n,
	\qquad
	d_Z\bigl(z,\Phi(c_{m,n})(x)\bigr)
	<
	\varepsilon_n.
	\]
	Hence
	\[
	d_Z(y,z)<2\varepsilon_n.
	\]
	Letting \(n\to\infty\) gives \(y=z\). Thus
	\[
	\mathfrak S(a,x)
	=
	\{\overline\Phi(a)(x)\}.
	\]
	
	Finally, let \((\widetilde c_k)\subset D\) be any other sequence with
	\(\widetilde c_k\to a\). Repeating the preceding Cauchy argument shows
	that \(\Phi(\widetilde c_k)\) converges locally uniformly to some
	\(\widetilde\Phi\in C_{\mathrm{loc}}(\mathsf X;Z)\). By
	\textnormal{(G3)},
	\[
	\widetilde\Phi(x)\in\mathfrak S(a,x)
	\qquad
	\text{for every }x\in\mathsf X.
	\]
	The singleton property therefore gives
	\[
	\widetilde\Phi=\overline\Phi(a).
	\]
	This proves the sequence-independent convergence in
	\textnormal{(i)}.
\end{proof}
	
	\begin{remark}
		\label{rem:abstract-extension-role}
		Theorem~\ref{thm:abstract-generic-canonical-extension} is used only as
		a topological extension mechanism. The substantive stochastic work in
		the present paper is the verification of \textnormal{(G2)} by the
		occupation-stable, set-valued estimate and of \textnormal{(G3)} by the
		closedness theorem for low-regularity jump equations.
	\end{remark}

	\begin{proof}[Proof of Theorem~\ref{thm:generic-main-fixed-horizon}]
	We apply Theorem~\ref{thm:abstract-generic-canonical-extension} with
	\[
	E=\mathfrak A_T,
	\qquad
	D=\mathfrak A_T^{\mathrm{Lip}},
	\qquad
	\mathsf X=\mathbb R^d,
	\qquad
	K_m=\overline B_m,
	\qquad
	Z=\mathcal S_T^2.
	\]
	The closed balls form a proper compact exhaustion because
	\[
	\overline B_m\subset B_{m+1}
	=\operatorname{int}_{\mathbb R^d}\overline B_{m+1}.
	\]
	Define also
	\[
	\Phi(c)(x):=X^{c,x},
	\qquad
	\mathcal S(a,x):=\operatorname{Sol}_T(a,x).
	\]
	Corollary~\ref{cor:core-initial-state-continuity} gives
	$\Phi(c)\in C_{\mathrm{loc}}(\mathbb R^d;\mathcal S_T^2)$.
	Proposition~\ref{prop:wellposed-core} verifies
	\textnormal{(G1)}. Corollary~\ref{cor:uniform-stability-compact-initial-set}
	verifies \textnormal{(G2)} with the sequence
	$\varepsilon_n$ from \eqref{eq:uniform-generic-epsilon}. Finally,
	Theorem~\ref{thm:closed-graph-solution-relation}, applied with a
	constant initial-value sequence, verifies \textnormal{(G3)}.
	
	For later use, write the open dense sets in this application as
	\begin{equation}
		\label{eq:uniform-Gmn-definition}
		G_{m,n}
		:=
		\cup_{c\in\mathfrak A_T^{\mathrm{Lip}}}
		B_{d_T}\left(c,\frac12r_{m,n}(c)\right),
	\end{equation}
	and define
	\begin{equation}
		\label{eq:initial-state-independent-residual-set}
		\mathfrak R_T
		:=
		\cap_{m=1}^{\infty}
		\cap_{n=1}^{\infty}
		G_{m,n}.
	\end{equation}
	The abstract theorem shows that $\mathfrak R_T$ is dense and residual
	and supplies a canonical map
	\[
	\Phi_T(a)\in C_{\mathrm{loc}}(\mathbb R^d;\mathcal S_T^2),
	\qquad
	\Phi_T(a)(x)=X^{a,x},
	\]
	for every $a\in\mathfrak R_T$. It also gives, for every
	$a_k\in\mathfrak A_T^{\mathrm{Lip}}$ with $d_T(a_k,a)\to0$,
	\[
	\sup_{x\in\overline B_m}
	\|X^{a_k,x}-X^{a,x}\|_{\mathcal S_T^2}
	\longrightarrow0
	\]
	for every $m$, independently of the approximating sequence, and
	\[
	\operatorname{Sol}_T(a,x)=\{X^{a,x}\}
	\qquad
	\text{for every }x\in\mathbb R^d.
	\]
	This proves Theorem~\ref{thm:generic-main-fixed-horizon}.
\end{proof}
	\begin{remark}
		\label{rem:initial-state-independent-generic-set}
		The residual set $\mathfrak R_T$ depends on the time horizon $T$ but
		not on the initial state. Thus no countable intersection over a dense
		subset of $\mathbb R^d$ is required. The conclusion is stronger than
		pointwise convergence in the initial state: every Lipschitz-core
		approximation converges locally uniformly in $x$, with values in
		$\mathcal S_T^2$.
	\end{remark}
	
	\subsection{The global-time extension}
	
	We now pass to the global-time coefficient space. For integers
	$1\leq m<\ell$, define on the Lipschitz core
	\[
	r_{m,\ell}^{\mathrm{Lip}}c
	:=
	c|_{[0,m]}.
	\]
	By the definition of the finite-horizon metrics,
	\begin{equation}
		\label{eq:restriction-metric-contraction}
		d_m\bigl(
		r_{m,\ell}^{\mathrm{Lip}}c,
		r_{m,\ell}^{\mathrm{Lip}}\widetilde c
		\bigr)
		\leq
		d_\ell(c,\widetilde c),
		\qquad
		c,\widetilde c\in\mathfrak A_\ell^{\mathrm{Lip}}.
	\end{equation}
	Hence $r_{m,\ell}^{\mathrm{Lip}}$ extends uniquely to a continuous map
	\[
	r_{m,\ell}:\mathfrak A_\ell\longrightarrow\mathfrak A_m.
	\]
	
	\begin{proposition}
		\label{prop:restriction-projective-completeness}
		The restriction maps satisfy
		\[
		r_{j,m}\circ r_{m,\ell}
		=
		r_{j,\ell},
		\qquad
		1\leq j<m<\ell,
		\]
		and the projective-limit space
		$\mathfrak A_{\mathrm{loc}}$ defined in
		\eqref{eq:local-coefficient-projective-limit} is complete under
		$d_{\mathrm{loc}}$.
	\end{proposition}
	
	\begin{proof}
		The compatibility identity holds on the Lipschitz core and therefore
		on the completions by continuity.
		
		Let $(a^k)$ be $d_{\mathrm{loc}}$-Cauchy. For every $m$,
		$(\pi_m(a^k))$ is Cauchy in the complete space $\mathfrak A_m$;
		write
		\[
		\pi_m(a^k)\longrightarrow a_m.
		\]
		For $m<\ell$, continuity of the restriction maps gives
		\[
		r_{m,\ell}(a_\ell)
		=
		\lim_{k\to\infty}
		r_{m,\ell}\bigl(\pi_\ell(a^k)\bigr)
		=
		\lim_{k\to\infty}\pi_m(a^k)
		=
		a_m.
		\]
		Thus $a=(a_m)_{m\geq1}\in\mathfrak A_{\mathrm{loc}}$.
		For each fixed  \(m\), we have
		\[
		d_m(a^k,a)\longrightarrow0.
		\]
		Splitting the defining series for \(d_{\mathrm{loc}}\) into finitely
		many leading terms and a uniformly small summable tail then yields
		\[
		d_{\mathrm{loc}}(a^k,a)\longrightarrow0.
		\]
		
	\end{proof}

	\begin{proposition}
		\label{prop:global-lipschitz-core-density}
		The set
		\[
		\mathfrak A_{\mathrm{loc}}^{\mathrm{Lip}}
		\]
		is non-empty and dense in
		\((\mathfrak A_{\mathrm{loc}},d_{\mathrm{loc}})\).
		More precisely, for every
		\(a\in\mathfrak A_{\mathrm{loc}}\) there exists a sequence
		\[
		c^{(N)}\in\mathfrak A_{\mathrm{loc}}^{\mathrm{Lip}}
		\]
		such that
		\[
		d_{\mathrm{loc}}(c^{(N)},a)\longrightarrow0.
		\]
	\end{proposition}
	
	\begin{proof}
		The constant coefficient
		\[
		a_\ast:=(0,\sqrt{2\lambda}\,I_d,0)
		\]
		defines an element of
		\(\mathfrak A_{\mathrm{loc}}^{\mathrm{Lip}}\), so the locally
		Lipschitz core is non-empty.
		
		Fix \(a=(a_m)_{m\geq1}\in\mathfrak A_{\mathrm{loc}}\) and
		\(\varepsilon>0\). Choose \(N\) so large that
		\[
		\sum_{m>N}2^{-m}=2^{-N}<\frac{\varepsilon}{2}.
		\]
		Since \(\mathfrak A_N^{\mathrm{Lip}}\) is dense in
		\(\mathfrak A_N\), choose
		\[
		\widetilde a_N\in\mathfrak A_N^{\mathrm{Lip}}
		\]
		such that
		\[
		d_N(\widetilde a_N,a_N)<\frac{\varepsilon}{2}.
		\]
		Choose an admissible Lipschitz representative of \(\widetilde a_N\), and define
		a coefficient on \([0,\infty)\) by
		\[
		c^{(N)}(t,x,\mu)
		:=
		\begin{cases}
			\widetilde a_N(t,x,\mu),&0\leq t\leq N,\\
			a_\ast(t,x,\mu),&t>N.
		\end{cases}
		\]
		No continuity in the time variable is required. Therefore, for every
		\(m\), the restriction of \(c^{(N)}\) to \([0,m]\) belongs to
		\(\mathfrak A_m^{\mathrm{Lip}}\), and hence
		\[
		c^{(N)}\in\mathfrak A_{\mathrm{loc}}^{\mathrm{Lip}}.
		\]
		
		For \(m\leq N\), the contraction property of the restriction maps
		gives
		\[
		d_m\bigl(\pi_m(c^{(N)}),a_m\bigr)
		\leq
		d_N(\widetilde a_N,a_N).
		\]
		Consequently,
		\[
		\begin{aligned}
			d_{\mathrm{loc}}(c^{(N)},a)
			&\leq
			\sum_{m=1}^{N}
			2^{-m}
			\left(
			1\wedge d_N(\widetilde a_N,a_N)
			\right)
			+
			\sum_{m>N}2^{-m}
			\\
			&\leq
			d_N(\widetilde a_N,a_N)+2^{-N}
			\\
			&<\varepsilon.
		\end{aligned}
		\]
		This proves density.
	\end{proof}
	
	\begin{proposition}
		\label{prop:global-core-wellposedness}
		Let
		\[
		c\in\mathfrak A_{\mathrm{loc}}^{\mathrm{Lip}}
		\]
		and \(x\in\mathbb R^d\). Then there exists a unique global strong
		solution
		\[
		X^{c,x}\in\mathcal S_{\mathrm{loc}}^2.
		\]
		For every integer \(m\geq1\),
		\[
		X^{c,x}|_{[0,m]}
		=
		X^{\pi_m(c),x}
		\]
		up to indistinguishability.
	\end{proposition}
	
	\begin{proof}
		For each \(m\geq1\), Proposition~\ref{prop:wellposed-core} gives a
		unique strong solution
		\[
		X^{(m)}:=X^{\pi_m(c),x}\in\mathcal S_m^2.
		\]
		If \(m<\ell\), then \(X^{(\ell)}|_{[0,m]}\) is a strong solution
		associated with
		\[
		r_{m,\ell}(\pi_\ell(c))=\pi_m(c).
		\]
		By Proposition~\ref{prop:rep-equivalence} and pathwise uniqueness on
		the Lipschitz core,
		\[
		X^{(\ell)}|_{[0,m]}=X^{(m)}
		\]
		up to indistinguishability. Choosing versions that agree on overlapping time intervals gives an adapted c\`adl\`ag process
		\[
		X^{c,x}\in\mathcal S_{\mathrm{loc}}^2.
		\]
		It satisfies the equation on every finite time interval. Any other
		global solution agrees with \(X^{c,x}\) on every interval
		\([0,m]\), which proves global pathwise uniqueness.
	\end{proof}

	\begin{proposition}
		\label{prop:compatible-global-representatives}
		Every
		\[
		a=(a_m)_{m\geq1}\in\mathfrak A_{\mathrm{loc}}
		\]
		admits admissible representatives $\widehat a_m$ satisfying
		\[
		\widehat a_\ell(t,x,\mu)
		=
		\widehat a_m(t,x,\mu),
		\qquad
		0\leq t\leq m<\ell.
		\]
		They therefore define a jointly measurable coefficient on
		$[0,\infty)$, unique up to local null-set equivalence.
		
		Moreover, for every $m\geq1$ and $\gamma_m\in\mathfrak A_m$, there exists
		$\widehat a\in\mathfrak A_{\mathrm{loc}}$ such that
		\[
		\pi_m(\widehat a)=\gamma_m,
		\qquad
		d_{\mathrm{loc}}(\widehat a,a)
		\leq
		d_m\bigl(\gamma_m,\pi_m(a)\bigr).
		\]
	\end{proposition}
	
\begin{proof}
	Choose an admissible representative \(\widehat a_1\) of \(a_1\).
	Suppose that \(\widehat a_m\) has been chosen, and let
	\(\widetilde a_{m+1}\) be an admissible representative of \(a_{m+1}\).
	Since
	\[
	r_{m,m+1}(a_{m+1})=a_m,
	\]
	Theorem~\ref{thm:completion-representation} implies that
	\(\widetilde a_{m+1}|_{[0,m]}\) and \(\widehat a_m\) agree, for every
	\(\mu\in\mathcal P_2(\mathbb R^d)\), outside a Borel
	\((dt,dx)\)-null set. Replacing
	\(\widetilde a_{m+1}\) on this null set by \(\widehat a_m\) therefore
	produces an admissible representative \(\widehat a_{m+1}\) satisfying
	\[
	\widehat a_{m+1}|_{[0,m]}=\widehat a_m.
	\]
	Induction yields the required compatible family. Any two such families
	represent the same element \(a_m\) on every finite interval \([0,m]\).
	Hence Theorem~\ref{thm:completion-representation} also gives uniqueness
	up to local \((dt,dx)\)-null-set equivalence.
	
	We now prove the second assertion. Choose compatible representatives of
	\(a\), denoted again by \(a(t,x,\mu)\), and let
	\(\widetilde\gamma_m\) be an admissible representative of
	\(\gamma_m\). Define
	\[
	c(t,x,\mu)
	:=
	\begin{cases}
		\widetilde\gamma_m(t,x,\mu), & 0\leq t\leq m,\\
		a(t,x,\mu), & t>m.
	\end{cases}
	\]
	
	We claim that \(c|_{[0,j]}\) represents an element of
	\(\mathfrak A_j\) for every \(j\geq1\). If \(j\leq m\), this follows
	directly from
	\[
	c|_{[0,j]}
	=
	r_{j,m}(\gamma_m).
	\]
	If \(j>m\), choose
	\[
	\gamma_m^{(k)}\in\mathfrak A_m^{\mathrm{Lip}},
	\qquad
	\widehat{\gamma}_j^{(k)}\in\mathfrak A_j^{\mathrm{Lip}},
	\]
	such that
	\[
	d_m(\gamma_m^{(k)},\gamma_m)\to0,
	\qquad
	d_j(\widehat{\gamma}_j^{(k)},a_j)\to0.
	\]
	Define
	\[
	c_j^{(k)}(t,x,\mu)
	:=
	\begin{cases}
		\gamma_m^{(k)}(t,x,\mu), & 0\leq t\leq m,\\
		\widehat{\gamma}_j^{(k)}(t,x,\mu), & m<t\leq j.
	\end{cases}
	\]
	Since no continuity in the time variable is required,
	\(c_j^{(k)}\in\mathfrak A_j^{\mathrm{Lip}}\). Moreover, from the
	definition of the metrics,
	\[
	d_j\bigl(c_j^{(k)},c|_{[0,j]}\bigr)
	\leq
	d_m(\gamma_m^{(k)},\gamma_m)
	+
	d_j(\widehat{\gamma}_j^{(k)},a_j)
	\longrightarrow0.
	\]
	Thus \(c|_{[0,j]}\) represents an element of \(\mathfrak A_j\).
	
	The family
	\[
	\bigl(c|_{[0,j]}\bigr)_{j\geq1}
	\]
	is compatible under restriction and therefore defines an element
	\(\widehat a\in\mathfrak A_{\mathrm{loc}}\). By construction,
	\[
	\pi_m(\widehat a)=\gamma_m.
	\]
	Furthermore, for every \(j\geq1\),
	\[
	d_j\bigl(\pi_j(\widehat a),\pi_j(a)\bigr)
	\leq
	d_m\bigl(\gamma_m,\pi_m(a)\bigr):
	\]
	for \(j\leq m\), this follows from the contraction property of the
	restriction maps, while for \(j>m\) the two representatives coincide
	on \((m,j]\). Hence
	\[
	\begin{aligned}
		d_{\mathrm{loc}}(\widehat a,a)
		&=
		\sum_{j=1}^{\infty}
		2^{-j}
		\left(
		1\wedge
		d_j\bigl(\pi_j(\widehat a),\pi_j(a)\bigr)
		\right)
		\\
		&\leq
		\sum_{j=1}^{\infty}
		2^{-j}
		\left(
		1\wedge
		d_m\bigl(\gamma_m,\pi_m(a)\bigr)
		\right)
		\\
		&\leq
		d_m\bigl(\gamma_m,\pi_m(a)\bigr).
	\end{aligned}
	\]
	This proves the claim.
\end{proof}
	
	For every $j,m,n\geq1$, let
	\[
	G_{m,n}^{[j]}\subset\mathfrak A_j
	\]
	denote the open dense set defined by
	\eqref{eq:uniform-Gmn-definition} for the time horizon $T=j$, the
	initial-state ball $\overline B_m$, and the  index $n$.
	Define
	\begin{equation}
		\label{eq:lifted-global-open-set}
		\mathcal G_{j,m,n}
		:=
		\pi_j^{-1}\bigl(G_{m,n}^{[j]}\bigr).
	\end{equation}
	
\begin{lemma}
	\label{lem:global-lifted-open-dense}
	Each \(\mathcal G_{j,m,n}\) is open and dense in
	\(\mathfrak A_{\mathrm{loc}}\).
\end{lemma}

\begin{proof}
	Since \(G_{m,n}^{[j]}\) is open in \(\mathfrak A_j\) and
	\[
	\pi_j:\mathfrak A_{\mathrm{loc}}\longrightarrow\mathfrak A_j
	\]
	is continuous, the set
	\[
	\mathcal G_{j,m,n}
	=
	\pi_j^{-1}\bigl(G_{m,n}^{[j]}\bigr)
	\]
	is open in \(\mathfrak A_{\mathrm{loc}}\).
	
	To prove density, let
	\[
	a\in\mathfrak A_{\mathrm{loc}},
	\qquad
	\delta>0.
	\]
	Since \(G_{m,n}^{[j]}\) is dense in \(\mathfrak A_j\), choose
	\[
	\gamma_j\in G_{m,n}^{[j]}
	\]
	such that
	\[
	d_j\bigl(\gamma_j,\pi_j(a)\bigr)<\delta.
	\]
	By Proposition~\ref{prop:compatible-global-representatives}, there exists
	\[
	\widehat a\in\mathfrak A_{\mathrm{loc}}
	\]
	such that
	\[
	\pi_j(\widehat a)=\gamma_j
	\]
	and
	\[
	d_{\mathrm{loc}}(\widehat a,a)
	\leq
	d_j\bigl(\gamma_j,\pi_j(a)\bigr)
	<
	\delta.
	\]
	Since
	\[
	\pi_j(\widehat a)
	=
	\gamma_j
	\in
	G_{m,n}^{[j]},
	\]
	we have
	\[
	\widehat a
	\in
	\pi_j^{-1}\bigl(G_{m,n}^{[j]}\bigr)
	=
	\mathcal G_{j,m,n}.
	\]
Since \(a\in\mathfrak A_{\mathrm{loc}}\) and \(\delta>0\) were arbitrary, \(\mathcal G_{j,m,n}\) is dense in \(\mathfrak A_{\mathrm{loc}}\).
\end{proof}
	
	Define
	\begin{equation}
		\label{eq:global-residual-set}
		\mathfrak R^{\mathrm{loc}}
		:=
		\cap_{j=1}^{\infty}
		\cap_{m=1}^{\infty}
		\cap_{n=1}^{\infty}
		\mathcal G_{j,m,n}.
	\end{equation}
	
	\begin{proof}[Proof of Corollary~\ref{cor:global-generic-main}]
		By Proposition~\ref{prop:restriction-projective-completeness},
		$\mathfrak A_{\mathrm{loc}}$ is complete. Hence
		Lemma~\ref{lem:global-lifted-open-dense} and the Baire category
		theorem show that $\mathfrak R^{\mathrm{loc}}$ is dense and
		residual.
		
		Fix
		\[
		a\in\mathfrak R^{\mathrm{loc}}.
		\]
		By Proposition~\ref{prop:compatible-global-representatives}, choose
		admissible representatives
		\[
		\widehat a_j
		\quad\text{of}\quad
		\pi_j(a),
		\qquad j\geq1,
		\]
		such that, whenever $1\leq j<\ell$,
		\begin{equation}
			\label{eq:compatible-global-coefficient-representatives}
			\widehat a_\ell|_{[0,j]}=\widehat a_j
		\end{equation}
		pointwise on
		\[
		[0,j]\times\mathbb R^d\times\mathcal P_2(\mathbb R^d).
		\]
		
		For every fixed $j\geq1$, the definition
		\eqref{eq:global-residual-set} gives
		\[
		a\in\mathcal G_{j,m,n}
		\qquad
		\text{for all }m,n\geq1.
		\]
		Using \eqref{eq:lifted-global-open-set}, we obtain
		\[
		\pi_j(a)\in G_{m,n}^{[j]}
		\qquad
		\text{for all }m,n\geq1.
		\]
		Consequently, by the finite-horizon definition
		\eqref{eq:initial-state-independent-residual-set}, applied with
		$T=j$,
		\begin{equation}
			\label{eq:global-projection-in-finite-residual-set}
			\pi_j(a)
			\in
			\cap_{m=1}^{\infty}
			\cap_{n=1}^{\infty}
			G_{m,n}^{[j]}
			=
			\mathfrak R_j.
		\end{equation}
		
		Fix an arbitrary $x\in\mathbb R^d$. By
		Theorem~\ref{thm:generic-main-fixed-horizon}, for every $j\geq1$
		there exists a unique strong solution
		\[
		X^{(j),x}
		\in
		\operatorname{Sol}_j(\widehat a_j,x)
		\subset
		\mathcal S_j^2.
		\]
		By Proposition~\ref{prop:rep-equivalence}, this solution set depends
		only on $\pi_j(a)\in\mathfrak A_j$ and not on the particular
		admissible representative $\widehat a_j$.
		
		Let $1\leq j<\ell$. Restricting the equation satisfied by
		$X^{(\ell),x}$ to $[0,j]$ and using
		\eqref{eq:compatible-global-coefficient-representatives}, we obtain
		\[
		X^{(\ell),x}|_{[0,j]}
		\in
		\operatorname{Sol}_j(\widehat a_j,x).
		\]
		Finite-horizon pathwise uniqueness from
		Theorem~\ref{thm:generic-main-fixed-horizon} therefore gives
		\begin{equation}
			\label{eq:compatible-global-solutions-all-initial}
			X^{(\ell),x}|_{[0,j]}=X^{(j),x}
		\end{equation}
		up to indistinguishability.
		
		Since the collection of pairs $(j,\ell)$ with $1\leq j<\ell$ is
		countable, we may choose versions such that
		\eqref{eq:compatible-global-solutions-all-initial} holds
		simultaneously outside a single $\mathbb P$-null set. Define
		\[
		X_t^{a,x}:=X_t^{(j),x}
		\qquad
		\text{for any integer }j\geq t.
		\]
		The compatibility shows that this definition is independent of $j$.
		The resulting process is adapted and c\`adl\`ag, and
		\[
		X^{a,x}|_{[0,j]}=X^{(j),x}\in\mathcal S_j^2
		\qquad
		\text{for every }j\geq1.
		\]
		Thus $X^{a,x}\in\mathcal S_{\mathrm{loc}}^2$, and the compatible
		representatives $(\widehat a_j)_{j\geq1}$ show that it satisfies the
		global equation on every finite time interval.
		
		If $Y\in\mathcal S_{\mathrm{loc}}^2$ is any other global strong
		solution associated with $a$ and initial value $x$, then for every
		$j\geq1$,
		\[
		Y|_{[0,j]}\in\operatorname{Sol}_j(\widehat a_j,x).
		\]
		Theorem~\ref{thm:generic-main-fixed-horizon} gives
		\[
		Y|_{[0,j]}=X^{a,x}|_{[0,j]}
		\]
		up to indistinguishability. Taking the countable intersection of
		the corresponding probability-one events proves global pathwise
		uniqueness. Since $x$ was arbitrary and
		$\mathfrak R^{\mathrm{loc}}$ is independent of $x$, existence and
		uniqueness hold simultaneously for all deterministic initial states.
		
		We next prove continuity with respect to the initial state. Fix a
		finite $T>0$ and choose an integer $j\geq T$. By
		Theorem~\ref{thm:generic-main-fixed-horizon}, the map
		\[
		x\longmapsto X^{(j),x}
		\]
		is continuous from $\mathbb R^d$ to $\mathcal S_j^2$. Since the
		restriction map from $\mathcal S_j^2$ to $\mathcal S_T^2$ is a
		contraction and
		\[
		X^{a,x}|_{[0,T]}=X^{(j),x}|_{[0,T]},
		\]
		the map $x\mapsto X^{a,x}|_{[0,T]}$ is continuous with values in
		$\mathcal S_T^2$.
		
		Finally, let
		\[
		a_k\in\mathfrak A_{\mathrm{loc}}^{\mathrm{Lip}},
		\qquad
		d_{\mathrm{loc}}(a_k,a)\longrightarrow0.
		\]
		Fix a finite $T>0$, an integer $j\geq T$, and $m\in\mathbb N$.
		The continuity of $\pi_j$ gives
		\[
		d_j\bigl(\pi_j(a_k),\pi_j(a)\bigr)\longrightarrow0.
		\]
		By Proposition~\ref{prop:global-core-wellposedness},
		\[
		X^{a_k,x}|_{[0,j]}=X^{\pi_j(a_k),x},
		\]
		while the construction above gives
		\[
		X^{a,x}|_{[0,j]}=X^{\pi_j(a),x}.
		\]
		Applying
		Theorem~\ref{thm:generic-main-fixed-horizon}\textnormal{(iii)} on
		$[0,j]$, we obtain
		\[
		\sup_{x\in\overline B_m}
		\|X^{a_k,x}-X^{a,x}\|_{\mathcal S_j^2}
		\longrightarrow0.
		\]
		Since the $\mathcal S_T^2$-norm is bounded by the
		$\mathcal S_j^2$-norm, this proves
		\eqref{eq:global-uniform-initial-state-convergence}. The independence
		of the approximating sequence follows from the corresponding
		finite-horizon assertion.
	\end{proof}
	
	\appendix
	
	\section{Global convex regularization of the Aleksandrov test function}
	\label{app:aleksandrov-regularization}
	
	This appendix constructs the smooth global test functions used in
	Section~\ref{sec:krylov-ito-levy}. The only non-elementary analytic
	input is Theorem~2 of Krylov's parabolic convex-potential construction
	\cite{Krylov-1976}; see also \cite[Chapter~2]{Krylov-1980}. After
	time reversal, that theorem supplies a bounded spatially convex
	potential, monotone in time, whose standard space--time regularizations
	satisfy the determinant-weighted differential inequality used below.
	We isolate explicitly which properties are imported from the classical
	theorem; the global convex extension and the final regularization are
	proved here. The remaining steps use standard facts from
	finite-dimensional convex analysis and distribution theory; see
	\cite[Sections~10 and~23]{Rockafellar-1970} and
	\cite[Chapter~4]{Hormander-2003}.
	
	\begin{lemma}
		\label{lem:global-aleksandrov-regularization}
		Let $T,R>0$, and let
		\[
		f\in C_c^\infty((0,T)\times B_R),
		\qquad f\geq0.
		\]
		Let $\widetilde f$ denote the zero extension of $f$ to
		$\mathbb R\times\mathbb R^d$. Then there exist constants
		$\kappa_d>0$ and $C_{d,T,R}>0$, a sequence
		$\delta_n\downarrow0$, and functions
		\[
		u^n\in C^\infty(\mathbb R\times\mathbb R^d),
		\qquad
		f^n\in C_c^\infty(\mathbb R\times\mathbb R^d),
		\qquad
		f^n\geq0,
		\]
		such that the following properties hold:
		\begin{enumerate}
			\item[\textnormal{(i)}]
			for every $t\in\mathbb R$, the function
			$x\mapsto u^n(t,x)$ is convex on $\mathbb R^d$;
			
			\item[\textnormal{(ii)}]
			$f^n\to\widetilde f$ uniformly on
			$\mathbb R\times\mathbb R^d$;
			
			\item[\textnormal{(iii)}]
			for every symmetric non-negative definite matrix
			$A\in\mathbb R^{d\times d}$,
			\begin{equation}
				\label{eq:appendix-smoothed-aleksandrov-inequality}
				\partial_tu^n(t,x)
				+
				\operatorname{tr}\!\left(A D^2u^n(t,x)\right)
				\geq
				\kappa_d(\det A)^{1/(d+1)}f^n(t,x)
			\end{equation}
			for all $(t,x)\in[0,T]\times B_R$;
			
			\item[\textnormal{(iv)}]
			uniformly in $n$,
			\begin{equation}
				\label{eq:appendix-global-growth-bound}
				|u^n(t,x)|
				\leq
				C_{d,T,R}
				\|f\|_{L^{d+1}((0,T)\times B_R)}
				(1+|x|),
			\end{equation}
			and
			\begin{equation}
				\label{eq:appendix-global-gradient-bound}
				|\nabla u^n(t,x)|
				\leq
				C_{d,T,R}
				\|f\|_{L^{d+1}((0,T)\times B_R)}
			\end{equation}
			for all $(t,x)\in\mathbb R\times\mathbb R^d$.
		\end{enumerate}
	\end{lemma}
	
	\begin{proof}
		Set
		\[
		F:=\|f\|_{L^{d+1}((0,T)\times B_R)}.
		\]

		\textbf{Step 1:}
		Apply Theorem~2 of \cite{Krylov-1976} on the enlarged cylinder
		\[
		(0,T)\times B_{R+2}
		\]
		to the time-reversed zero extension of
		$(s,x)\mapsto f(T-s,x)$, and then return to the original time variable
		$t=T-s$. In the notation of that theorem, the potential is spatially
		convex and monotone in time, its regularizations satisfy the
		determinant-weighted inequality for every constant symmetric
		non-negative definite matrix, and its size is controlled by the
		$L^{d+1}$-norm of the source. The potential is locally integrable by
		boundedness, so its standard regularizations converge to it in
		$L^1_{\mathrm{loc}}$, while the regularized sources converge to the
		original source in $L^{d+1}_{\mathrm{loc}}$. Testing the regularized
		inequalities against non-negative compactly supported smooth functions
		and passing to the limit therefore gives a bounded jointly Borel function
		\[
		v:[0,T]\times B_{R+2}\longrightarrow\mathbb R
		\]
		such that, for every $t\in[0,T]$, the map $x\mapsto v(t,x)$ is finite
		and convex, the map $t\mapsto v(t,x)$ is non-decreasing for every $x$,
		and, for every constant symmetric non-negative definite matrix $A$,
		\begin{equation}
			\label{eq:appendix-local-distributional-inequality}
			\partial_tv+\operatorname{tr}(A D^2v)
			\geq
			\kappa_d(\det A)^{1/(d+1)}f
		\end{equation}
		in the sense of distributions on
		$(0,T)\times B_{R+1}$. Moreover,
		\begin{equation}
			\label{eq:appendix-local-sup-bound}
			\|v\|_{L^\infty((0,T)\times B_{R+2})}
			\leq
			C_{d,T,R}F.
		\end{equation}
		We choose the endpoint values of $v$ to be its one-sided monotone
		limits. This convention does not alter the distributional inequality.
		
		A standard interior estimate for convex functions implies that $v$
		is uniformly Lipschitz on the smaller ball $B_{R+1}$. More
		precisely, using the local Lipschitz property and the existence of
		supporting hyperplanes for finite convex functions
		\cite[Sections~10 and~23]{Rockafellar-1970}, and enlarging
		$C_{d,T,R}$ if necessary, we have
		\begin{equation}
			\label{eq:appendix-local-subgradient-bound}
			\sup_{t\in[0,T]}
			\sup_{x\in B_{R+1}}
			\sup_{p\in\partial v(t,x)}
			|p|
			\leq L,
			\qquad
			L:=C_{d,T,R}F.
		\end{equation}

		\textbf{Step 2:}
		Let
		\[
		\mathcal Y:=\mathbb Q^d\cap B_{R+1},
		\qquad
		\mathcal P:=\mathbb Q^d\cap \overline B_L.
		\]
		For $(p,q)\in\mathcal P\times\mathbb Q$, define
		\[
		I_{p,q}
		:=
		\left\{
		t\in[0,T]:
		q+p\cdot y\leq v(t,y)
		\text{ for every }y\in\mathcal Y
		\right\}.
		\]
		Since $\mathcal Y$ is countable and
		$t\mapsto v(t,y)$ is Borel for every $y\in\mathcal Y$,
		each set $I_{p,q}$ is Borel. For every fixed $t$, the admissible
		family is non-empty: one may take $p=0$ and any rational number
		\[
		q<\inf_{y\in B_{R+1}}v(t,y).
		\]
		
		Define
		\begin{equation}
			\label{eq:appendix-countable-convex-extension}
			\overline v(t,x)
			:=
			\sup
			\left\{
			p\cdot x+q:
			(p,q)\in\mathcal P\times\mathbb Q,\;
			t\in I_{p,q}
			\right\},
			\qquad
			(t,x)\in[0,T]\times\mathbb R^d.
		\end{equation}
		Because the supremum is taken over a countable family of jointly
		Borel functions, $\overline v$ is jointly Borel. For each fixed
		$t$, it is a supremum of affine functions whose slopes have norm at
		most $L$. Hence $\overline v(t,\cdot)$ is convex and globally
		$L$-Lipschitz. Moreover, $\overline v$ is non-decreasing in $t$:
		if $s\leq t$, then every affine minorant admitted at time $s$ is also
		admitted at time $t$.
		
		We claim that
		\begin{equation}
			\label{eq:appendix-extension-agrees-locally}
			\overline v(t,x)=v(t,x),
			\qquad
			(t,x)\in[0,T]\times B_{R+1}.
		\end{equation}
		Indeed, every affine function admitted in
		\eqref{eq:appendix-countable-convex-extension} lies below
		$v(t,\cdot)$ on $B_{R+1}$. The inequality first holds on the dense
		set $\mathcal Y$ and then on all of $B_{R+1}$ by continuity.
		Therefore,
		\[
		\overline v(t,x)\leq v(t,x),
		\qquad x\in B_{R+1}.
		\]
		
		Conversely, fix $t\in[0,T]$ and $x\in B_{R+1}$. Choose
		\[
		p_x\in\partial v(t,x).
		\]
		By \eqref{eq:appendix-local-subgradient-bound},
		$|p_x|\leq L$. Choose a sequence
		\[
		p_k\in\mathcal P,
		\qquad
		p_k\longrightarrow p_x,
		\]
		and define
		\[
		c_k
		:=
		\inf_{y\in B_{R+1}}
		\bigl(v(t,y)-p_k\cdot y\bigr).
		\]
		Choose $q_k\in\mathbb Q$ such that
		\[
		c_k-\frac1k<q_k<c_k.
		\]
		Then $p_k\cdot y+q_k\leq v(t,y)$ on $B_{R+1}$, so
		$(p_k,q_k)$ is admissible. Since $B_{R+1}$ is bounded and
		$p_k\to p_x$,
		\[
		c_k+p_k\cdot x
		\longrightarrow
		v(t,x).
		\]
		It follows that
		\[
		\overline v(t,x)\geq v(t,x),
		\]
		which proves \eqref{eq:appendix-extension-agrees-locally}.
		
		Combining \eqref{eq:appendix-local-sup-bound},
		\eqref{eq:appendix-extension-agrees-locally}, and the global
		$L$-Lipschitz property gives
		\begin{equation}
			\label{eq:appendix-extension-growth}
			|\overline v(t,x)|
			\leq
			C_{d,T,R}F(1+|x|),
			\qquad
			(t,x)\in[0,T]\times\mathbb R^d.
		\end{equation}
		
		\medskip
		\noindent
		\textbf{Step 3: }
		Extend $\overline v$ to all $t\in\mathbb R$ by setting
		\begin{equation}
			\label{eq:appendix-time-extension}
			\widetilde v(t,x)
			:=
			\begin{cases}
				\overline v(0,x),&t<0,\\
				\overline v(t,x),&0\leq t\leq T,\\
				\overline v(T,x),&t>T.
			\end{cases}
		\end{equation}
		The endpoint values were chosen as the corresponding one-sided
		monotone limits. Since $\overline v=v$ on
		$[0,T]\times B_{R+1}$, the constant extension introduces no jump at
		$t=0$ or $t=T$ and hence no Dirac mass at either temporal endpoint.
		Any singular part of the distributional time derivative inside
		$(0,T)$ is non-negative, by monotonicity, and is already included in
		\eqref{eq:appendix-local-distributional-inequality}.
		
		For $t\notin[0,T]$, the function $\widetilde v$ is constant in
		time. Moreover, every spatial section is convex, and hence
		\[
		D^2\widetilde v\geq0
		\]
		in the sense of symmetric matrix-valued distributions. Therefore, for
		every constant symmetric non-negative definite matrix $A$,
		\begin{equation}
			\label{eq:appendix-extended-distributional-inequality}
			\partial_t\widetilde v
			+
			\operatorname{tr}(A D^2\widetilde v)
			\geq
			\kappa_d(\det A)^{1/(d+1)}\widetilde f
		\end{equation}
		in the sense of distributions on
		$\mathbb R\times B_{R+1}$.
		
		\medskip
		\noindent
		\textbf{Step 4: }
		Choose a non-negative product mollifier
		\[
		\rho(t,x)=\rho_0(t)\rho_1(x)
		\in C_c^\infty((-1,1)\times B_1),
		\qquad
		\int_{\mathbb R^{d+1}}\rho(t,x)\,dt\,dx=1,
		\]
		and define
		\[
		\rho_\delta(t,x)
		:=
		\delta^{-(d+1)}
		\rho(t/\delta,x/\delta).
		\]
		Let $0<\delta_n<1/2$, with $\delta_n\downarrow0$, and set
		\begin{equation}
			\label{eq:appendix-convolution-definition}
			u^n:=\widetilde v*\rho_{\delta_n},
			\qquad
			f^n:=\widetilde f*\rho_{\delta_n}.
		\end{equation}
		
		The convolution defining $u^n$ is finite because
		$\widetilde v$ has at most linear growth. Standard properties of
		convolution of distributions
		\cite[Chapter~4]{Hormander-2003} show that $u^n$ and $f^n$ are
		smooth and that differentiation commutes with convolution. Since
		$\rho_{\delta_n}\geq0$, we also have $f^n\geq0$.
		
		For every fixed $t$, $u^n(t,\cdot)$ is convex, because it is a
		non-negative average of translates of convex functions. The global
		$L$-Lipschitz estimate and
		\eqref{eq:appendix-extension-growth} are preserved by convolution, up
		to enlarging $C_{d,T,R}$. Thus
		\eqref{eq:appendix-global-growth-bound} and
		\eqref{eq:appendix-global-gradient-bound} hold.
		
		If $(t,x)\in[0,T]\times B_R$, then
		$\delta_n<1/2$ ensures that the spatial convolution only uses points
		in $B_{R+1}$. Convolving
		\eqref{eq:appendix-extended-distributional-inequality} with the
		non-negative kernel $\rho_{\delta_n}$ therefore yields
		\[
		\partial_tu^n(t,x)
		+
		\operatorname{tr}\!\left(A D^2u^n(t,x)\right)
		\geq
		\kappa_d(\det A)^{1/(d+1)}f^n(t,x)
		\]
		for every symmetric non-negative definite matrix $A$. This proves
		\eqref{eq:appendix-smoothed-aleksandrov-inequality}.
		
		Finally, since
		\[
		f\in C_c^\infty((0,T)\times B_R),
		\]
		its zero extension satisfies
		\[
		\widetilde f\in
		C_c^\infty(\mathbb R\times\mathbb R^d).
		\]
		The standard approximation properties of mollifiers therefore give
		\[
		f^n\longrightarrow\widetilde f
		\]
		uniformly on $\mathbb R\times\mathbb R^d$. The proof is complete.
	\end{proof}
	
	\begin{remark}
		\label{rem:constant-time-extension}
		The constant time extension in
		\eqref{eq:appendix-time-extension} permits regularization up to
		$t=0$ and $t=T$. Choosing the endpoint values as one-sided monotone
		traces prevents boundary Dirac masses in
		$\partial_t\widetilde v$, while spatial convexity makes the
		second-order contribution non-negative outside $[0,T]$.
	\end{remark}

\end{document}